\documentclass[12pt]{amsart}
\usepackage[utf8]{inputenc}
\usepackage[style=alphabetic,url=false, isbn=false, doi=false,maxbibnames=99]{biblatex}
\AtEveryBibitem{
    \clearfield{urlyear}
    \clearfield{urlmonth}
}

\usepackage{xspace,amssymb,epsfig,calligra, enumerate, enumitem,euscript,enumerate,amsmath,nicefrac,stmaryrd}
\DeclareMathAlphabet{\mathbbm}{U}{bbm}{m}{n}
\usepackage{tikz,etex,xspace,afterpage}
\usepackage{tikz-cd}
 \usepackage{epstopdf,ifpdf,todonotes}
\usepackage{accents}
\usepackage{stackengine}
\tikzset{
	MyPersp/.style={scale=1.8,x={(-0.8cm,0cm)},y={(0cm,.3cm)},
    z={(0cm,1cm)}},
MyPoints/.style={fill=white,draw=black,thick}
		}

\stackMath
\usepackage{nameref,zref-xr} \zxrsetup{toltxlabel} \zexternaldocument*[I-]{coherent-coulomb}[https://uwaterloo.ca/scholar/sites/ca.scholar/files/b2webste/files/coherent-coulomb.pdf] \usepackage[margin=3cm,footskip=25pt,headheight=20pt]{geometry}

\usepackage{hyperref}

\usepackage{tabularx}
\usepackage{anindex}
\anindexsetup{
    heading,
    lines,
}

\usepackage{fancyhdr,mathrsfs}

  \usepackage{pdfsync}

\begin{document}
\newtheoremstyle{all}{11pt}{11pt}{\slshape}{}{\bfseries}{}{.5em}{}

\theoremstyle{all}
\newtheorem{itheorem}{Theorem}
\newtheorem{theorem}{Theorem}[section]
\newtheorem{proposition}[theorem]{Proposition}
\newtheorem{corollary}[theorem]{Corollary}
\newtheorem{lemma}[theorem]{Lemma}
\newtheorem{assumption}[theorem]{Assumption}
\newtheorem{definition}[theorem]{Definition}
\newtheorem{ques}[theorem]{Question}
\newtheorem{conjecture}[theorem]{Conjecture}

\theoremstyle{remark}
\newtheorem{remark}[theorem]{Remark}
\newtheorem{examplex}{Example}
\newenvironment{example}
  {\pushQED{\qed}\renewcommand{\qedsymbol}{$\clubsuit$}\examplex}
  {\popQED\endexamplex }
\renewcommand{\theexamplex}{{\arabic{section}.\roman{examplex}}}
\newcommand{\nc}{\newcommand}
\newcommand{\renc}{\renewcommand}
\numberwithin{equation}{section}
\renc{\theequation}{\arabic{section}.\arabic{equation}}

\newcounter{subeqn}
\renewcommand{\thesubeqn}{\theequation\alph{subeqn}}
\newcommand{\subeqn}{\refstepcounter{subeqn}\tag{\thesubeqn}}\makeatletter
\@addtoreset{subeqn}{equation}
\newcommand{\newseq}{\refstepcounter{equation}}
\nc{\corre}{\mathfrak{A}}
  \nc{\kac}{\kappa^C}
  \nc{\uf}{\mathring{U}}
  \nc{\ubf}{{}_U\mathring{B}_R}
  \nc{\fbu}{{}_R\mathring{B}_U}
\nc{\alg}{T}
\nc{\salg}{W}
\nc{\zero}{o}
\nc{\weights}{\iota}
\nc{\psalg}{\mathscr{W}}
\nc{\MaxSpec}{\operatorname{MaxSpec}}
\nc{\inter}{d}

\nc{\Isalg}{\mathscr{H}}
\nc{\Lco}{L_{\la}}
\nc{\qD}{q^{\nicefrac 1D}}
\nc{\ocL}{M_{\la}}
\nc{\excise}[1]{}
\nc{\Dbe}{D^{\uparrow}}
\nc{\Dfg}{D^{\mathsf{fg}}}
\nc{\frob}{\mathsf{f}}

\nc{\op}{\operatorname{op}}
\nc{\Sym}{\operatorname{Sym}}
\nc{\Symt}{S}
\nc{\tr}{\operatorname{tr}}
\newcommand{\Mirkovic}{Mirkovi\'c\xspace}
\nc{\tla}{\mathsf{t}_\la}
\nc{\llrr}{\langle\la,\rho\rangle}
\nc{\lllr}{\langle\la,\la\rangle}
\nc{\K}{\mathbbm{k}}
\nc{\Stosic}{Sto{\v{s}}i{\'c}\xspace}
\nc{\cd}{\mathcal{D}}
\nc{\cT}{\mathcal{T}}
\nc{\vd}{\mathbb{D}}
\nc{\Fp}{{\mathbb{F}_p}}
\nc{\lift}{\gamma}
\nc{\Proj}{\operatorname{Proj}}
\nc{\Alc}{\nabla}
\nc{\ProjC}{\mathbf{R}}
\nc{\cox}{h}
\nc{\Aut}{\operatorname{Aut}}
\nc{\R}{\mathbb{R}}
\renc{\wr}{\operatorname{wr}}
  \nc{\Lam}[3]{\La^{#1}_{#2,#3}}
  \nc{\Lab}[2]{\La^{#1}_{#2}}
  \nc{\Lamvwy}{\Lam\Bv\Bw\By}
  \nc{\Labwv}{\Lab\Bw\Bv}
  \nc{\nak}[3]{\mathcal{N}(#1,#2,#3)}
  \nc{\hw}{highest weight\xspace}
  \nc{\al}{\alpha}
  \nc{\gK}{K}
  \nc{\gk}{\mathfrak{k}}
  
\newcommand{\LLoc}{\mathbb{L}\!\operatorname{Loc}}
\newcommand{\Rsecs}{\mathbb{R}\Gamma_\bS}

\newlength{\dhatheight}
\newcommand{\doublehat}[1]{\settoheight{\dhatheight}{\ensuremath{\hat{#1}}}\addtolength{\dhatheight}{-0.35ex}\hat{\vphantom{\rule{1pt}{\dhatheight}}\smash{\hat{#1}}}}

\newcommand{\dgmod}{\operatorname{-dg-mod}}
  \nc{\be}{\beta}
  \nc{\bM}{\mathbf{m}}
  \nc{\Bu}{\mathbf{u}}

  \nc{\bkh}{\backslash}
  \nc{\Bi}{\mathbf{i}}
  \nc{\Bm}{\mathbf{m}}
  \nc{\Bj}{\mathbf{j}}
 \nc{\Bk}{\mathbf{k}}
  \nc{\Bs}{\mathbf{s}}
\newcommand{\bS}{\mathbb{S}}
\newcommand{\bT}{\mathbb{T}}
\newcommand{\bt}{\mathbbm{t}}

\nc{\hatD}{\widehat{\Delta}}
\nc{\bd}{\mathbf{d}}
\nc{\D}{\mathcal{D}}
\nc{\mmod}{\operatorname{-mod}}  
\nc{\fdmod}{\operatorname{-fdmod}}  
\nc{\AS}{\operatorname{AS}}
\newcommand{\red}{\mathfrak{r}}

\nc{\RAA}{R^\A_A}
  \nc{\Bv}{\mathbf{v}}
  \nc{\Bw}{\mathbf{w}}
\nc{\Id}{\operatorname{Id}}
\nc{\Cth}{S_h}
\nc{\Cft}{S_1}
\def\MHM{{\operatorname{MHM}}}

\newcommand{\cM}{\mathcal{M}}
\newcommand{\cD}{\mathcal{D}}
\newcommand{\LCP}{\operatorname{LCP}}
  \nc{\By}{\mathbf{y}}
\nc{\eE}{\EuScript{E}}
  \nc{\Bz}{\mathbf{z}}
  \nc{\coker}{\mathrm{coker}\,}
  \nc{\C}{\mathbb{C}}
\nc{\ab}{{\operatorname{ab}}}
\nc{\wall}{\mathbbm{w}}
  \nc{\ch}{\mathrm{ch}}
  \nc{\de}{\delta}
  \nc{\ep}{\epsilon}
  \nc{\Rep}[2]{\mathsf{Rep}_{#1}^{#2}}
  \nc{\Ev}[2]{E_{#1}^{#2}}
  \nc{\fr}[1]{\mathfrak{#1}}
  \nc{\fp}{\fr p}
  \nc{\fq}{\fr q}
  \nc{\fl}{\fr l}
  \nc{\fgl}{\fr{gl}}
  \nc{\Fr}{\operatorname{Fr}}
\nc{\rad}{\operatorname{rad}}
\nc{\ind}{\operatorname{ind}}
  \nc{\GL}{\mathrm{GL}}
\newcommand{\arxiv}[1]{\href{http://arxiv.org/abs/#1}{\tt arXiv:\nolinkurl{#1}}}
  \nc{\Hom}{\mathrm{Hom}}
  \nc{\im}{\mathrm{im}\,}
  \nc{\La}{\Lambda}
  \nc{\la}{\lambda}
  \nc{\mult}{b^{\mu}_{\la_0}\!}
  \nc{\mc}[1]{\mathcal{#1}}
  \nc{\om}{\omega}
\nc{\gl}{\mathfrak{gl}}
  \nc{\cF}{\mathcal{F}}
\nc{\cC}{\mathcal{C}}
  \nc{\Mor}{\mathsf{Mor}}
  \nc{\HOM}{\operatorname{HOM}}

  \nc{\sHom}{\mathscr{H}\text{\kern -3pt {\calligra\large om}}\,}
  \nc{\Ob}{\mathsf{Ob}}
  \nc{\Vect}{\operatorname{-Vect}}
\nc{\gVect}{\mathsf{gVect}}
  \nc{\modu}{\mathsf{-mod}}
\nc{\pmodu}{\mathsf{-pmod}}
  \nc{\qvw}[1]{\La(#1 \Bv,\Bw)}
  \nc{\van}[1]{\nu_{#1}}
  \nc{\Rperp}{R^\vee(X_0)^{\perp}}
  \nc{\si}{\sigma}
\nc{\sgns}{{\boldsymbol{\sigma}}}
  \nc{\croot}[1]{\al^\vee_{#1}}
\nc{\di}{\mathbf{d}}
  \nc{\SL}[1]{\mathrm{SL}_{#1}}
  
  \nc{\slhat}[1]{\mathfrak{\widehat{sl}}_{#1}}
  \nc{\sllhat}{\slhat{\ell}}
  \nc{\slnhat}{\slhat{n}}
    \nc{\slehat}{\slhat{e}}
    \nc{\rif}{B}
   
 \nc{\sle}{\mathfrak{sl}_e}
    \nc{\Th}{\theta}
  \nc{\vp}{\varphi}
  \nc{\wt}{\mathrm{wt}}
\nc{\te}{\tilde{e}}
\nc{\tf}{\tilde{f}}
\nc{\hwo}{\mathbb{V}}
\nc{\soc}{\operatorname{soc}}
\nc{\cosoc}{\operatorname{cosoc}}
 \nc{\Q}{\mathbb{Q}}
\nc{\LPC}{\mathsf{LPC}}
  \nc{\Z}{\mathbb{Z}}
  \nc{\Znn}{\Z_{\geq 0}}
  \nc{\ver}{\EuScript{V}}
  \nc{\Res}[2]{\operatorname{Res}^{#1}_{#2}}
  \nc{\edge}{\EuScript{E}}
  \nc{\Spec}{\operatorname{Spec}}
  \nc{\tie}{\EuScript{T}}
  \nc{\ml}[1]{\mathbb{D}^{#1}}
  \nc{\fQ}{\mathfrak{Q}}
        \nc{\fg}{\mathfrak{g}}
        \nc{\ft}{\mathfrak{t}}
        \nc{\fm}{\mathfrak{m}}
  \nc{\Uq}{U_q(\fg)}
        \nc{\bom}{\boldsymbol{\omega}}
\nc{\bla}{{\underline{\boldsymbol{\la}}}}
\nc{\bmu}{{\underline{\boldsymbol{\mu}}}}
\nc{\bal}{{\boldsymbol{\al}}}
\nc{\bet}{{\boldsymbol{\eta}}}
\nc{\rola}{X}
\nc{\wela}{Y}
\nc{\fM}{\mathfrak{M}}
\nc{\tfM}{\mathfrak{\tilde M}}
\nc{\fX}{\mathfrak{X}}
\nc{\fH}{\mathfrak{H}}
\nc{\fE}{\mathfrak{E}}
\nc{\fF}{\mathfrak{F}}
\nc{\fI}{\mathfrak{I}}
\nc{\qui}[2]{\fM_{#1}^{#2}}
\nc{\cL}{\mathcal{L}}
\nc{\ca}[2]{\fQ_{#1}^{#2}}
\nc{\wtG}{\widetilde{G((t))}}

\nc{\cat}{\mathcal{V}}
\nc{\cata}{\mathfrak{V}}
\nc{\catf}{\mathscr{V}}
\nc{\hl}{\mathcal{X}}
\nc{\hld}{\EuScript{X}}
\nc{\hldbK}{\EuScript{X}^{\bla}_{\bar{\mathbb{K}}}}
\nc{\Iwahori}{\mathrm{Iwa}}
\nc{\hE}{\mathfrak{E}^{(1)}}
\nc{\Eh}{\mathfrak{E}^{(2)}}
\nc{\hF}{\mathfrak{F}^{(1)}}
\nc{\Fh}{\mathfrak{F}^{(2)}}
\nc{\<}{\langle}
\renc{\>}{\rangle}

\nc{\pil}{{\boldsymbol{\pi}}^L}
\nc{\pir}{{\boldsymbol{\pi}}^R}
\nc{\cO}{\mathcal{O}}
\nc{\Ko}{\text{\Denarius}}
\nc{\Ei}{\fE_i}
\nc{\Fi}{\fF_i}
\nc{\fil}{\mathcal{H}}
\nc{\brr}[2]{\beta^R_{#1,#2}}
\nc{\brl}[2]{\beta^L_{#1,#2}}
\nc{\so}[2]{\EuScript{Q}^{#1}_{#2}}
\nc{\EW}{\mathbf{W}}
\nc{\rma}[2]{\mathbf{R}_{#1,#2}}
\nc{\Dif}{\EuScript{D}}\nc{\MDif}{\EuScript{E}}
\renc{\mod}{\mathsf{mod}}
\nc{\modg}{\mathsf{mod}^g}
\nc{\fmod}{\mathsf{mod}^{fd}}
\nc{\id}{\operatorname{id}}
\nc{\compat}{\EuScript{K}}
\nc{\DR}{\mathbf{DR}}
\nc{\End}{\operatorname{End}}
\nc{\Fun}{\operatorname{Fun}}
\nc{\Ext}{\operatorname{Ext}}
\nc{\Coh}{\operatorname{Coh}}
\nc{\tw}{\tau}
\nc{\second}{\tau}
\nc{\A}{\EuScript{A}}
\nc{\Loc}{\mathsf{Loc}}
\nc{\eF}{\EuScript{F}}
\nc{\LAA}{\Loc^{\A}_{A}}
\nc{\perv}{\mathsf{Perv}}
\nc{\gfq}[2]{B_{#1}^{#2}}
\nc{\qgf}[1]{A_{#1}}
\nc{\qgr}{\qgf\rho}
\nc{\tqgf}{\tilde A}
\nc{\Tr}{\operatorname{Tr}}
\nc{\Tor}{\operatorname{Tor}}
\nc{\cQ}{\mathcal{Q}}
\nc{\st}[1]{\Delta(#1)}
\nc{\cst}[1]{\nabla(#1)}
\nc{\ei}{\mathbf{e}_i}
\nc{\Be}{\mathbf{e}}
\nc{\Hck}{\mathfrak{H}}
\renc{\P}{\mathbb{P}}
\nc{\bbB}{\mathbb{B}}
\nc{\ssy}{\mathsf{y}}
\nc{\cI}{\mathcal{I}}
\nc{\cG}{\mathcal{G}}
\nc{\cH}{\mathcal{H}}
\nc{\coe}{\mathfrak{K}}
\nc{\pr}{\operatorname{pr}}
\nc{\bra}{\mathfrak{B}}
\nc{\rcl}{\rho^\vee(\la)}
\nc{\tU}{\mathcal{U}}
\nc{\dU}{{\stackon[8pt]{\tU}{\cdot}}}
\nc{\dT}{{\stackon[8pt]{\cT}{\cdot}}}
\nc{\BFN}{\EuScript{R}}
\nc{\sfB}{\ensuremath{\mathsf{B}}}
 \newcommand{\barLambda}{\bar{\Lambda}}
 \nc{\groupK}{K}
\nc{\sfA}{\ensuremath{\mathsf{A}}}
\nc{\sfAhat}{\ensuremath{\widehat{\mathsf{A}}}}

\nc{\RHom}{\mathrm{RHom}}
\nc{\tcO}{\tilde{\cO}}
\nc{\Yon}{\mathscr{Y}}
\nc{\sI}{{\mathsf{I}}}
\nc{\sptc}{X_*(T)_1}
\nc{\spt}{\ft_1}
\nc{\Bpsi}{u}
\nc{\acham}{\eta}
\nc{\hyper}{\mathsf{H}}
\nc{\AF}{\EuScript{Fl}}
\nc{\VB}{\EuScript{X}}
\nc{\OHiggs}{\cO_{\operatorname{Higgs}}}
\nc{\OCoulomb}{\cO_{\operatorname{Coulomb}}}
\nc{\tOHiggs}{\tilde\cO_{\operatorname{Higgs}}}
\nc{\tOCoulomb}{\tilde\cO_{\operatorname{Coulomb}}}
\nc{\indx}{\mathcal{I}}
\nc{\redu}{K}
\nc{\Ba}{\mathbf{a}}
\nc{\Bb}{\mathbf{b}}
\nc{\Bc}{\mathbf{c}}
\nc{\Lotimes}{\overset{L}{\otimes}}
\nc{\AC}{C}
\nc{\rAC}{rC}\nc{\defr}{\operatorname{def}}
\nc{\bbeta}{\boldsymbol{\beta}}
\nc{\bgamma}{\boldsymbol{\gamma}}

\nc{\rACp}{\mathsf{C}}
\nc{\ideal}{\mathscr{I}}
\nc{\ACs}{\mathscr{C}}
\nc{\Stein}{\mathscr{X}}
\nc{\pStein}{p\mathscr{X}}
\nc{\pSteinK}{\overline{\mathscr{X}}}
\nc{\No}{H}
\nc{\To}{Q}
\nc{\tNo}{\tilde{H}}
\nc{\tTo}{\tilde{Q}}
\nc{\gaugeG}{G}
\nc{\weylW}{W}
\nc{\matterV}{V}
\nc{\quiver}{\Gamma}
\nc{\Coulomb}{\fM}
\nc{\Rring}{\mathring{R}}
\nc{\hRring}{\mathring{R}^h}
\nc{\MQ}{\mathfrak{M}^{\To}}

\nc{\scrB}{\mathscr{B}}\nc{\scrT}{\mathscr{T}}
\nc{\Asph}{\EuScript{A}^{\operatorname{sph}}}
\nc{\What}{\widehat{\weylW}}
\nc{\tM}{\tilde{\fM}}
\nc{\flav}{\phi}
\nc{\tF}{\tilde{F}}
\newcommand{\cOg}{\mathcal{O}_{\!\operatorname{g}}}
\newcommand{\tcOg}{\mathcal{\tilde O}_{\!\operatorname{g}}}
\newcommand{\dOg}{D_{\cOg}}
\newcommand{\preO}{p\cOg}
\newcommand{\dpreO}{D_{p\cOg}}
\nc{\vertex}{\EuScript{V}}
\nc{\Wei}{\EuScript{W}}
\nc{\longi}{\boldsymbol{\ell}}
\nc{\supp}{\operatorname{supp}}
\nc{\efA}{\EuScript{A}}
\newcommand{\Sg}{\mathsf{G}}
\newcommand{\Sr}{\mathsf{R}}
\newcommand{\Sc}{\mathsf{C}}
\newcommand{\Sgr}{\mathsf{GR}}
\newcommand{\Sall}{\mathsf{CGR}}
\newcommand{\stck}[2]{\genfrac{}{}{0pt}{0}{#1}{#2}}
\newcommand{\Ab}{\mathbb{A}}
\newcommand{\tslabar}{\mathbin{
\setbox0=\hbox{/\!\!/\!\!/}\rule[0.4\ht0]{\wd0}{.3\dp0}\kern-\wd0\box0}}
\setcounter{tocdepth}{2}
\newcommand{\thetitle}{Coherent sheaves and quantum Coulomb branches II:\\ quiver gauge theories and knot homology}
\newcommand{\theshorttitle}{Coherent sheaves and quantum Coulomb branches II}
\renc{\theitheorem}{\Alph{itheorem}}

\excise{
\newenvironment{block}
\newenvironment{frame}
\newenvironment{tikzpicture}
\newenvironment{equation*}
}
\newcounter{dummy}
\baselineskip=1.1\baselineskip

 \usetikzlibrary{decorations.pathreplacing,backgrounds,decorations.markings,shapes.geometric,decorations.pathmorphing}
\tikzset{wei/.style={draw=red,double=red!40!white,double distance=1.5pt,thin}}
\tikzset{awei/.style={draw=blue,double=blue!40!white,double distance=1.5pt,thin}}
\tikzset{bdot/.style={fill,circle,color=blue,inner sep=3pt,outer sep=0}}
    \tikzset{ weyl/.style={decorate, decoration={snake}, draw=black!50!green, very thick}}
\tikzset{fringe/.style={gray,postaction={decoration=border,decorate,draw,gray, segment length=4pt,thick}}}
\tikzset{dir/.style={postaction={decorate,decoration={markings, mark=at position .8 with {\arrow[scale=1.3]{>}}}}}}
\tikzset{rdir/.style={postaction={decorate,decoration={markings, mark=at position .8 with {\arrow[scale=1.3]{<}}}}}}
\tikzset{edir/.style={postaction={decorate,decoration={markings, mark=at position .2 with {\arrow[scale=1.3]{<}}}}}}

\newcommand{\rectcolor}{red!40!white}

\newcommand{\braidup}{to[out=up,in=down]}
\newcommand{\identify}[4]{\draw[red,dashed] (#1,#2) -- (#1,#4);
    \draw[red,dashed] (#3,#2) -- (#3,#4)
}

\newcommand{\coupon}[3][0.15]{\filldraw[draw=black,fill=white] (#2) circle (#1);
    \node at (#2) {$\scriptscriptstyle{#3}$}}

\tikzset{anchorbase/.style={>=To,baseline={([yshift=-0.5ex]current bounding box.center)}}}
\begin{center}
\noindent {\large  \bf \thetitle}
\medskip

\noindent {\sc Ben Webster}\footnote{Supported by the NSERC through a Discovery Grant. This research was supported in part by Perimeter Institute for Theoretical Physics. Research at Perimeter Institute is supported in part by the Government of Canada through the Department of Innovation, Science and Economic Development Canada and by the Province of Ontario through the Ministry of Colleges and Universities.}\\  
Department of Pure Mathematics, University of Waterloo \& \\
 Perimeter Institute for Theoretical Physics\\
Waterloo, ON\\
Email: {\tt ben.webster@uwaterloo.ca}
\end{center}
\bigskip
{\small
\begin{quote}
\noindent {\em Abstract.}
We continue our study of noncommutative resolutions of Coulomb branches in the case of quiver gauge theories.  These include the Slodowy slices in type A and symmetric powers in $\mathbb{C}^2$ as special cases. These resolutions are based on vortex line defects in quantum field theory, but have a precise mathematical description, which in the quiver case is a modification of the formalism of KLRW algebras.  While best understood in a context which depends on the geometry of the affine Grassmannian and representation theory in characteristic $p$, we give a description of the Coulomb branches and their commutative and non-commutative resolutions which can be understood purely in terms of algebra.  

This allows us to construct a purely algebraic version of the knot homology theory defined using string theory by Aganagi\'c, categorifying the Reshetikhin-Turaev invariants for minuscule representations of type ADE Lie algebras.  We show that this homological invariant agrees with the categorification of these invariants previously defined by the author, and thus with Khovanov-Rozansky homology in type A.     
\end{quote}
}
\setcounter{section}{4}
\setcounter{itheorem}{3}

{\it Author's note}: As this is a continuation of the first part of this paper
\cite{WebcohI}, we will use the notation and constructions from that paper
without additional reference or comment. You can spot links to part I as they will have blue outline instead of red.  Note that in some PDF viewers, these links will not open correctly, due to a {\tt \#} getting converted to {\tt \%23}.  Using Adobe Acrobat seems to solve this issue, as does manually changing the {\tt \%23} in the URL back to {\tt \#}.

\section{(Re)introduction}
\label{sec:reintroduction}

In \cite{WebcohI,WebGT, websterKoszulDuality2019}, we developed a general theory of Coulomb
branches from an algebraic perspective.  We showed that the Coulomb
branch algebra itself, its extended category (of line operators) and
various related algebras, such as the noncommutative resolution
constructed through quantization by Bezrukavnikov and Kaledin
\cite{BKpos,KalDEQ, BezNon} and the category controlling their
Gelfand-Tsetlin modules, all have explicit combinatorial descriptions.  In this sequel to
these papers, we focus on understanding this construction in the
quiver gauge case, especially the noncommutative resolution and
corresponding geometric constructions with coherent sheaves.
Applications to the representation theory of quantum Coulomb branches
in characteristic 0 have already been covered extensively in
\cite{KTWWYO, WebGT, Webalt}, so we will only discuss these in
passing.

At the root of this perspective is a description of the Coulomb branch
algebra as linear combinations of paths in the space $ {T_{\R}/W}$, the quotient of the compact
torus of the gauge group $\gaugeG$ modulo the Weyl group $\weylW$, modulo certain
relations (see \cite[(2.5a--c)]{websterKoszulDuality2019}).  In the case
of $GL_n$, this space can be identified with the configuration of $n$
points on the circle $\R/\Z$ (allowing collisions), and thus a path in this space can be identified with a
diagram drawn on the cylinder.

\notation{\ensuremath{\quiver}}{A fixed quiver, often but not always a Dynkin diagram of type ADE.}
\notation{\ensuremath{\vertex}}{The vertex set of the quiver $\quiver$.}
\notation{\ensuremath{\edge}}{The edge set of the quiver $\quiver$.}
\notation{\ensuremath{\Bv}}{The dimension vector ${\vertex}\to
\Z_{\geq 0}$ which controls the rank of the factors of the gauge group.  In the cylindrical KLR algebra, this corresponds to the number of black strands with a given label.}
\notation{\ensuremath{\Bw}}{The dimension vector ${\vertex}\to
\Z_{\geq 0}$ which controls the dimension of the framing, or the number of edges connecting the new node in the Crawley-Boevey graph $\Gamma^{\Bw}$.  In the cylindrical KLR algebra, this corresponds to the number of red strands with a given label.}

Fix a quiver $\quiver$
with vertex set ${\vertex}(\Gamma)$, (oriented) edge set $\edge (\Gamma) $, and also fix dimension vectors $\Bv,\Bw\colon {\vertex}\to
\Z_{\geq 0}$ for this quiver.  We should emphasize that we do allow edge loops, though our focus will be on Dynkin diagrams.   By a {\bf quiver gauge theory} we mean
the $3d, \mathcal{N}=4$ gauge theory attached to the gauge group and matter $(G,V)$ given by: 
\begin{equation}
\gaugeG=\prod GL(\C^{v_i})\qquad \matterV=\Big(\bigoplus_{i\to j}
\Hom(\C^{v_i},\C^{v_j})\Big)\bigoplus \Big(\bigoplus_{i\in {\vertex}}
  \Hom(\C^{w_i},\C^{v_i}) \Big).\label{eq:quiver-gauge2}
\end{equation}
Similar to the approach in \cite{NaCoulomb} (where it is denoted $T^*V \tslabar G$), we do not define this theory as a precise mathematical object, but we will discuss mathematical constructions inspired by it.   
This group and representation may be familiar to readers since the hyperhamiltonian quotient of $T^*V$ by the compact form of $G$ is a Nakajima quiver variety \cite{nakajimaQuiverVarieties1998}.  In physics terminology, this quiver variety is the Higgs branch of this theory.
  
As described above, we can think of a path in $T_{\R}/W$ as a path in
a labeled configuration space where $v_i$ points have label $i$, that
is, as a string diagram on the cylinder where strands are labeled by
points in the Dynkin diagram.  When we translate the relations
\cite[(2.5a--c)]{websterKoszulDuality2019} into this framework, they suddenly become very
familiar to many practitioners of categorification---they are the local relations of the KLRW (weighted KLR) algebra, as presented in \cite{WebwKLR}.  The author and his
collaborators exploited this in \cite{KTWWYO} to study the
representation theory of shifted Yangians, but here we apply the same idea with a more geometric perspective.  Based on these ideas, we attach to the choice of $\Gamma,\Bv,\Bw$ and some additional auxiliary data a {\bf cylindrical flavored KLRW algebra}; these are defined by diagrams of red and black strands drawn on a cylinder subject to local relations like those of the {\bf weighted KLR algebra} \cite{WebwKLR}.

Recall that a {\bf noncommutative crepant resolution of singularities} for a commutative algebra $A$ (or the variety $\Spec A$) is an associative algebra $R$ such that $R\mmod$ behaves like the category of coherent sheaves on a crepant resolution of  $\Spec A$. For our purposes, we only need to consider the case where there is an idempotent $e$ such that $A=eRe$: in this case, $R$ is an NCCR if $eR$ is Cohen-Macaulay as an $A$-module, and $R$ has global dimension equal to the Krull dimension of $A$. Since a (usual) crepant resolution of a symplectic singularity is symplectic, we may as well call an NCCR of $A$ a {\bf noncommutative symplectic resolution} in the case where $\Spec A$ has symplectic singularities.

\notation{\ensuremath{\Coulomb}}{The Coulomb branch of the gauge theory with gauge group $\gaugeG$ and matter representation $\matterV$ (Definition \ref{I-def:Coulomb-branch}).}

Our main result is that:
\begin{itheorem}\label{ith:NCCR}
  If $\Coulomb$ is the Coulomb branch of a quiver gauge theory
  that admits a BFN resolution, then the cylindrical KLRW algebra $\Rring$ with
  the same underlying combinatorial data defines a noncommutative
  symplectic resolution of singularities of  $\Coulomb$.
\end{itheorem}
Since we'll be interested in considering versions of these varieties in positive characteristic, let us note that this result is only proven for these varieties over a field of characteristic 0 or sufficiently large (to the worried physicist: in particular, this result holds over $\C$), though we believe it actually is true in arbitrary characteristic.

By a {\bf BFN resolution}, we mean a symplectic resolution of singularities arising from the constructions of Braverman-Finkelberg-Nakajima in \cite{BFNline}.  In physics terms, this means that the Coulomb branch becomes smooth at a generic choice of mass parameters (and trivial FI parameters). 

The quiver gauge theories which admit BFN resolutions include those with cyclic or linear quivers (affine type A), and those for type D and E quivers with $w_i$ only non-zero on nodes with minuscule fundamental representations, in both cases chosen so they correspond to dominant weight spaces (these theories are often called ``good'' in the physics literature, to contrast them with ``bad'' and ``ugly'' theories).  
The most familiar examples of these
are resolved Slodowy slices (or more generally, S3 varieties) of type $A$ and the Hilbert scheme of points on $\C^2$ (or more generally,
the resolved Kleinian singularity $\C^2/\Z/\ell\Z$).   We'll discuss these examples in more detail in Appendix \ref{sec:slodowy-slices-type}.

 These noncommutative resolutions arise as the endomorphisms of tilting generators on symplectic resolutions of the same varieties; in fact, this is how we prove that they give noncommutative resolutions.
 \notation{\ensuremath{\tM}}{A BFN resolution of the Coulomb branch $\Coulomb$ (Definition \ref{def:BFN-resolution}).}
\begin{itheorem}\label{thm:D-equivalence}
  If $\Coulomb$ is the Coulomb branch of a quiver gauge
  theory, and $\tM$ a BFN resolution, then:
  \begin{enumerate}
  \item the homogeneous coordinate ring of $\tM$ is an algebra of
    twisted cylindrical KLR diagrams, modulo local relations.
  \item For each NCSR $R$ of Theorem \ref{ith:NCCR}, the variety $ \tM$ admits a tilting generator $\mathcal{T}$,
    described as an explicit module over the homogeneous coordinate
    ring, such that $\End(\mathcal{T})=R$; in particular, $\End(\mathcal{T})$ is a cylindrical flavored KLRW algebra $\Rring$.
  \item The wall-crossing functors relating different tilting
    generators are given by tensor product with explicit bimodules,
    modeled on the braiding functors of \cite[\S 6]{Webmerged}; more
    generally, the Schober connected to these functors can be constructed
    using the representation theory of related algebras.
  \end{enumerate}
\end{itheorem}
As with Theorem \ref{ith:NCCR}, this is only currently proven in characteristic 0, though we believe it holds in all characteristics.  

Finally, we turn to applying these algebras in topology for the construction of homological knot invariants.  Recent work \cite{aganagicKnotCategorification2020} of Aganagi\'c defines such a knot invariant based on the coherent sheaves on $\tM$ (which she denotes $\mathcal{X}$; see \cite[\S 3.1.1]{aganagicKnotCategorification2020}) in the case where $\Gamma$ is of ADE type.  In this case, our quiver gauge theory must correspond to a tensor product of minuscule representations in order to possess a BFN resolution.  Aganagi\'c's construction depends on an action of affine tangles on the category $D^b(\Coh \tM)$.  As described, this action arises naturally from a family of central charge functions on the Grothendieck group of this category: the action of affine braids is given by wall-crossing functors associated to certain singular loci for this function, and the action of cups and caps from a filtration that arises on the category as we approach certain walls.  We can capture this precisely in a {\bf real variation of stability} in the sense of Anno, Bezrukavnikov, and Mirkovi\'c \cite{annoStabilityConditions2015}.  

By Theorem \ref{thm:D-equivalence}, we can translate Aganagi\'c's action into complexes of bimodules over a cylindrical KLRW algebra.  In particular, we give a combinatorial description of the central charge function, which can be thought of as ``integration over a noncommutative resolution.''  The affine braid group action is covered in Theorem \ref{thm:D-equivalence}, and we add cup and cap functors that extend this to  
an annular tangle action on the categories $D^b(\Coh \tM)$ for different theories.  This gives us an annular knot invariant defined using only KLRW algebras, which agrees with that of Aganagi\'c, since at each point, our construction matches the description in \cite{aganagicKnotCategorification2020} under the equivalences of Theorem \ref{thm:D-equivalence}. In particular, there is no need for the asterisk in \cite[Th. 5*]{aganagicKnotCategorification2020}. 

More precisely, let $\Gamma$ be of ADE type, and let $\Rring^{\Bj}$ be the corresponding cylindrical KLR algebra, summed over all possible numbers of black strands. Consider 
 an oriented affine ribbon tangle $T$ labeled with minuscule fundamental representations.  We can read off the labels at the bottom of this tangle (taking the dual representation if the strand is oriented downward) to get a sequence $\Bj$ of minuscule fundamental representations, and similarly read off $\Bj'$ from the top.  
\begin{itheorem} We have an induced functor $\Phi(T)\colon D^b( {\Rring^{\Bj}}\fdmod)\to D^b(\Rring^{\Bj'}\fdmod)$ which is compatible with composition of tangles up to isomorphism of functors.  This defines a link invariant that agrees with those of \cite[Th. 5*]{aganagicKnotCategorification2020} and \cite[\S 8]{Webmerged}, and in type A with Khovanov-Rozansky homology.
\end{itheorem}
While we can use the abstract machinery of categorified skew Howe duality to show that this invariant is equivalent to many knot invariants which have appeared in the literature, our construction is closest in spirit to that of Cautis and Kamnitzer \cite{CKII}, since in type A, the resolved Coulomb branch appearing in our construction is an open subset of their convolution variety in the affine Grassmannian by \cite[Th. 3.10]{BFNplus}.  We have not carefully verified this, but it seems virtually certain that our coherent sheaves defining functors for tangles match those of \cite{CKII}.

It's worth noting that this defines an invariant of annular links (we think of a usual link as an annular link by embedding $B^3$ into the annulus times an interval).  We expect that in type A, this agrees with the annular Khovanov-Rozansky homology of Queffelec and Rose \cite{QRannular};  we lay out some preliminary steps to check this fact, but verifying it carefully is beyond the scope of this paper.

\excise{
Aside from making explicit a construction that otherwise requires some
rather challenging techniques, it's generally believed that the
algebras $\End(\mathcal{T})$ are Koszul; we hope that this can be
proved geometrically, much as similar results have been proven for
usual KLRW algebras.  Other possible applications include
studying Bezrukavnikov and Okounkov's conjecture relating
wall-crossing functors to quantum cohomology and generalizing Anno and
Nandakumar's work on affine tangles in 2-block Springer fibers to more
general actions of webs.  }
 \subsection*{Acknowledgements}

Many thanks to Mina Aganagi\'c, Roman Bezrukavnikov, Alexander Braverman, Kevin Costello, Tudor
Dimofte, Michael Finkelberg, Justin Hilburn, Joel Kamnitzer, Gus Lonergan, Ivan Losev,
Alex Weekes and Philsang Yoo for useful discussions on
these topics and Youssef Mousaaid and Alistair Savage for not
complaining about the theft of their diagrams.

\section{Cylindrical KLRW algebras}
\label{sec:cylindr-klrw-algebr}

\subsection{The definition of cylindrical KLRW algebras}
\notation{$\No$}{$\No=N_{GL(V)}(G)$}
As in the introduction, let $\quiver$ be a quiver, and  $\Bv,\Bw\colon {\vertex}\to
\Z_{\geq 0}$ dimension vectors.  First, we note some basic facts about quiver gauge theories.  In this case, the group $\No=N_{GL(V)}(G)$ is generated by $\gaugeG$, the product
$GL(\C^{ {w_i}})$ acting by precomposition on $\C^{w_i}$, and by
$GL(\C^{\chi_{i,j}})$ where $\chi_{i,j}$ is the number of edges
$i\to j$, acting by taking linear combinations of the maps along these
edges, i.e. via the isomorphism
\[\bigoplus_{i\to j}\Hom(\C^{ {v_i}},\C^{ {v_j}})\cong \bigoplus_{(i,j)\in {\vertex}\times \vertex}
  \Hom(\C^{ {v_i}},\C^{ {v_j}})\otimes
  \C^{\chi_{i,j}}.\]

Thus, for any unitary element $g\in \No$, we can find a conjugate (i.e. change bases in each of the spaces $\C^{v_i},\C^{w_i},\C^{\chi_{i,j}}$) which is a product of diagonal matrices.  If we expand to considering the coset
$g\gaugeG$, some element of this coset will be conjugate to a product of diagonal matrices in $GL(\C^{ {w_i}})$ and $GL(\C^{\chi_{i,j}})$ to an element which acts diagonally in the usual bases on $\mathbb{C}^{w_i}$ and $\C^{\chi_{i,j}}$.  Thus, we can find classes $\beta_e\in \R/\Z$ for $e\in \edge$ and $ \beta_{i,k}\in \R/\Z$ for $i\in \vertex,
k=1,\dots, w_i$ such that $g$ acts on the
map attached to the edge $e$ by the scalar
$\exp(2\pi i \beta_e)$ and on $\C^{w_j}$ with eigenvalues $\exp(2\pi i
\beta_{j,k})$.  

\notation{\ensuremath{\Gamma^{\Bw}}}{The Crawley-Boevey quiver of the dimension vector $\Bw$.}
This is slightly cleaner if we use the ``Crawley-Boevey trick'' of adding a new vertex
$\infty$ and $w_i$ new edges from $i$ to $\infty$ to make a larger quiver $\Gamma^{\Bw}$.  We extend $ \Bv $ to this new vertex by setting $ v_\infty = 1 $.  Then, we can think of an element
of $V$ as a representation of $\Gamma^{\Bw}$, using $ \Hom(\C^{v_i},
\C^{w_i}) = \Hom(\C^{v_i}, \C^{v_\infty})^{\oplus w_i} $, so 
$ N = \bigoplus_{e \in \edge(\Gamma^\Bw)} \Hom(\C^{v_{t(e)}},\C^{v_{h(e)}}) $.  We can extend $\beta$ to the set of edges of $\Gamma^{\Bw}$ by making $\beta_{i,k}$ for $k=1,\dots, w_i$ its values on the $w_i$ new edges $i\to\infty$.

\subsubsection{Cyclic flavored sequences} 
\label{sec:cyl-flav-seq}

A frequent theme in this paper will be generalizations of notions from earlier papers on Coulomb branches and KLR algebras, with the real line $\R$ replaced by the circle $\R/\Z$.  

Recall that a {\bf cyclic order} on a set $Y$ is a ternary relation $C\subset Y^3$, which is cyclic, asymmetric, transitive and total.  That is, $C$ is closed under cyclic permutations, and for any $a\in Y$, the order $<_a$ on $Y$ defined by $b<c$ if $(a,b,c)\in C$ or if $b=a$ and $c\neq a$ is a total order.   Note that any cyclic order on a finite set $Y$ is induced by a map of $Y$ to $\R/\Z$, where the latter is given a cyclic order where   $(\bar{a},\bar{b},\bar{c})\in C$ if and only if these classes have representatives such that $a<b<c<a+1$.

\notation{\ensuremath{\Ab}}{The abelian group from which longitudes are taken.}
\notation{$\sim$}{An equivalence relations on $\Ab$, whose equivalence classes are cyclically ordered.}
Consider a triple of:
\begin{enumerate}
    \item an abelian group $(\Ab,+)$,
    \item  an equivalence relation $\sim$ on the set $\Ab$ compatible with addition with a fixed element: $a\sim b$ if and only if $a+x\sim b+x$, and 
    \item a total cyclic order on each equivalence class compatible with addition with a fixed element:  for all $x\in \Ab$, we have that $(a,b,c) $ is cyclically ordered if and only if $(a+x,b+x,c+x)$ is.
\end{enumerate}  We encode these cyclic orders as a subset $C_{\Ab}\subset \Ab^3$ which satisfies invariance under the usual action of $\Ab$ on $\Ab^3$.  
 
The circle $\Ab=\R/\Z$ is an example of such a group with $\sim$ equal to the full equivalence relation.    
This restricts to the same structure on $\mathbb{F}_p$, induced by the group homomorphism $a\mapsto a/p\colon \mathbb{F}_p\to \R/\Z$.  

 Another example illustrating why we need this equivalence relation is $\Ab=\K$, an arbitrary field of characteristic $p$.  We endow this with the equivalence relation that $a\sim b$ if $a-b\in \Fp$ (or equivalently, $a^p-a=b^p-b$), and the cyclic order on each of these equivalence classes induced by the action of $\Fp$.   Perhaps a stranger example that we can consider is $\Ab=\R$ or $\Z$, with $C_{\Ab}$ the cyclic closure of the usual total order $<$, that is, $(a,b,c)\in C _{\Ab} $ if $a<b<c$ or $b<c<a$ or $c<a<b$.  
 \begin{remark}
If we worked strictly analogously to \cite{kamnitzerLieAlgebra2024},  we would also consider a separate set $\longi$ that carries an $\Ab$-action, equipped with its own cyclic order $C_{\longi}$ and equivalence relation $\sim$ compatible with those on $\Ab$.  That is, if $a,b,c\in \Ab, x,y,z\in \longi$, then
\begin{enumerate}
    \item $a\sim b, x\sim y$ implies that $a+x\sim b+y$.
    \item $(a,b,c)\in C_{\Ab} $ implies $(a+x,b+x,c+x)\in C_{\longi}$.
    \item $(x,y,z)\in C_{\longi}$ implies $(a+x,a+y,a+z)\in C_{\longi}$.  
\end{enumerate}
In this paper, we'll only consider the case $\Ab=\longi$, but in future papers it will be useful to consider where these sets are not equal.  
 \end{remark}

\begin{definition}
	An {\bf $\Ab$-flavoring} is a choice of $\beta_e\in \Ab$ for each edge $e\in \edge(\Gamma^{\Bw})$. 
\end{definition}  

\notation{$\beta=( {\beta_e}, {\beta_{i,k}})$}{The flavor in $\Ab$ defining the relationship between the longitudes of corporeal, ghostly and red elements of a flavored sequence.}
Fix an ordered $n$-tuple $\Bi\in  \vertex^n$.
Consider the sets $\Sc=[1,n]$  and 
\begin{equation}
   \Sg=\big\{(k,e) \in \Sc\times \edge\mid   i_k=h(e)\big\} \qquad \Sr= \big\{(\star,e) \mid e \in \edge(\Gamma^\Bw), \infty=h(e) \big\}.\label{eq:Gplus}	
\end{equation}
\notation{$\Sc$,$\Sg$,$\Sr$}{The sets indexing the corporeal, ghostly and red elements of the a flavored sequence.}
As in \cite{kamnitzerLieAlgebra2024}, we refer to the element of these sets as corporeal, ghostly and red, respectively, and let $\Sall$ denote their disjoint union.  Similarly, we let pairs like $\Sgr$ denote the union of the two sets. 

\begin{definition}
	A {\bf cyclic flavored sequence} is a triple $(\Bi,
	\Ba,C)$:
	\begin{enumerate}
		\item An ordered $n$-tuple $\Bi\in  \vertex^n$.
		\item An ordered $n$-tuple $\Ba=(a_1,\dots,a_n)\in \Ab^n$.  We call these the {\bf longitudes} of the elements of $\Sc$.
		\item A (total) cyclic order $C$ on the set $\Sall$.
	\end{enumerate}
	We also endow the elements of $(k,e)\in \Sgr$ with longitudes $a_k+\beta_e$ and the elements of $\Sr$ with $t(e)=i$ with the longitudes $\beta_{i,k}$.  We require the following properties to be satisfied:
\renewcommand{\theenumi}{\roman{enumi}}
\begin{enumerate}
  \item The longitude function $\mathfrak{a}\colon \Sall\to \Ab$ is weakly monotone: if $(x,y,z)\notin C $ with $x, y, z$ distinct, then $(\mathfrak{a}(x),\mathfrak{a}(y),\mathfrak{a}(z))\notin C_{\Ab}$. 
\item If $g\in \Sgr, m\in \Sc, x\in \Sall$ and $\mathfrak{a}(g)=\mathfrak{a}(m)\neq \mathfrak{a}(x)$, then $(x,m,g)\notin C $.  
    \end{enumerate} 
\end{definition}

We visualize a cyclic flavored sequence by choosing a monotone map $\Sall\to \R/\Z$ to the circle $\R/\Z$, and labeling the corresponding points of the circle with the longitudes.  With this perspective, the condition (i) is that as we read the longitudes around the circle in a given equivalence class, they are in cyclic order, but with repeats possible; (ii) is the statement that whenever corporeal strands and ghostly/red strands have the same longitude, we break the tie in their order by pushing the corporeal strand in the positive direction.  
\begin{example}\label{ex:flavored-cyclic}
	If our graph $\Gamma$ is given by $i\overset{e}\leftarrow j \overset{f}\leftarrow k$, and we fix $\Bw=(1,0,1)$, then the Crawley-Boevey quiver is 
	$$
	\begin{tikzcd}
		i \arrow[swap]{rd}{r} & j \arrow[swap]{l}{e} & k\arrow{ld}{s}  \arrow[swap]{l}{f} \\
		& \infty &
	\end{tikzcd}
	$$ 
Consider the dimension vector $\Bv=(1,1,1),$ and fix $\Bi=(k,j,i)$.  In this case, we have  
	\[\Sc=\{1,2,3\}\quad \Sg=\{(2,f),(3,e)\}\qquad \Sr=\{(\star, r),(\star,s)\} \] 
	 If we consider $\Ab=\mathbb{F}_{19}$ and choose the flavoring $\beta_e=10,\beta_f=14, \beta_{i,1}=2,\beta_{k,1}=12$, then the longitudes of all strands will be fixed by the choice of longitudes on $\Sc$.  One possible function $\mathfrak{a}$ is defined by:
	\[ \{4,7,16\} \quad \{18=4+14,7=16+10\}\quad \{2,12\}\]
	
	The only possible cyclic order compatible with this longitude function is the cyclicalization of the total order:
	\[ (\star, r) < 1 <(3,e)< 2 <(\star,s)< 3< (2,f) \]
	Note here that $((3,e), 2,x)\in C$ for any other $x$ by the ``tie-break rule'' since these have equal longitudes.  
\end{example}

\excise{For example, if $\Gamma$ is the Kronecker quiver $$
	\begin{tikzcd}
		0  \arrow[bend right = 30,swap]{r}{f} & 1 \arrow[bend right = 30,swap]{l}{e}
	\end{tikzcd}
	$$
	and $n=3$, then we must have a sequence $(i_1,i_2,i_3)$, for example    $\Bi=(0,1,0)$.  }

If $\Ab=\R/\Z$ then a particularly natural way to do this is to map each element of $\Sall$ to its longitude.  This map is only monotone if all longitudes are distinct, but this can be fixed by a small deformation, mapping $x$ to $\mathfrak{a}(x)+\epsilon_x$ for some small $\epsilon_x\in \R$. 

\begin{definition} 
	We say that two cyclic flavored sequences $(\Bi,\Ba,C)$ and $(\Bi',\Ba',C')$ are {\bf equivalent} if there is a bijection $\sigma\colon \Sc\to \Sc$ such that for all $r\in \Sc$ we have that:
    \begin{enumerate}
    \item $i'_{\sigma(r)}=i_r$
    \item  for every $(x,e)\in \Sgr$ such that $i_r=t(e)$, we
    have that
    $r<(x,e)$ if and only if $\sigma(r)<'(\sigma(x),e)$ (by
    convention $\sigma(\star)=\star$).  
    \end{enumerate}
\end{definition}

\notation{$\mathbb{L}$}{The set $\prod_{i\in \vertex}\Ab^{v_i}$, regarded as the set of possible longitudes on corporeal strands in a flavored sequence.}
Let $\mathbb{L}=\prod_{i\in \vertex}\Ab^{v_i}$; given a point $(\alpha_{i,k})$ in this space, there is a unique flavored sequence up to equivalence where the coordinates in $\Ab^{v_i}$ are the longitudes on corporeal strands with label $i$.  
\begin{definition}\label{def:prefered}
	Given $(\alpha_{i,k})\in \mathbb{L}$, the associated preferred cyclic flavored sequence is defined by:
	\begin{enumerate}
	\item Fix a total order on the vertices $\vertex$, and a choice of total order $<_{\Ab}$ on $\Ab$ whose cyclicalization $\tilde{C}_{\Ab}$ is a refinement of  $C_{\Ab}$.   
	\item The values $\alpha_{i,k}$ are the longitudes on the corporeal strands of label $i$.  We construct the vector $\Ba$ by ordering these values in increasing order with respect to $<_{\Ab}$; this fixes $\Bi$ up to reordering of strands with the same $a_*$.  We order these groups so that if $a_{j}=a_{k}$, then $i_j\leq i_k$ in the order on vertices.
	\item Order the set $\{x\in \Sall \mid \mathfrak{a}(x)=a\}$ so that the elements of $\Sr$ are lowest, then the elements of $\Sg$ and then the elements of $\Sc$, with strands in these groups ordered by the order on $\vertex$ applied to $t(e)$ for $(r,e)\in \Sgr$, and by $i_x$ when $x\in \Sc$.  More formally, we say that  $x<y$ whenever one of the conditions below holds
\[ (x,y)\in \Sr\times \Sg\cup \Sr\times \Sc\cup \Sg\times \Sc\]
\[ (x,y)\in \Sc\times \Sc\text{ and }i_x<i_y\]\[(x=(r,e),y=(r',e'))\in \Sr\times \Sr\cup\Sg\times \Sg  
\text{ and }t(e)<t(e').\]
\item 	We now endow  $\Sall$ with the cyclic order that first considers the longitude $\mathfrak{a}(x)$ and then within strands of fixed longitude, uses the order above.  That is $C$ is the saturation under cyclic permutations of the $(x,y,z)$ such that either 
	\begin{enumerate} \item $(\mathfrak{a}(x),\mathfrak{a}(y),\mathfrak{a}(z))\in \tilde{C}_{\Ab}$ or 
	\item $\mathfrak{a}(x)=\mathfrak{a}(y)\neq \mathfrak{a}(z)$ and  $x<y$ in the order above or
	\item $\mathfrak{a}(x)=\mathfrak{a}(y)= \mathfrak{a}(z)$ and  $x<y<z$ in the order above.
	\end{enumerate}
\end{enumerate}
This associates a preferred cyclic flavored sequence to each element of $\mathbb{L}$.  If we write $(\Bi,\Ba)$ without reference to a choice of $C$, then we implicitly mean that we use this preferred cyclic order.
\end{definition}

\begin{example}
	For example, if we return to the setting of Example \ref{ex:flavored-cyclic}, then when we chose $\{4,7,16\}$ to be the longitudes of the corporeal strands, all longitudes were distinct and no tie-breaks were necessary.  However, if two corporeal strands have the same longitude, for example, the strands with labels $i,j,k$ have longitudes $7,4,7$ respectively, then either $\Bi=(j,k,i)$ or $(j,i,k)$ would give a flavored sequence, so we must choose one of them, depending on the order of $j$ and $k$ in our order on $\vertex$.  Similarly, if these longitudes are $16,11,7$, then the two elements of $\Sg$ will both have longitude 2, and either ordering of them is compatible with the flavored sequence conditions. 
\end{example}

The space $\mathbb{L}$ is then naturally divided into equivalence classes of flavored sequences; if $\Ab$ has finitely many equivalence classes under $\sim$, there will only be finitely many equivalence classes of flavored sequences.
In the case of $\Ab=\R/\Z$, we can think of $\mathbb{L}$ as a compact manifold diffeomorphic to $(\R/\Z)^n$; we'll endow this space with the usual
measure with volume 1 on $(\R/\Z)^n$ (that is, its Haar measure as a
Lie group).  In this case, the equivalence classes given by the chambers for the subtorus arrangement bounded by the subtori where $\alpha_{j,m}-\alpha_{i,k}=\beta_e$ when $e\colon i\to j$ and those with $\alpha_{i,m}=\beta_{i,k}$; compare with \cite[Prop. 2.12]{WebwKLR} for the linear case. 
  \begin{example}
 	If we have $\Gamma=1\to 2$ and $v_1=v_2=w_1=w_2$ and we choose $\beta_e=1/4$, $\beta_{1,1}=1/2,$ $\beta_{2,1}=1/3$, then the space $\mathbb{L}=\R^2/\Z^2$, and the division into chambers looks like:
 	\[\tikz[very thick,scale=4]{\draw[fringe] (0,0)--(0,1) --(1,1) -- (1,0)--cycle;\draw[red](.5,0)--(.5,1); \draw[red](0,.33) --(1,.33); \draw (0,.25) --(.75,1); \draw (.75,0) --(1,.25); }\]
 	with the $x$-coordinate giving the longitude of the point with label $1$ and the $y$-coordinate giving the longitude of the point with label $2$.
 \end{example}
 If we deform $\beta_e$ and $\beta_{i,k}$, usually the set of these chambers will not change, but it will when we hit a point where there is a redundancy between these equations.  These redundancies correspond to paths in $\Gamma$ as an unoriented graph.  Such a path is given by a list of edges $e_1,\dots, e_n$ together with signs $\varepsilon_i\in \{\pm 1\}$ that indicate whether we transverse $e_i$ in the positive or negative direction.  

\begin{lemma}\label{lem:circuits}
If the set of possible equivalence classes of cyclic flavored sequences is not locally constant near $(\beta_e,\beta_{i,k})$ then for some $i,j\in \vertex$ and some $k\in [1,w_i],\ell\in [1,w_j]$, we have an unoriented path \[i=i_0\overset{e_1}\longrightarrow i_1\overset{e_2}\longrightarrow i_2\overset{e_3}\longrightarrow\cdots \overset{e_n}\longrightarrow i_n=j\] such that in $\R/\Z$:
\begin{equation}\label{eq:circuits}
    \beta_{j,\ell}-\beta_{i,k}+\sum_{p=1}^n\varepsilon_i\beta_{e_i} =0
\end{equation}
\end{lemma}
Note that we can use the edges corresponding to $(i,k)$ and $(j,\ell)$ to close this path up in the Crawley-Boevey graph, and if we think of the weighting $\beta_*$ as a $\R/\Z$-valued 1-cochain, the LHS of \eqref{eq:circuits} is the integral of this 1-cochain over the path thought of as a 1-cycle.  

In the example above, this local constancy will fail if the three lines all go through a common point, which would mean that we have $\beta_{1,1}-\beta_{2,1}+\beta_e=0$.  This corresponds to the unique circuit in the Crawley-Boevey graph, which is a cycle on the 3 nodes $\infty,1,2$.

\subsubsection{Cylindrical KLRW diagrams} 

Throughout the rest of this paper, a {\bf strand} will mean a curve in $ \R/\Z\times [0,1]$ (or in some contexts in later sections, in $\R\times [0,1]$)  of the form $\{(\bar{\pi}(t),t)\mid t\in [0,1]\}$ for some path $\bar{\pi}\colon [0,1]\to \R/\Z$.  
  
  \begin{definition}
The {\bf number of signed intersections} of this curve with a second curve defined by $\bar{\sigma}\colon [0,1]\to \R/\Z$ is the number of solutions to $\Delta(t)=\bar{\pi}(t)	-\bar{\sigma}(t)=0$ where this quantity is increasing minus the number where it is decreasing (not counting points where the strands are tangent but never cross); put differently, if we choose a continuous lift $\tilde{\Delta}$ of $\Delta$ to $\R$, then this number is $\lfloor \tilde{\Delta}(1)\rfloor-\lfloor\tilde{\Delta}(0)\rfloor$.  
  \end{definition} 
 Note that this number is antisymmetric: It switches sign if we swap the roles of $\bar{\pi}$ and $\bar{\sigma}$. The most important special case is that $\bar{\sigma}(t)=\bar{\pi}-nt+c$ for $c\in \R/\Z$ and $n\in \Z$, then this number of intersections is $n$.  In particular, if $n=0$, the difference between these strands is constant and there are no intersection points.  
  
\begin{definition}\label{def:cylindrical-diagram}
 A {\bf cylindrical flavored KLRW} diagram is a collection of finitely many strands in $\R/\Z\times [0,1]$.   The strands are divided into three sets: 
 \begin{enumerate}
     \item The {\bf corporeal} (which are drawn as solid black lines).
     \item The {\bf red} (which are drawn as solid red lines).
  \item The {\bf ghostly} (which are drawn as dashed black lines). 
 \end{enumerate}
  These must satisfy the usual genericity property of avoiding tangencies and triple points between any set of strands, as well as strands that meet at $y=0$ or $y=1$.  
  
  Note that the order in which strands meet the circles $y=0$ and and $y=1$ induce two distinct cyclic orders on the set of strands, which we'll call the bottom and top orders.
  
    In addition, a cylindrical KLRW diagram carries the data of cyclic flavored sequences $(\Bi,\Ba,C)$ and $(\Bi',\Ba',C')$ corresponding to the lines $y=0$ and $y=1$, together with bijections of the sets $\Sall$ and $\Sall'$ to the set of strands, matching the cyclic order $C$ to the bottom order and $C'$ to the top order.  We require that:
       \begin{enumerate}
    \item These bijections are compatible with the division of $\Sall,\Sall'$ into $\Sc,\Sg,\Sr$.  In particular, the bijection $\Sc  \to \Sc'$ induces a permutation $\sigma\in \Sigma_n$ of corporeal strand which respects labels: $i_{m}=i'_{\sigma(m)}$.
    \item The bijection $\Sg\to \Sg'$ between ghostly elements in the flavored sequence is given by $(k,e)\mapsto (\sigma(k),e)$, that is, it is induced by the bijection on corporeals.
    \item Furthermore, if $(k,e)\in \Sg$, then the corresponding ghostly strand and the corporeal strand for $k$ have $0$ signed intersections.  Note that we will have this behavior if we draw our diagrams with a fixed difference in $\R/\Z$ between the corporeal strand and its ghost.   
   \item The longitudes satisfy $a_{m}\sim a'_{\sigma(m)}$, that is, the difference between the longitudes at the top and bottom are in the same equivalence class; if this holds for corporeal strands, then it automatically follows for ghostly/red strands.
  \end{enumerate} 
  We can put dots on corporeal strands at any point that avoids a crossing and the lines $y=0$ and $y=1$.   We identify any diagrams that differ by isotopies preserving these genericity conditions.  
    \end{definition}

In the interest of brevity, we will typically leave ``flavored KLRW'' out and call these just ''cylindrical diagrams'' or just ``diagrams.'' Let us now describe our conventions for drawing cylindrical diagrams.  
We'll draw these on the page in the rectangle $[0,1]\times
[0,1]$ with seams on the left and right side of the diagram where we
should glue to obtain the cylindrical diagram.  Strands corresponding to corporeals are drawn as solid black lines,  strands corresponding to ghostly elements are drawn as dashed black lines, and strands corresponding to red elements as solid red lines.  At the top and bottom of each diagram, we include two rows of information:  
\begin{enumerate}
	\item  In the first row, we write the corresponding vertex $i_k$ for $k\in \Sc$, the edge $e$ for $(k,e)\in \Sg$ and the tail vertex $t(e)$ for $(\star,e)\in \Sr$.  
	\item In the second row we write the longitude.  
\end{enumerate}
\begin{example}
If our graph $\Gamma$ is given by $i\overset{e}\leftarrow j \overset{f}\leftarrow k$, and we take $\Ab =\mathbb{F}_{19}$ with $\beta_e=10,\beta_f=14, \beta_{i,1}=2,\beta_{k,1}=12$, then an example of such a diagram is given by \begin{equation*} 
       \tikz[very thick,xscale=4,yscale=1.25]{
          \draw[fringe] (-1,-1)-- (-1,1);
          \draw[fringe] (1,1)-- (1,-1);
          \draw[wei] (-.8,-1)--node[below, at start ]{$\stck{i}{2}$} (-.8,1);
          \draw[wei] (.4 ,-1)--node[below, at start ]{$\stck{k}{12}$} (.4,1);
\draw[dashed](-.9,1) to[out=-90,in=30] node[above, at start]{$\stck{\vphantom{j}e}{1}$} (-1,.65);
\draw (-1,.2) to[out=30,in=-90] node[above, at end]{$\stck{i}{10}$} (.1,1);
           \draw[dashed] (-.4 ,-1) to[out=90,in=-150] node[below, at start ]{$\stck{\vphantom{j}e}{7}$}(1,.65);
           \draw (.6 ,-1) to[out=90,in=-150] node[below, at start ]{$
           \stck{i}{16}$}(1,.2);
           \draw (-1,-.2) to[out=-30,in=90]node[below, at end ]{$\stck{k}{4}$} (-.6,-1);
           \draw (-.2 ,1) to[out=-90,in=150] node[pos=.4,circle,fill=black,inner sep=2pt]{} node[above, at start ]{$\stck{k}{8}$}(1,-.2);
           \draw (-.2,-1) to[out=90,in=-90] node[below, at start ]{$\stck{j}{7}$} node[above, at end]{$\stck{j}{5}$} (-.5,1);
           \draw[dashed]  (.8,-1) to[out=90,in=-90] node[below, at start ]{$\stck{f}{18}$} node[above, at end]{$\stck{f}{16}$}   (.5,1);
           }
        \end{equation*}
        On the other hand, the diagram below is not allowed, since the corporeal strand with label $i$ and its ghost with label $e$ don't have 0 signed intersections.  \begin{equation*} 
       \tikz[very thick,xscale=4,yscale=1.25]{
          \draw[fringe] (-1,-1)-- (-1,1);
          \draw[fringe] (1,1)-- (1,-1);
          \draw[wei] (-.8,-1)--node[below, at start ]{$\stck{i}{2}$} (-.8,1);
          \draw[wei] (.4 ,-1)--node[below, at start ]{$\stck{k}{12}$} (.4,1);
\draw (-1,.2) to[out=30,in=-90] node[above, at end]{$\stck{i}{10}$} (.1,1);
           \draw[dashed] (-.4 ,-1) to[out=90,in=-90] node[below, at start ]{$\stck{\vphantom{j}e}{6}$} node[above, at end]{$\stck{\vphantom{j}e}{1}$} (-.9,1);
           \draw (.6 ,-1) to[out=90,in=-150] node[below, at start ]{$
           \stck{i}{15}$}(1,.2);
           \draw (-1,-.2) to[out=-30,in=90]node[below, at end ]{$\stck{k}{4}$} (-.6,-1);
           \draw (-.2 ,1) to[out=-90,in=150] node[pos=.4,circle,fill=black,inner sep=2pt]{} node[above, at start ]{$\stck{k}{8}$}(1,-.2);
           \draw (-.2,-1) to[out=90,in=-90] node[below, at start ]{$\stck{j}{7}$} node[above, at end]{$\stck{j}{5}$} (-.5,1);
           \draw[dashed]  (.8,-1) to[out=90,in=-90] node[below, at start ]{$\stck{f}{18}$} node[above, at end]{$\stck{f}{16}$}   (.5,1);
           }
        \end{equation*}
\end{example}

For a given choice of $(\Bi,\Ba)$ or $(\alpha_{i,k})\in \mathbb{L}$, as discussed in Definition \ref{def:prefered}, we have a corresponding cyclic flavored sequence. 
\begin{definition}\label{def:i-a}
We let $e(\Bi,\Ba)$ or $e(\boldsymbol{\alpha})$ denote the KLR diagram where each strand is straight vertical (i.e. $\bar{\pi}$ is constant) with this corresponding flavored sequence at both top and bottom.  
\end{definition}
\notation{$e(\Bi,\Ba)$, $e(\boldsymbol{\alpha})$}{The idempotent in $\Rring$ defined by the labels $\Bi$ on corporeal strands and longitudes $\Ba$.}

\subsubsection{Relations}

\notation{$\K$}{A commutative ring which we use as the base ring for $\Rring$, and for the variety $\fM_{\K}$.}
\notation{\ensuremath{\Rring}}{The cylindrical flavored KLRW algebra.}
\begin{definition}\label{def:cKLRW}  The {\bf cylindrical flavored KLRW algebra} $ {\Rring}$ attached to the data $\Gamma,\Bv,\Bw, \beta_*$  is the quotient of the formal span over a commutative ring $\K$ of cylindrical KLRW  diagrams for these data by the local relations below:
  \newseq\begin{equation*}\subeqn\label{c-first-QH}
    \begin{tikzpicture}[scale=.8,baseline]
      \draw[very thick](-4,0) +(-1,-1) -- +(1,1) node[below,at start]
      {$\stck{i}{a}$}; \draw[very thick](-4,0) +(1,-1) -- +(-1,1) node[below,at
      start] {$\stck{j}{b}$}; \fill (-4.5,.5) circle (3pt);
      \node at (-2.25,0){=}; \draw[very thick](-.5,0) +(-1,-1) -- +(1,1)
      node[below,at start] {$\stck{i}{a}$}; \draw[very thick](-.5,0) +(1,-1) --
      +(-1,1) node[below,at start] {$\stck{j}{b}$}; \fill (0,-.5) circle (3pt);
    \end{tikzpicture}
    \qquad \begin{tikzpicture}[scale=.8,baseline]
      \draw[very thick](-4,0) +(-1,-1) -- +(1,1) node[below,at start]
      {$\stck{i}{a}$}; \draw[very thick](-4,0) +(1,-1) -- +(-1,1) node[below,at
      start] {$\stck{j}{b}$}; \fill (-3.5,.5) circle (3pt);
      \node at (-2.25,0){=}; \draw[very thick](-.5,0) +(-1,-1) -- +(1,1)
      node[below,at start] {$\stck{i}{a}$}; \draw[very thick](-.5,0) +(1,-1) --
      +(-1,1) node[below,at start] {$\stck{j}{b}$}; \fill (-1,-.5) circle (3pt);  \node at (3.75,.5){$i\neq j$ or};
      \node at (3.754,-.5){$a\not\sim b$ };
    \end{tikzpicture}
  \end{equation*}
  \begin{equation*}\subeqn\label{c-third-QH}
    \begin{tikzpicture}[scale=.8,baseline]
      \draw[very thick,dashed](-4,0) +(-1,-1) -- +(1,1) node[below,at start]
      {$\stck{\vphantom{j}e}{c}$}; \draw[very thick](-4,0) +(1,-1) -- +(-1,1) node[below,at
      start] {$\stck{i}{a}$}; \fill (-4.5,.5) circle (3pt);
      \node at (-2,0){=}; \draw[very thick,dashed](0,0) +(-1,-1) -- +(1,1)
      node[below,at start] {$\stck{\vphantom{j}e}{c}$}; \draw[very thick](0,0) +(1,-1) --
      +(-1,1) node[below,at start] {$\stck{i}{a}$}; \fill (.5,-.5) circle (3pt);
    \end{tikzpicture}\qquad \qquad
    \begin{tikzpicture}[scale=.8,baseline]
      \draw[very thick](-4,0) +(-1,-1) -- +(1,1) node[below,at start]
      {$\stck{i}{a}$}; \draw[very thick,dashed](-4,0) +(1,-1) -- +(-1,1) node[below,at
      start] {$\stck{\vphantom{j}e}{c}$}; \fill (-3.5,.5) circle (3pt);
      \node at (-2,0){=}; \draw[very thick](0,0) +(-1,-1) -- +(1,1)
      node[below,at start] {$\stck{i}{a}$}; \draw[very thick,dashed](0,0) +(1,-1) --
      +(-1,1) node[below,at start] {$\stck{\vphantom{j}e}{c}$}; \fill (-.5,-.5) circle (3pt);
    \end{tikzpicture}
  \end{equation*}
  \begin{equation*}\subeqn\label{c-psi2}
    \begin{tikzpicture}[very thick,scale=.8,baseline]
      \draw (-2.8,0) +(0,-1) .. controls (-1.2,0) ..  +(0,1)
      node[below,at start]{$\stck{i}{a}$}; \draw (-1.2,0) +(0,-1) .. controls
      (-2.8,0) ..  +(0,1) node[below,at start]{$\stck{i}{a}$}; \node at (-.5,0)
      {=}; \node at (0.4,0) {$0$};
    \end{tikzpicture}\qquad     \begin{tikzpicture}[scale=.8,baseline,very
      thick]
      \draw[very thick] (-2.8,0) +(0,-1) .. controls (-1.2,0) ..  +(0,1)
      node[below,at start]{$\stck{i}{a}$}; \draw (-1.2,0) +(0,-1) .. controls
      (-2.8,0) ..  +(0,1) node[below,at start]{$\stck{j}{b}$}; 
     \node at (-.5,0){$=$};  \draw[very thick](1.5,0) +(-1,-1) -- +(-1,1)
      node[below,at start]
      {$\stck{i}{a}$}; \draw[very thick](1.5,0) +(0,-1) --
      +(0,1) node[below,at start] {$\stck{j}{b}$};   \node at (4,.5){$i\neq j$ or};
      \node at (4,-.5){$a\not\sim b$ };
    \end{tikzpicture}
      \end{equation*} \begin{equation*}\subeqn\label{c-nilHecke-1}
      \begin{tikzpicture}[scale=.8,baseline]
      \draw[very thick](-4,0) +(-1,-1) -- +(1,1) node[below,at
      start] {$\stck{i}{a}$}; \draw[very thick](-4,0) +(1,-1) -- +(-1,1) node[below,at start]
      {$\stck{i}{b}$}; \fill (-4.5,-.5) circle (3pt);
      \node at (-2.25,0){$-$}; \draw[very thick](-.5,0) +(-1,-1) -- +(1,1)
      node[below,at
      start] {$\stck{i}{a}$}; \draw[very thick](-.5,0) +(1,-1) --
      +(-1,1) node[below,at start]
      {$\stck{i}{b}$}; \fill (0,.5) circle (3pt);
      \node at (1.25,0){$=$};
    \end{tikzpicture}
    \begin{tikzpicture}[scale=.75,baseline]\draw[very thick](-4,0) +(-1,-1) -- +(1,1) node[below,at start]
      {$\stck{i}{a}$}; \draw[very thick](-4,0) +(1,-1) -- +(-1,1) node[below,at
      start] {$\stck{i}{b}$}; \fill (-4.5,.5) circle (3pt);
      \node at (-2.25,0){$-$}; \draw[very thick](-.5,0) +(-1,-1) -- +(1,1)
      node[below,at start]
      {$\stck{i}{a}$}; \draw[very thick](-.5,0) +(1,-1) --
      +(-1,1) node[below,at start] {$\stck{i}{b}$}; \fill (0,-.5) circle (3pt);
     \node at (1.25,0){$=$};  \draw[very thick](3,0) +(-1,-1) -- +(-1,1)
      node[below,at start]
      {$\stck{i}{a}$}; \draw[very thick](3,0) +(0,-1) --
      +(0,1) node[below,at start] {$\stck{i}{b}$};  
      \node at (4.5,0){$a\sim b$ };
    \end{tikzpicture}
  \end{equation*}
\begin{equation*}\subeqn\label{w-cost-1}
  \begin{tikzpicture}[very thick,baseline,scale=.9]
    \draw (-2.8,0)  +(0,-1) .. controls (-1.2,0) ..  +(0,1) node[below,at start]{$\stck{i}{b}$};
       \draw[wei] (-2,0)  +(0,-1)--node[below,at start]{$\stck{i}{a}$}  +(0,1);
  \end{tikzpicture}
= 
  \begin{tikzpicture}[very thick,baseline,scale=.9]
 \draw[wei] (2.3,0)  +(0,-1) -- node[below,at start]{$\stck{i}{b}$} +(0,1);
       \draw (1.5,0)  +(0,-1) -- +(0,1) node[below,at start]{$\stck{i}{a}$};
       \fill (1.5,0) circle (3pt);      \node at (3.5,0){$a\sim b$ };
\end{tikzpicture}\qquad \begin{tikzpicture}[very thick,baseline,scale=.9]
    \draw (-2.8,0)  +(0,-1) .. controls (-1.2,0) ..  +(0,1) node[below,at start]{$\stck{j}{b}$};
       \draw[wei] (-2,0)  +(0,-1)--node[below,at start]{$\stck{i}{a}$}  +(0,1);
  \end{tikzpicture}
= 
  \begin{tikzpicture}[very thick,baseline,scale=.9]
 \draw[wei] (2.3,0)  +(0,-1) -- node[below,at start]{$\stck{i}{a}$} +(0,1);
       \draw (1.5,0)  +(0,-1) -- +(0,1) node[below,at start]{$\stck{j}{b}$};   \node at (3.5,.5){$i\neq j$ or};
      \node at (3.5,-.5){$a\not\sim b$ };
\end{tikzpicture}
\end{equation*}\begin{equation*}
    \subeqn\label{w-cost-2}
  \begin{tikzpicture}[very thick,baseline,scale=.9]
          \draw[wei] (-2,0)  +(0,-1)-- node[below,at start]{$\stck{i}{a}$} +(0,1);
  \draw (-1.2,0)  +(0,-1) .. controls (-2.8,0) ..  +(0,1)
  node[below,at start]{$\stck{i}{b}$};  
  \end{tikzpicture}
=
  \begin{tikzpicture}[very thick,baseline,scale=.9]
    \draw (2.5,0)  +(0,-1) -- +(0,1) node[below,at start]{$\stck{i}{a}$};
       \draw[wei] (1.7,0)  +(0,-1) -- node[below,at start]{$\stck{i}{b}$} +(0,1) ;
       \fill (2.5,0) circle (3pt); \node at (4,0){$a\sim b$ };
 \end{tikzpicture}\qquad  \begin{tikzpicture}[very thick,baseline,scale=.9]
          \draw[wei] (-2,0)  +(0,-1)-- node[below,at start]{$\stck{i}{a}$} +(0,1);
  \draw (-1.2,0)  +(0,-1) .. controls (-2.8,0) ..  +(0,1)
  node[below,at start]{$\stck{j}{b}$};\end{tikzpicture}
=
  \begin{tikzpicture}[very thick,baseline,scale=.9]
    \draw (2.5,0)  +(0,-1) -- +(0,1) node[below,at start]{$\stck{j}{b}$};
       \draw[wei] (1.7,0)  +(0,-1) -- node[below,at start]{$\stck{i}{a}$} +(0,1) ;   \node at (3.5,.5){$i\neq j$ or};
      \node at (3.5,-.5){$a\not\sim b$ };

       \end{tikzpicture}
     \end{equation*}
     Given an edge $e\colon j\to i$, we'll use the convention $a'=a-\beta_e, b'=b-\beta_e,c'=c-\beta_e$, and we have that:
   \begin{equation*}\subeqn\label{w-black-bigon1}
      \begin{tikzpicture}[very thick,scale=.65,baseline]
      \draw(-2.8,0) +(0,-1) .. controls (-1.2,0) ..  +(0,1)
      node[below,at start]{$\stck{i}{a'}$}; 
\end{tikzpicture}\quad   \begin{tikzpicture}[very thick,scale=.65,baseline]
      \draw[dashed] (-2.8,0) +(0,-1) .. controls (-1.2,0) ..  +(0,1)
      node[below,at start]{$\stck{\vphantom{j}e}{a}$}; \draw (-1.2,0) +(0,-1) .. controls
(-2.8,0) ..  +(0,1) node[below,at start]{$\stck{k}{b}$};
\end{tikzpicture}
=
\begin{cases}
  \begin{tikzpicture}[very thick,scale=.65,baseline]
      \draw(-2.8,0) +(0,-1) -- +(0,1)
      node[below,at start]{$\stck{i}{a'}$}; 
      \draw[dashed] (-.8,0) +(0,-1)-- +(0,1)
      node[below,at start]{$\stck{\vphantom{j}e}{a}$}; \draw (.2,0) +(0,-1) --+(0,1) node[below,at start]{$\stck{k}{b}$};
    \end{tikzpicture} & j\neq k\text{ or }a\not\sim b\\
\begin{tikzpicture}[very thick,scale=.65,baseline]
      \draw(-2.8,0) +(0,-1) -- +(0,1)
      node[below,at start]{$\stck{i}{a'}$}; 
      \draw[dashed] (-.8,0) +(0,-1)-- +(0,1)
      node[below,at start]{$\stck{\vphantom{j}e}{a}$}; \draw (.2,0) +(0,-1) --node[midway,fill=black, inner sep=2pt, circle]{}+(0,1) node[below,at start]{$\stck{k}{b}$};
\end{tikzpicture}-\begin{tikzpicture}[very thick,scale=.65,baseline]
      \draw(-2.8,0) +(0,-1) -- node[midway,fill=black, inner sep=2pt, circle]{} +(0,1)
      node[below,at start]{$\stck{i}{a'}$}; 
      \draw[dashed] (-.8,0) +(0,-1)-- +(0,1)
      node[below,at start]{$\stck{\vphantom{j}e}{a}$}; \draw (.2,0) +(0,-1) --+(0,1) node[below,at start]{$\stck{k}{b}$};
\end{tikzpicture}    & j=k\text{ and }a\sim b
\end{cases}
\end{equation*}
   \begin{equation*}\subeqn\label{w-black-bigon2}
      \begin{tikzpicture}[very thick,scale=.65,baseline]
      \draw (-1.2,0) +(0,-1) .. controls
(-2.8,0) ..  +(0,1) node[below,at start]{$\stck{i}{a'}$};
\end{tikzpicture}\quad   \begin{tikzpicture}[very thick,scale=.65,baseline]
      \draw[dashed] (-1.2,0) +(0,-1) .. controls
(-2.8,0) ..  +(0,1) 
      node[below,at start]{$\stck{\vphantom{j}e}{a}$}; \draw (-2.8,0) +(0,-1) .. controls (-1.2,0) ..  +(0,1) node[below,at start]{$\stck{k}{b}$};
\end{tikzpicture}
=
\begin{cases}
  \begin{tikzpicture}[very thick,scale=.65,baseline,xscale=.9]
    \draw(-1.8,0) +(0,-1) -- +(0,1) node[below,at
    start]{$\stck{i}{a'}$}; \draw[dashed](.2,0) +(0,-1)-- +(0,1) node[below,at
    start]{$\stck{\vphantom{j}e}{a}$}; \draw (-.8,0)+(0,-1) -- +(0,1) node[below,at start]{$\stck{k}{b}$};
  \end{tikzpicture}& j\neq k\text{ or }a\not\sim b\\
  \begin{tikzpicture}[very thick,scale=.65,baseline,xscale=.9]
    \draw(-1.8,0) +(0,-1) -- +(0,1) node[below,at
    start]{$\stck{i}{a'}$}; \draw[dashed](.2,0) +(0,-1)-- +(0,1) node[below,at
    start]{$\stck{\vphantom{j}e}{a}$}; \draw (-.8,0)+(0,-1) --node[midway,fill=black, inner
    sep=2pt, circle]{}+(0,1) node[below,at start]{$\stck{k}{b}$};
  \end{tikzpicture}-\begin{tikzpicture}[very
    thick,scale=.65,baseline,xscale=.9] \draw(-1.8,0) +(0,-1) --
    node[midway,fill=black, inner sep=2pt, circle]{} +(0,1)
    node[below,at
    start]{$\stck{i}{a'}$}; \draw[dashed] (.2,0) +(0,-1)-- +(0,1) node[below,at
    start]{$\stck{\vphantom{j}e}{a}$}; \draw (-.8,0) +(0,-1) --+(0,1) node[below,at
    start]{$\stck{k}{b}$};
  \end{tikzpicture}& j=k\text{ and }a\sim b\\
\end{cases}
\end{equation*}
In the next three relations, we assume that $ a\sim b\sim c$:  
\begin{equation*}\subeqn
    \begin{tikzpicture}[very thick,baseline]\label{red-triple-correction}
      \draw (-3,0)  +(1,-1) -- +(-1,1) node[at start,below]{$\stck{i}{c}$};
      \draw (-3,0) +(-1,-1) -- +(1,1)node [at start,below]{$\stck{i}{a}$};
      \draw[wei] (-3,0)  +(0,-1) .. controls (-4,0) .. node[below, at start]{$\stck{i}{b}$}  +(0,1);
      \node at (-1,0) {=};
      \draw (1,0)  +(1,-1) -- +(-1,1) node[at start,below]{$\stck{i}{c}$};
      \draw (1,0) +(-1,-1) -- +(1,1) node [at start,below]{$\stck{i}{a}$};
      \draw[wei] (1,0) +(0,-1) .. controls (2,0) ..  node[below, at start]{$\stck{i}{b}$} +(0,1);   
\node at (2.6,0) {$+ $};
      \draw (4.5,0)  +(1,-1) -- +(1,1) node[at start,below]{$\stck{i}{c}$};
      \draw (4.5,0) +(-1,-1) -- +(-1,1) node [at start,below]{$\stck{i}{a}$};
      \draw[wei] (4.5,0) +(0,-1) -- node[below, at start]{$\stck{i}{b}$} +(0,1);
 \end{tikzpicture} 
  \end{equation*}
\begin{equation*}\subeqn\label{w-triple-point}
    \begin{tikzpicture}[very thick,xscale=1.7,baseline]
      \draw[dashed] (-2.5,0) +(.35,-1) -- +(-.35,1) node[below,at start]{$\stck{\vphantom{j}e}{c}$};
 \draw[dashed]      (-2.5,0) +(-.35,-1) -- +(.35,1) node[below,at start]{$\stck{\vphantom{j}e}{a}$}; 
    \draw (-1.5,0) +(.35,-1) -- +(-.35,1) node[below,at start]{$\stck{i}{c'}$}; \draw
      (-1.5,0) +(-.35,-1) -- +(.35,1) node[below,at start]{$\stck{i}{a'}$}; 
 \draw (-2.5,0) +(0,-1) .. controls (-3,0) ..  +(0,1) node[below,at
      start]{$\stck{j}{b}$};\node at (-.75,0) {=};  \draw[dashed] (0,0) +(.35,-1) -- +(-.35,1) node[below,at start]{$\stck{\vphantom{j}e}{c}$}; ;
 \draw[dashed]      (0,0) +(-.35,-1) -- +(.35,1) node[below,at start]{$\stck{\vphantom{j}e}{a}$}; 
    \draw (1,0) +(.35,-1) -- +(-.35,1) node[below,at start]{$\stck{i}{c'}$}; \draw
      (1,0) +(-.35,-1) -- +(.35,1) node[below,at start]{$\stck{i}{a'}$}; 
 \draw (0,0) +(0,-1) .. controls (.5,0) ..  +(0,1) node[below,at
      start]{$\stck{j}{b}$};
\node at (2,0) {$+$};
     \draw (4,0)
      +(.35,-1) -- +(.35,1) node[below,at start]{$\stck{i}{c'}$}; \draw (4,0)
      +(-.35,-1) -- +(-.35,1) node[below,at start]{$\stck{i}{a'}$}; 
 \draw[dashed] (3,0)
      +(.35,-1) -- +(.35,1) node[below,at start]{$\stck{\vphantom{j}e}{c}$}; \draw[dashed] (3,0)
      +(-.35,-1) -- +(-.35,1) node[below,at start]{$\stck{\vphantom{j}e}{a}$}; 
\draw (3,0)
      +(0,-1) -- +(0,1) node[below,at start]{$\stck{j}{b}$};
\end{tikzpicture}
  \end{equation*}
\begin{equation*}\subeqn\label{w-triple-point2}
    \begin{tikzpicture}[very thick,xscale=1.6,yscale=.8,baseline]
\draw[dashed] (-2.5,0) +(0,-1) .. controls (-3,0) ..  +(0,1) node
[below, at start]{$\stck{\vphantom{j}e}{b}$};  
  \draw (-2.5,0) +(.35,-1) -- +(-.35,1) node[below,at start]{$\stck{j}{c}$}; \draw
      (-2.5,0) +(-.35,-1) -- +(.35,1) node[below,at start]{$\stck{j}{a}$}; 
 \draw (-1.5,0) +(0,-1) .. controls (-2,0) ..  +(0,1) node[below,at
      start]{$\stck{i}{b'}$};\node at (-.75,0) {=};  
    \draw (0,0) +(.35,-1) -- +(-.35,1) node[below,at start]{$\stck{j}{c}$}; \draw
      (0,0) +(-.35,-1) -- +(.35,1) node[below,at start]{$\stck{j}{a}$}; 
 \draw[dashed] (0,0) +(0,-1) .. controls (.5,0) ..  +(0,1) node[below,at start]{$\stck{\vphantom{j}e}{b}$};
 \draw (1.5,0) +(0,-1) .. controls (2,0) ..  +(0,1) node[below,at
      start]{$\stck{i}{b'}$};
\node at (2.25,0)
      {$-$};   
     \draw (3,0)
      +(.35,-1) -- +(.35,1) node[below,at start]{$\stck{j}{c}$}; \draw (3,0)
      +(-.35,-1) -- +(-.35,1) node[below,at start]{$\stck{j}{a}$}; 
\draw[dashed] (3,0)
      +(0,-1) -- +(0,1) node[below,at start]{$\stck{\vphantom{j}e}{b}$};\draw (4.5,0)
      +(0,-1) -- +(0,1) node[below,at start]{$\stck{i}{b'}$};
\end{tikzpicture}.
  \end{equation*}
For all other triple points, we set the two sides of the isotopy
through it equal.
\end{definition}

\begin{remark}
Note that in many earlier works, such as \cite{Webmerged, WebRou}, we
had an additional non-local relation setting a diagram to 0 if it had
a black strand at far left of the diagram.  We just wish to clarify that we are not imposing that relation here.
\end{remark}

\subsubsection{Grading}

As with usual (planar) KLRW algebras, this algebra is graded by a notion of degree of KLRW diagrams:
\begin{definition}
	The degree of a diagram is a sum of local contributions from dots and crossings given by:
	\begin{equation}
   \deg    \begin{tikzpicture}[scale=.6,baseline]
      \draw[very thick](-4,0) +(-1,-1) -- +(1,1) node[below,at start]
      {$\stck{i}{a}$}; \draw[very thick](-4,0) +(1,-1) -- +(-1,1) node[below,at
      start] {$\stck{j}{b}$}; 
    \end{tikzpicture}
  =
  \begin{cases}
    -2 & i=j \text{ and }a\sim b \\
   0 & i \neq j \text{ or }a\not \sim b\\
  \end{cases} \qquad 
 \deg  \begin{tikzpicture}[scale=.6,baseline]
      \draw[very thick](-4,0) +(0,-1) -- +(0,1) node[below,at start]
      {$\stck{i}{a}$};  \fill (-4,0) circle (3pt);
     \end{tikzpicture} =2 
	\end{equation}
	\begin{equation}
		\deg    \begin{tikzpicture}[scale=.6,baseline]
      \draw[very thick](-4,0) +(-1,-1) -- +(1,1) node[below,at start]
      {$\stck{i}{a}$}; \draw[very thick, dashed ](-4,0) +(1,-1) -- +(-1,1) node[below,at
      start] {$\stck{\vphantom{j}e}{b}$}; 
    \end{tikzpicture}=  \begin{cases}
    1 & i=t(e) \text{ and }a\sim b \\
   0 & i \neq t(e) \text{ or }a\not \sim b\\
  \end{cases}	\end{equation}
	\begin{equation}	\deg    \begin{tikzpicture}[scale=.6,baseline]
      \draw[very thick](-4,0) +(-1,-1) -- +(1,1) node[below,at start]
      {$\stck{i}{a}$}; \draw[wei ](-4,0) +(1,-1) -- +(-1,1) node[below,at
      start] {$\stck{\vphantom{i}j}{b}$}; 
    \end{tikzpicture}=\begin{cases}
    1 & i=j \text{ and }a\sim b \\
   0 & i \neq j \text{ or }a\not \sim b\\
  \end{cases}
	\end{equation}
\end{definition} 
This induces a grading on $\Rring$, 
since the relations above are homogeneous with respect to the grading.

\subsubsection{Comparison with Coulomb branches}

Given $\boldsymbol{\alpha}\in \mathbb{L}$, we have an associated idempotent $e(\boldsymbol{\alpha})$.  
By construction, the strands with label $i$ and longitude $a$ are all consecutive to each other in this idempotent.   Let
$\mu_{i,a}$ be the number of strands with label $i$ and longitude $a$.
 Acting with crossings and dots on the strands with label $i$ and longitude $a$ gives a homomorphism
of 
the nilHecke algebra of rank $\mu_{i,a}$ to
$e(\boldsymbol{\alpha})\Rring e(\boldsymbol{\alpha})$. This nilHecke algebra contains
a primitive idempotent projecting to the invariants of the symmetric group $\Sigma_{\mu_{i,a}}$  in
the usual polynomial representation.
\begin{definition}\label{def:i-a2}
   We let $e'(\boldsymbol{\alpha})\in
 {\Rring}$ be the product of these nilHecke idempotents over all pairs of a 
vertex $i$ and value $a$.  
\end{definition}

Note that the idempotent is independent of the choice of order of
$\vertex$ up to equivalence because we can reorder strands with
different labels by (\ref{c-third-QH}).  

\begin{definition}\label{def:ea}
  Let $e(a)=e'((a,\dots, a))$ be the idempotent where we set all $\alpha_{i,k}$'s equal to a single
value $a\neq \beta_{i,j}$. In this
case, the ordering of $\Bi$ is irrelevant; when constructing the flavored sequence from Definition \ref{def:prefered} will put the indices in our fixed order, and as discussed above,
the order won't change the isomorphism type of the corresponding
idempotent. 
\end{definition}
\notation{$e(a)$}{The idempotent defined by setting all longitudes equal to $a$ and projecting by a primitive idempotent in the nilHecke algebra (Definition \ref{def:ea}).  }

One primary reason for our interest in this idempotent is the following
result:
\begin{theorem}\label{coulomb-idempotent}
  The algebra $A_0=e(a) \Rring e(a)$ is isomorphic to the (undeformed) Cou\-lomb branch algebra $\K[\Coulomb]$ of the quiver gauge theory associated to $\Gamma$ with the dimension vectors $\Bv,\Bw$.
\end{theorem}
This is proven on page \pageref{proof-coulomb-idempotent}.
Note that this theorem is independent of the choice of $a$ and of $\beta_*$; we obtain the same algebra $A_0$ here regardless of
these parameters.  This is a generalization of
\cite[Cor. 4.13]{weekesGeneratorsCoulomb2019}.   
More generally, for each $e(\Bi,\Ba)$, we have a corresponding vortex
line operator, that is, an object in the category $\scrB^+$ (Definition \ref{I-def:extended-BFN}), and
$\Rring$ is the sum of the morphism spaces between these
objects; see Physics Motivation \ref{I-physics:line}.

\subsection{Change of flavor}
\label{sec:change-flavor-1}

Now we consider the relationship between different choices of flavoring sets $\Ab$ and flavor $\beta$.  
As in \cite{kamnitzerLieAlgebra2024}, we consider two groups $\Ab,\Ab'$ satisfying the conditions of Section \ref{sec:cyl-flav-seq} as above,  and a choice of subset $\corre\subset \Ab \times \Ab'$, which we assume has the property that 
\begin{itemize}
    \item[$(\ddag)$] For any  two pairs $(a,b), (a',b')\in \corre $,  the statements $a\sim a'$ and $b\sim b'$ are both true or both false.  That is, $\corre $ induces a bijection between subsets of the equivalence classes in $\Ab,\Ab'.$  
\end{itemize}  
The simplest case is when $\Ab,\Ab'$ both have a single equivalence class, in which case $\corre =\Ab\times \Ab'$ is the only possible non-empty choice.   One case of special interest to us will be when $\Ab=\R/\Z$, and $\Ab'=\mathbb{F}_p$, with $\corre =\R/\Z\times \mathbb{F}_p$.  We further choose flavors 
\[\beta\colon \edge(\Gamma^{\Bw})\to \Ab\qquad \beta'\colon \edge(\Gamma^{\Bw})\to \Ab'\] such that $(\beta_e,\beta_e')\in \corre$ for all $e\in \edge(\Gamma^{\Bw})$.

We'll want to consider below the case where we have three sets $\Ab,\Ab',\Ab''$ and correspondences $\corre\subset \Ab\times \Ab'$ and $\corre'\subset \Ab' \times \Ab''$ that both satisfy condition $(\ddag)$ above.  Clearly: 
\begin{lemma}
In this case, the composition $\corre\circ \corre'$ also satisfies the condition $(\ddag)$.
\end{lemma}

We want to consider a bimodule consisting of KLRW diagrams with the top labeled by flavored sequences in $\Ab$ and the bottom labeled by flavored sequences in $\Ab'$.  We also want to modify condition (3) in the definition of KLRW diagrams.  Given $\mathbf{\inter}\colon \edge(\Gamma^{\Bw}) \to \Z$, we define a $\mathbf{\inter}$-twisted KLRW diagram to be one where the top is labeled by a flavored sequence from $(\Ab,\beta)$, the bottom by a flavored sequence from $(\Ab',\beta')$ and condition (3) is changed to:
\begin{itemize}
    \item [(3')] If $(k,e)\in \Sg$, then the corresponding ghostly strand and the corporeal strand for $k$ have $\inter_{e}$ signed intersections.  If $(\star,e)\in \Sr$, then the corresponding red strand and vertical line $x=0$ have $\inter_e$ signed intersections.  Note that we will have this behavior if we draw our diagrams with a difference in $\R/\Z$ between the corporeal strand and its ghost given by $\inter_{e}t+r$ for some $r\in \R/\Z$ and the put red strands at $x$-value $\inter_{i,k}t+r$.
\end{itemize}

Below, we show an example of a twisted cylindrical KLRW diagram, once in our usual convention where the cylinder is cut open, and the same
example in perspective on a cylinder:
\begin{equation*}
  \begin{tikzpicture}[scale=2]
    \draw[very thick] (-.58,-1)-- (-1.18,1);
    \draw[very thick] (1.40,-1)-- (0.82,1);
    \draw[very thick] (-.75,-1)-- (.58,1);
    \draw[very thick] (-1.08,-1)-- (-.48,1); 
    \draw[very thick] (.92,-1)-- (1.52,1);
    \draw[wei] (-.4,-1)--(1.6,1);
    \draw[wei] (-2.4,-1)--(-.4,1);
    \draw[very thick,dashed] (.53,-1)-- (-1.97,1);
    \draw[very thick,dashed] (2.53,-1)-- (.03,1);
    \draw[very thick,dashed] (.36,-1)-- (-.31,1);
    \draw[very thick,dashed] (.03,-1)-- (-1.47,1); 
    \draw[very thick,dashed] (2.03,-1)-- (.53,1);
    \fill[white](-1,-1) -- (-1,1)-- (-2.5,1)-- (-2.5,-1)--cycle;
    \fill[white](1,-1) -- (1,1)-- (2.54,1)-- (2.54,-1)--cycle;
    \draw[fringe] (-1,-1) -- (-1,1);
    \draw[fringe] (1,1) -- (1,-1);
  \end{tikzpicture}
  \begin{tikzpicture}[MyPersp,font=\large]
\def\h{1.5}

\fill[blue,fill opacity=.05]  
		 (1,0,{\h})--(1,0,0)
		\foreach \t in {0,-2,-4,...,-180}
			{--({cos(\t)},{sin(\t)},0)}
-- (-1,0,0)--(-1,0,{\h})
		\foreach \t in {180,178,...,0}
			{--({cos(\t)},{sin(\t)},{\h})}--cycle;
\draw[gray,   very thick] (1,0,0)
		\foreach \t in {0,2,4,...,180}
			{--({cos(\t)},{sin(\t)},0)};
\draw[dashed, very thick] ({cos(50)},{sin(50)},0) 
		\foreach \t in {0,2,...,45}
		{--({cos((3*\t+50))} ,{sin((3*\t+50))},{.01*\t})};
\draw[dashed, very thick] ({cos((3*101+50))} ,{sin((3*101+50))},{.01*101}) 
				\foreach \t in {101,103,...,150}
		{--({cos((3*\t+50))} ,{sin((3*\t+50))},{.01*\t})};
\draw[dashed, very thick] ({cos(140)},{sin(140)},0) 		\foreach \t in {0,2,...,16}
		{--({cos((1.8*\t+140))} ,{sin((1.8*\t+140))},{.01*\t})};
\draw[dashed, very thick] ({cos((1.8*124+140))} ,{sin((1.8*124+140))},{.01*124}) 
		\foreach \t in {124,126,...,150}
		{--({cos((1.8*\t+140))} ,{sin((1.8*\t+140))},{.01*\t})};
\draw[dashed, very thick] ({cos(80)},{sin(80)},0) 
		\foreach \t in {0,2,...,120}
		{--({cos((.8*\t+80))} ,{sin((.8*\t+80))},{.01*\t})};
\draw[ very thick] ({cos((-1.6*60-80))} ,{sin((-1.6*60-80))},{.01*60}) 
		\foreach \t in {60,62,...,150}
		{--({cos((-1.6*\t-80))}
                  ,{sin((-1.6*\t-80))},{.01*\t})};
\draw[ wei] ({cos((-2.4*24-130))} ,{sin((-2.4*24-130))},{.01*24}) 
		\foreach \t in {26,28,...,106}
		{--({cos((-2.4*\t-130))} ,{sin((-2.4*\t-130))},{.01*\t})};
\fill[blue!20!white,fill opacity=.5]
		 (1,0,{\h})--(1,0,0)
		\foreach \t in {0,-2,-4,...,-180}
			{--({cos(\t)},{sin(\t)},0)}
-- (-1,0,0)--(-1,0,{\h})
		\foreach \t in {-180,-178,...,0}
			{--({cos(\t)},{sin(\t)},{\h})}--cycle;
\draw[dashed, very thick] ({cos((3*45+50))} ,{sin((3*45+50))},{.01*45}) 
		\foreach \t in {45,47,...,101}
		{--({cos((3*\t+50))} ,{sin((3*\t+50))},{.01*\t})};
\draw[ very thick] ({cos(-110)},{sin(-110)},0) 
		\foreach \t in {0,2,...,150}
		{--({cos((.6*\t-110))} ,{sin((.6*\t-110))},{.01*\t})};
\draw[dashed, very thick] ({cos((1.8*16+140))} ,{sin((1.8*16+140))},{.01*16}) 
		\foreach \t in {16,18,...,124}
		{--({cos((1.8*\t+140))} ,{sin((1.8*\t+140))},{.01*\t})};
\draw[ very thick] ({cos(-20)},{sin(-20)},0) 
		\foreach \t in {0,2,...,150}
		{--({cos((-.6*\t-20))} ,{sin((-.6*\t-20))},{.01*\t})};
\draw[fringe] ({cos(-35)},{sin(-35)},0)--({cos(-35)},{sin(-35)},\h);
\draw[dashed, very thick] ({cos((.8*120+80))},{sin((.8*120+80))},{.01*120}) 
		\foreach \t in {120,122,...,150}
		{--({cos((.8*\t+80))} ,{sin((.8*\t+80))},{.01*\t})};
\draw[ very thick] ({cos(-80)},{sin(-80)},0) 
		\foreach \t in {0,2,...,60}
		{--({cos((-1.6*\t-80))}
                  ,{sin((-1.6*\t-80))},{.01*\t})};
\draw[ wei] ({cos((-130))} ,{sin((-130))},0) 
		\foreach \t in {2,4,...,24}
		{--({cos((-2.4*\t-130))}
                  ,{sin((-2.4*\t-130))},{.01*\t})};
\draw[ wei] ({cos((-2.4*104-130))} ,{sin((-2.4*104-130))},{.01*104}) 
		\foreach \t in {106,108,...,150}
		{--({cos((-2.4*\t-130))} ,{sin((-2.4*\t-130))},{.01*\t})};
\draw[gray] (1,0,0)--(1,0,{\h});
\draw[gray] (-1,0,0)--(-1,0,{\h});
\draw[gray, very thick] (1,0,0) 
		\foreach \t in {0,-2,-4,...,-180}
			{--({cos(\t)},{sin(\t)},0)};
\draw[gray, very thick] (1,0,\h) 
		\foreach \t in {2,4,...,360} 
			{--({cos(\t)},{sin(\t)},{\h})}--cycle;
                      \end{tikzpicture}
                      \qquad
                     \qquad
\end{equation*}

\begin{definition}
\label{def:bimodule}
    Let $\rif^{\mathbf{\inter}}(\corre)$ be the set of $\mathbf{\inter}$-twisted flavored KLRW diagrams equipped with a $\Ab$-flavored sequence at the top and a $\Ab'$-flavored sequence at the bottom, such that the labels $t$ at the top and $b$ at the bottom of each strand satisfy $(t,b)\in \corre$ modulo the flavored KLRW relations (\ref{c-first-QH}--\ref{w-triple-point2}) in the case where the labels $T$ at the top of the strands involved and labels $B$ at the bottom of all strands involved satisfy $T\times B\subset \corre $, and the isotopy relations otherwise.  
\end{definition}

\notation{$\rif^{\mathbf{\inter}}(\corre)$}{The natural bimodule between cylindrical flavored KLR algebras determined by a correspondence $\corre$, and a vector $\mathbf{\inter}$ which controls the number of times a ghost intersects with its corporeal (Definition \ref{def:bimodule}).}

  We can compose twisted KLRW diagrams with matching top and bottom as usual:
  
  \begin{lemma}\label{lem:bimod-mult}
  For functions $\mathbf{\inter},\mathbf{\inter}'$ and correspondences $\corre\subset \longi\times \longi'$ and $\corre'\subset \longi' \times \longi''$ composition of diagrams induces  a map $\rif^{\mathbf{\inter}}(\corre)\otimes \rif^{\mathbf{\inter}'}(\corre')\to \rif^{\mathbf{\inter}+\mathbf{\inter}'}(\corre\circ \corre')$, which is associative.
  \end{lemma}
  \begin{proof}
  First, note that the composition of diagrams sends a pair of diagrams in $\rif^{\mathbf{\inter}}(\corre)$ and $ \rif^{\mathbf{\inter}'}(\corre')$ to one in $\rif^{\mathbf{\inter}+\mathbf{\inter}'}(\corre\circ \corre')$.  To show that this map is well-defined on the vector spaces, we need to check that any relation in $\rif^{\mathbf{\inter}}(\corre)$ or $\rif^{\mathbf{\inter}'}(\corre')$ is sent to a relation in $\rif^{\mathbf{\inter}+\mathbf{\inter}'}(\corre\circ \corre')$ under product with an arbitrary diagram.       This follows from the locality of the isotopy and flavored KLRW relations.  If we consider a relation in $\rif^{\mathbf{\inter}}(\corre)$ where all strands involved satisfy $T\times B\subset \corre$, then the composition with a diagram of $\rif^{\mathbf{\inter}'}(\corre')$ will change the labels at the bottom of these strands to a different set $B'$, but each element of $B'$ is related to an element of $B$ by $\corre'$, so $T\times B'\subset \corre\circ \corre'$.  
  
  Now assume that $T\times B\not\subset \corre$. This implies that at least two of the labels in $T$ are not equivalent, so the same is true of the labels in $B$ and $B'$, so $T\times B'\not\subset \corre\circ \corre'$.  
  
  Thus, in either case, the same local relation holds in $\rif^{\mathbf{\inter}+\mathbf{\inter}'}(\corre\circ \corre')$.  The same argument works for relations in $\rif^{\mathbf{\inter}'}(\corre')$ so this map is well-defined.  
  \end{proof}
  
   If $\Ab=\Ab'$, $\corre$ is the equivalence relation $\sim$ and $\mathbf{\inter}=\mathbf{0}$, then $\rif^{\mathbf{\inter}}(\corre)=\Rring(\Ab)$.  Thus,   Lemma \ref{lem:bimod-mult} defines a $\Rring(\Ab)\operatorname{-}\Rring(\Ab')$-bimodule structure on $\rif^{\mathbf{\inter}}(\corre)$ for any correspondence $\corre\subset \Ab\times \Ab'$ as above.

\subsubsection{Relation to resolved Coulomb branches}  
For any fixed $\Ab=\Ab'$ and  $\mathbf{\inter}$, with $\corre$ the equivalence relation $\sim$,  we have multiplications
  \begin{equation}
\rif^{k\mathbf{\inter}}(\sim)\otimes
  \rif^{m\mathbf{\inter}}(\sim)\to
  \rif^{(k+m)\mathbf{\inter}}(\sim).\label{eq:k-m-mult}
\end{equation} We can combine these into a $\Z_{\geq 0}$-graded ring $\ProjC^{\mathbf{\inter}}=\bigoplus_{k\geq
    0}\rif^{k\mathbf{\inter}}(\sim)$.  
  \begin{theorem}\label{thm:partial-resolution}
    The ring $\mathbf{A}^{\mathbf{\inter}}=e(a) \ProjC^{\mathbf{\inter}} e(a)$ is commutative, and $\Proj(\mathbf{A}^{\mathbf{\inter}})$ is the partial resolution $ {\tilde{\fM}^{\mathbf{\inter}}}$ defined in \cite{BFNline}.
  \end{theorem}
  This is a more precise statement of Theorem \ref{thm:D-equivalence}(1).
This is proven on page \pageref{proof-thm:partial-resolution}.
 \begin{definition}\label{def:BFN-resolution}
 	If for some $\mathbf{\inter}$, the space $\tilde{\fM}^{\mathbf{\inter}}$ is a resolution of singularities,  we call it a {\bf BFN resolution}. 
 \end{definition}  By \cite[Thm. 5]{weekesQuiverGauge2022}, a BFN resolution is necessarily symplectic.  By \cite[Prop. 19]{namikawaFlopsPoisson2008}, the space $\tilde{\fM}^{\mathbf{\inter}}$ is a symplectic resolution if and only if the twistor deformation for the corresponding ample line bundle is generically smooth (i.e. the Coulomb branch becomes smooth when we consider $\mathbf{\inter}$ as FI parameters).

  This allows us to state one of the main results of our paper.  Let $\Ab=\R/\Z$, and assume we have fixed a choice of flavors $\beta_e$, which are generic in the sense of avoiding the subtori defined by \eqref{eq:circuits}:
  \begin{theorem}\label{thm:NCSR}
    If the space $\tilde{\fM}^{\mathbf{\inter}}$ is a symplectic resolution, then the ring $ {\Rring_{\beta}}$ is a noncommutative symplectic resolution of singularities  and $\D^b(\Coh(\tilde{\fM}^{\mathbf{\inter}}))\cong D^b(\Rring_{\beta}\mmod)$.
  \end{theorem}
  This follows from Theorem \ref{th:Q-equiv-2}. One point we should emphasize here is that $\beta_e$ is in no way related to $\mathbf{\inter}$.  We prove this result by showing that $\tilde{\fM}^{\mathbf{\inter}}$ possesses a tilting generator with endomorphisms given by $\Rring_{\beta}$ for any generic choice of $\beta$, which proves the desired properties.

\subsubsection{Wall crossing functors}
\label{sec:wall-cross}

\notation{$\hat\beta=( {\hat\beta_e}, {\hat\beta_{i,k}})$}{A point in $\R^{\edge}\times \prod_{i\in \vertex}{\R^{w_i}}$, which is a real lift of the choice of flavor $\beta$ in $\R/\Z$.}
Let $V=\R^{\edge}\times \prod_{i\in \vertex}{\R^{w_i}}$.  
Given a point $\hat\beta=( {\hat\beta_e}, {\hat\beta_{i,k}})$ in this space, we can interpret it as the lift to $\R$ of a flavor in the group $\Ab=\R/\Z$.
For this choice of flavor, we have a corresponding algebra $ {\Rring_{\beta}}$, and for any pair $\hat\beta,\hat\beta'$, we have a bimodule $ {\rif_{\hat\beta',\hat\beta}}=\rif^{\mathbf{\inter}}(\corre)$, where $\corre\subset \R/\Z\times \R/\Z$ is the full correspondence, and ${\inter_e}=\lfloor \hat\beta_e'\rfloor-\lfloor \hat\beta_e\rfloor$; since $\mathbf{\inter}$ is determined by $\hat\beta$ and $\hat\beta'$ it can be safely left out from the notation.
This last statistic might look strange, but it reflects the number of crossings we will find if the distance between corporeal and ghost strands continuously varies from $\hat\beta$ to $\hat\beta'$, for example, if it is given at $y=t$ by 
$\bbeta=(1-t)\hat\beta+t\hat\beta'$.  
\notation{$\rif_{\hat\beta',\hat\beta}$}{A bimodule over the algebras $\Rring_{\beta}$ and $\Rring_{\beta'}$, with the number of intersections of each ghost with the corresponding strand determined by the choice of lifts $\hat\beta',\hat\beta$.}

Another point where the reader should be cautious: while the algebras $\Rring_{\beta}$ and $\Rring_{\beta'}$ only depend on the image of these points modulo $\Z$, the bimodule depends on the difference $\hat\beta-\hat\beta'$, since $d_e$ depends on this difference.

We can define an equivalence relation on $V$ by $\hat\beta\sim\hat\beta'$ if $\rif_{\hat\beta',\hat\beta}$ and $\rif_{\hat\beta,\hat\beta'}$ induce Morita equivalences.  Considering the bimodules $\rif_{\hat\beta',\hat\beta}$ and $\rif_{\hat\beta,\hat\beta'}$ as a Morita context (or ``pre-equivalence datum'' in the terminology of \cite{BassK}), by \cite[II.3.4]{BassK}, this defines a Morita equivalence if and only  the multiplication maps 
\[\rif_{\hat\beta',\hat\beta}\otimes \rif_{\hat\beta,\hat\beta'}\to \Rring_{\hat\beta}\qquad \rif_{\hat\beta,\hat\beta'}\otimes \rif_{\hat\beta',\hat\beta}\to \Rring_{\hat\beta'}\]
are surjective.

\begin{proposition}\label{prop:alcove-equivalence}
This equivalence relation is refined by the alcoves of the hyperplane arrangement defined by the equations 
  \begin{equation}\label{eq:circuit-hyperplane}
       {\hat\beta_{j,\ell}}-\hat\beta_{i,k}+\sum_{p=1}^n \varepsilon_i {\hat\beta_{e_i}}=m
  \end{equation} for each unoriented path \[i=i_0\overset{e_1}\longrightarrow i_1\overset{e_2}\longrightarrow i_2\overset{e_3}\longrightarrow\cdots \overset{e_n}\longrightarrow i_n=j\] in $\Gamma$, all $k\in [1,w_i],\ell\in [1,w_j]$ and all integers $m$; as before, $\varepsilon\in \{\pm 1\}$ keeps track of whether the unoriented path matches or reverses the orientation of the quiver on the edge $e_i$.

If we restrict to choices of flavor where $\hat\beta_e=0$ for all $e\in \edge$, then the third term on the LHS is identically 0, and so we have the hyperplane arrangement   \begin{equation}\label{eq:circuit-hyperplane01}
       {\hat\beta_{j,\ell}}-\hat\beta_{i,k}=m.
  \end{equation}
 \end{proposition}
\begin{proof}
If there is a path from $\hat\beta$ to $\hat\beta'$ that never crosses one of these hyperplanes, then Lemma \ref{lem:circuits} shows that the set of equivalence classes of loadings never changes, and so we can draw diagrams in $\rif_{\hat\beta',\hat\beta}$ and $\rif_{\hat\beta,\hat\beta'}$ with no strands crossing which join any loading in $\Rring_{\beta}$ to the same loading in $\Rring_{\beta'}$.  Multiplying these in either order gives the desired idempotent, and shows the Morita equivalence.
\end{proof}
If the bimodules $ {\rif_{\hat\beta',\hat\beta}}$ and $\rif_{\hat\beta,\hat\beta'}$ do not induce Morita equivalences, then the situation is more complicated.  
\begin{conjecture}\label{conj-DE}
For any generic $\hat\beta,\hat\beta'$, the bimodule $\rif_{\hat\beta',\hat\beta}\Lotimes -$ induces an equivalence of derived
categories for any coefficient ring $\K$.
\end{conjecture}
Establishing this in full generality requires more algebraic machinery than we want to develop at this moment, but we can use already known results to show it holds in the most important cases for us:
\begin{proposition}\label{prop:DE} For each $\Gamma,\Bv,\Bw$, 
Conjecture \ref{conj-DE} holds for $\K=\Q$ (or more generally any
characteristic 0 field) if the space
$\tilde{\fM}^{\mathbf{\inter}}$ is a symplectic resolution for
generic $\mathbf{\inter}$.  
\end{proposition} This is proven on page \pageref{proof-prop:DE}.

  \subsubsection{Real variation of stability conditions}
  \label{sec:real-variation}
  
  Recall the definition of a real variation of stability conditions:
let $\mathcal{D}$ be a finite type triangulated category and $V$ a real vector space.
Fix a finite set $\hat{\Sigma}\subset V^*$ of linear functionals on $V$.  By changing sign if necessary, we may assume that there's an open subset $V^+$ where all these functionals are positive.  Consider the ``unrolled'' hyperplane arrangement  $\Sigma$ given by 
\[H_{k,\xi}=\{v\in V \mid \xi(v)=k \}\qquad \xi\in \hat{\Sigma}, k\in \Z.\]
Let $V^0=V\setminus \cup H_{k,\xi}$ denote the complement of the union of these hyperplanes. The hyperplane $H_{k,\xi}$ is naturally cooriented: its complement is naturally divided into positive and negative half-spaces
\[(V\setminus H_{k,\xi})^+=\{v\in V \mid \xi(v)>k \}\qquad (V\setminus H_{k,\xi})^-=\{v\in V \mid \xi(v)<k \}.\]
By an {\bf alcove}, we mean a connected component of $V^0$ and we let $\Alc$ denote the set of alcoves.
For two alcoves
 $A$, $A'\in \Alc$ sharing a codimension one face which is contained
in a hyperplane $H\in \Sigma$, we will say that $A'$ is {\bf above}
$A$ and $A$ is {\bf below} $A'$ if $A'\in (V\setminus H)^+$.

\begin{definition}\label{def1}
A {\bf real variation of stability conditions} on $\mathcal{D}$ parameterized by $V^0$ and directed to $V^+$
is the data $(Z,\tau)$, where $Z$ (the central charge) is a polynomial map $Z:V\to (K^0(\mathcal{D})\otimes \R)^*$, and $\tau$ is a map from $\Alc$ to the set of bounded
$t$-structures on $\mathcal{D}$, subject to the following conditions.

\begin{enumerate}
\item
If $M$ is a nonzero object in the heart of $\tau(A)$, $A\in \Alc$, then
$\< Z(x), [M]\> >0$ for $x\in A$.\label{RV1}

\item
 Suppose $A$, $A'\in \Alc$ share a codimension one face $H$ and $A'$ is above $A$.
Let $\mathcal{C}$ be the heart of $\tau(A)$; for $n\in {\mathbb{N}}$ let $\mathcal{C}_n\subset \mathcal{C}$
 be the full
subcategory in $\mathcal{C}$ given by: $M\in \mathcal{C}_n$  if the polynomial function on $V$,
$x\mapsto \< Z(x), [M] \>$ has zero of order at least $n$ on $H$.
One can check that $\mathcal{C}_n$ is a Serre subcategory in $\mathcal{C}$, thus $\mathcal{D}_n=\{
C\in \mathcal{C}\ |\ H^i_{\tau(A)}(C)\in \mathcal{C}_n\}$ is a thick subcategory in $\mathcal{D}$.
We require that

\begin{enumerate}
\item The $t$-structure $\tau(A')$ is compatible with the filtration by thick
subcategories $\mathcal{D}_n$.\label{RV2a}

\item The functor of shift by $n$ sends the $t$-structure on $gr_n(\mathcal{D})=\mathcal{D}_n/\mathcal{D}_{n+1}$ induced
by $\tau(A)$ to that induced by $\tau(A')$. In other words, $$gr_n(\mathcal{C}')=gr_n(\mathcal{C})[n]$$ where $\mathcal{C}'$ is the heart of $\tau(A')$, $gr_n=\mathcal{C}'_n/\mathcal{C}'_{n+1}$, $\mathcal{C}'_n=\mathcal{C}'\cap \mathcal{C}_n$. \label{RV2b}

\end{enumerate}
\end{enumerate}

\end{definition}
Assuming that $D^b(\mathcal{C})\cong \mathcal{D}\cong D^b(\mathcal{C}')$, condition (2) can be rephrased as saying that the induced equivalence of derived categories is perverse with respect to the filtration $\mathcal{C}_n$ in the sense of \cite{chuangPerverseEquivalences}.

In our context, $V=\R^{\edge}\times \prod_{i\in \vertex}\R^{w_i}$ and
$\hat{\Sigma}$ is the set of functionals defined by the left-hand side of the equation \eqref{eq:circuit-hyperplane}. Note that this element depends on an oriented path, and switching the orientation will negate the functional; we'll choose a preferred sign below. Thus, elements of $V_0$ are choices of parameters avoiding the conditions  \eqref{eq:circuit-hyperplane}. Each element $\hat\beta\in V_0$ thus defines a flavor in $\Ab=\R/\Z$ and for each alcove $C\in \Delta$, the algebras defined by $\Rring_{\beta}$ for $\hat\beta\in C$ are all Morita equivalent to each other by Proposition \ref{prop:alcove-equivalence}.  For each alcove $C$, we let $ {\Rring_C}$ be the cylindrical KLRW algebra corresponding to a fixed choice of parameters in that alcove.  
\notation{${\Rring_C}$}{the cylindrical KLRW algebra corresponding to a fixed choice of parameters in an alcove $C$.  }

In order to coorient the hyperplanes, we fix a generic integral element $\chi$ in
$\Z^{\edge}\times \prod_{i\in \vertex}\Z^{w_i}$, which we can also
interpret as a cocharacter into $\No$.  As described above, elements of $\hat{\Sigma}$ correspond to oriented paths. Integrating $\chi$ (thought of as a 1-cocycle) over this path, we obtain an integer, which by genericity is never 0.  We choose signs on the elements of $\hat{\Sigma}$ so that these integers are always positive; that is, we define the positive side of each hyperplane as the side any ray parallel to $\chi$ points toward.

Now, let us define a real variation of stability conditions for each alcove $C_0$ in $\Alc$.  For a given alcove $C$, we define an equivalence \[\mathbb{B}_{C,C_0}\colon D^b(\Rring_{C}\fdmod)\to D^b(\Rring_{C_0}\fdmod)\] by the rule that for any pair of alcoves $C^{\pm}$ on the positive/negative sides of a hyperplane (i.e. $C^+$ is above $C^-$), the equivalences are related by \[\mathbb{B}_{C^-,C_0}(M)=\mathbb{B}_{C^+,C_0}( {\rif_{\beta',\beta}}\Lotimes M),\] where $\beta$ is in the interior of $C^-$ and $\beta'$ is in the interior of $C^+$.  

\begin{definition} Let $\tau$ be the map sending the alcove $C$ to the image in $D^b( {\Rring_{C_0}}\fdmod)$ of the standard $t$-structure of $D^b(\Rring_{C}\fdmod)$ under the derived equivalence $\mathbb{B}_{C,C_0}$.  In particular, it sends $C_0$ to the standard $t$-structure on $\Rring_{C_0}$-modules.
\end{definition}

Now, we need to define the central charge $Z$.  Since this depends polynomially on $V$, we need only define it on $C_0$. Recall that by Definition \ref{def:prefered}, for each $\boldsymbol{\alpha}\in \mathbb{L}$, we have an associated idempotent $e(\boldsymbol{\alpha})\in \Rring$.  
 Note that $e(\boldsymbol{\alpha})$ depends on the choice of $\beta$, since the order of ghosts depends on these.  We then define the central charge by the formula:
\begin{equation}\label{eq:central-charge}
 {Z_{\beta }(M)} =\int_{ \mathbb{L}} \dim e(\boldsymbol{\alpha}) M.
\end{equation}
\notation{$Z_{\beta}$}{The central charge function for $\beta$, which sums the dimensions of the images $ e(\boldsymbol{\alpha}) M$ weighted by the volume of the subset of $\mathbb{L}$ which gives this idempotent.}
This is a sum of the dimensions of the images of idempotents,
weighted by the volume of the corresponding alcove. 
This is a combinatorial version of the central charge function
$\mathcal{Z}^0$ defined in \cite[(4.7)]{aganagicKnotCategorification2020}, as we will
explain in more detail in Remark \ref{rem:Z-match}.

Note that the alcove $C_0$ is arbitrary, so we can define this central
charge function on $K^0(\Rring_{C}\fdmod)$ for any alcove using
the formula \eqref{eq:central-charge} on the alcove $C$. For
any integral $\nu$, the translation $C_0+\nu$ is another alcove with
$\Rring_{C_0}\cong \Rring_{C_0+\nu}$, but the
corresponding functions of the $K$-group differ by translation in
$V$.  
Of
course, we can reasonably ask how to match these functions for
different alcoves:
\begin{lemma}\label{lem:Z-match}
Using the isomorphism $[\mathbb{B}_{C,C_0}]\colon K^0( {\Rring_{C}}\fdmod)\cong K^0(\Rring_{C_0}\fdmod)$  to identify $K$-groups, the functions $ {Z_{\beta }(M)}$ match.
\end{lemma}
This is proven on page \pageref{proof-lem:Z-match}.

\begin{theorem}\label{thm:real-variation}
  The data $(Z,\tau)$ define a real variation of $t$-structures.
\end{theorem}This is proven on page
\pageref{proof-thm:real-variation}.

One structure we need to understand when studying this real variation
of $t$-structures is the structure of the categories $\mathcal{C}_n$
corresponding to a given wall.  Consider a wall of $C_0$, defined by
an equality of the form \eqref{eq:circuit-hyperplane} for some path
(possibly lazy).  Let $\hat\beta$ denote a choice of parameters in
the interior of $C_0$, and $\hat\beta'$ a generic point on the
wall.  In the case where $\hat\beta_e=0$, this wall must be of the form
$\hat\beta'_{i,k}=\hat\beta'_{j,\ell}$.

At $\hat\beta$, we have a finite number of equivalence classes of flavored sequences.  Fix one of these and consider its corresponding closure, a convex polytope $P_{\hat\beta}\subset \mathbb{L}$.  Of course, if $P_{\hat\beta}$ has non-zero volume, it is not contained in any affine subspace.  We wish to consider how this volume varies as we vary the parameter $\hat\beta_t=(1-t)\hat\beta'+t\hat\beta$ and it approaches the wall along a linear path, since this is a key ingredient in understanding real variations of stability.  We can understand the vanishing order of $\operatorname{Vol}(P_{\hat\beta_t})$ as $t\to 0$, in terms of the geometry of $P_{\hat\beta'}$.  By the affine span of $P_{\hat\beta'}$, we mean that we choose a polytope in $\prod_{i\in \vertex}\R^{v_i}$ mapping isomorphically to $P_{\hat\beta'}$ and consider the image of its affine span in $\mathbb{L}$.  

\begin{lemma}\label{lem:vanishing-order}
	The vanishing order of $\operatorname{Vol}(P_{\hat\beta_t})$ at $t=0$ is equal to the codimension of the affine span of the polytope $P_{\hat\beta'}$.  
\end{lemma}

This codimension arises from the fact that on $P_{\hat\beta'}$, we can have pairs of a corporeal and a ghostly or red strand which have the same longitude at each point.  We call such a pair {\bf locked}.  This gives a linear relation between the longitude of the two corporeals.  Taking transitive closure, we obtain a {\bf locking equivalence relation}.

\begin{lemma}
All the affine linear relations on $P_{\hat\beta'}$ come from the locking equivalence relation.  The dimension of the affine span of $P_{\hat\beta'}$ is the number of equivalence classes of locking.
\end{lemma}
\begin{proof}
If there is any linear relation which does not come from locking, it must be some linear combination of the longitude of corporeals not locked to a red strand; since the longitude of a corporeal which is locked to a red strand is fixed already, we can eliminate it from our relation WLOG.

For each element of $P_{\hat\beta}$ we let $\epsilon$ be the minimal positive distance between a corporeal $p\in \Sc$ and a red strand or ghost $g\in \Sgr$.  If we have a locking equivalence class that does not contain a red, then we can add any real number in $[-\nicefrac{\epsilon}{2},\nicefrac{\epsilon}{2}]$ to the longitude of each corporeal in this equivalence class without changing the equivalence class.  This shows that the coefficients in our relation of the longitudes in this class must sum to 0.  This means we can eliminate them using the relations coming from the locking.  Since this shows that all locking equivalence classes can be eliminated, this shows that there are no other linear relations.  
\end{proof}
For a given $\Bi,\Ba$, this defines a statistic $d(\Bi,\Ba)$ given by the codimension of the affine span of the corresponding degenerate polytope at $\hat\beta'$.   We call this the {\bf codimension} of the equivalence class (with respect to the wall $H$). 
 Let
$ {e_m}$ be the sum of all idempotents $e(\Bi,\Ba)$ with $d(\Bi,\Ba)> m$.
\notation{$d(\Bi,\Ba)$}{The codimension of the polytope defined by the limit of an equivalence class as the parameter $\hat\beta$ approaches a generic point of a wall.}
Let $H$ be the wall in which $\hat\beta'$ is generic.
\begin{corollary}\label{cor:wall-crush}
	The module $M$ lies in the category $\mathcal{C}_m$ for the wall $H$ if and only if $e_mM=0$.
The subcategory $\mathcal{D}_m$ is the subcategory of complexes $M^{\bullet}$ such that the complex of vector spaces $e_mM^{\bullet}$ is exact.  
\end{corollary}

\subsection{Comparison to the unflavored case}

In \cite[\S 7.4]{kamnitzerLieAlgebra2024}, the author and collaborators discussed the relationship between flavored KLRW algebras and weighted KLR algebras, as defined in \cite{WebwKLR}.  This is most conveniently dealt with via a bimodule between these two algebras, which is a Morita equivalence in many circumstances.  We can also define a cylindrical version of these bimodules.  For the purposes of this paper, it will be most relevant to compare to a cylindrical version of the unflavored KLRW algebra, since we will be interested primarily in the case of Dynkin diagrams.  Consider parameters $\beta_*\in \R/\Z$ as before, with $\beta_e=0$ for  all $e\in \edge$, but with $\beta_{i,k}$  taking on arbitrary values.  

\notation{$\uf$}{The unflavored cylindrical KLR algebra (Definition \ref{def:unflavored-diagram}.}
\begin{definition}\label{def:unflavored-diagram}
	An {\bf unflavored cylindrical KLRW diagram} is a KLRW diagram where each corporeal strand coincides with all of its ghosts, and each red strand is vertical at  $x=\beta_{i,k}$ for $k\in [1,v_i]$. This, of course, violates the genericity conditions from the original definition (Definition \ref{def:cylindrical-diagram}), so we assume that they hold once we think of the corporeal strand and its ghosts as a single strand; we call these strands {\bf black}. We label strands of the diagram with labels from $\vertex$, but not with longitudes. One example of such a diagram is below:
	      \begin{equation*} 
       \tikz[very thick,xscale=2]{
          \draw[fringe] (-1,-1)-- (-1,1);
          \draw[fringe] (1,1)-- (1,-1);
          \draw[wei] (-.8,-1)--node[below, at start ]{$i$} (-.8,1);
          \draw[wei] (.4 ,-1)--node[below, at start ]{$j$} (.4,1);
\draw (-1,.2) to[out=30,in=-90] node[above, at end]{$i$} (.2,1);
           \draw (.6 ,-1) to[out=90,in=-150] node[below, at start ]{$i$}(1,.2);
           \draw (-1,-.2) to[out=-30,in=90]node[below, at end ]{$k$} (-.5,-1);
           \draw (-.2 ,1) to[out=-90,in=150] node[midway,circle,fill=black,inner sep=2pt]{} node[above, at start ]{$k$}(1,-.2);
           \draw (-.2,-1) to[out=90,in=-90] node[below, at start ]{$i$} node[above, at end]{$i$} (-.5,1);
           }
        \end{equation*}
	
	The unflavored cylindrical KLRW algebra $\uf$ is the algebra over $\K$ spanned by unflavored cylindrical KLRW diagrams, modulo the relations induced by (\ref{c-first-QH}--\ref{w-black-bigon2}). We can apply these by making a small deformation pushing the ghosts off the corporeal strands to become generic again, applying the relations and then undoing the deformation.     For completeness, let us note that if we have a single edge $j\to i$, then applying this approach to the relations (\ref{c-nilHecke-1},\ref{w-black-bigon1},\ref{w-black-bigon2}), we have that
   \begin{equation}\label{black-bigon1}
      \begin{tikzpicture}[very thick,scale=.65,baseline]
\end{tikzpicture}\quad   \begin{tikzpicture}[very thick,scale=.65,baseline]
      \draw(-2.8,0) +(0,-1) .. controls (-1.2,0) ..  +(0,1)
      node[below,at start]{$i$}; \draw (-1.2,0) +(0,-1) .. controls
(-2.8,0) ..  +(0,1) node[below,at start]{$j$};
\end{tikzpicture}
=
\begin{tikzpicture}[very thick,scale=.65,baseline]
      \draw (-.8,0) +(0,-1)-- +(0,1)
      node[below,at start]{$i$}; \draw (.2,0) +(0,-1) --node[midway,fill=black, inner sep=2pt, circle]{}+(0,1) node[below,at start]{$j$};
\end{tikzpicture}-\begin{tikzpicture}[very thick,scale=.65,baseline]
      \draw(-.8,0) +(0,-1) -- node[midway,fill=black, inner sep=2pt, circle]{} +(0,1)
      node[below,at start]{$i$}; \draw (.2,0) +(0,-1) --+(0,1) node[below,at start]{$j$};
\end{tikzpicture}    \qquad \qquad 
     \begin{tikzpicture}[very thick,scale=.65,baseline]
      \draw (-1.2,0) +(0,-1) .. controls
(-2.8,0) ..  +(0,1) 
      node[below,at start]{$i$}; \draw (-2.8,0) +(0,-1) .. controls (-1.2,0) ..  +(0,1) node[below,at start]{$j$};
\end{tikzpicture}
=
  \begin{tikzpicture}[very thick,scale=.65,baseline,xscale=.9]
    \draw(.2,0) +(0,-1) -- +(0,1) node[below,at
    start]{$i$}; \draw (-.8,0)+(0,-1) --node[midway,fill=black, inner
    sep=2pt, circle]{}+(0,1) node[below,at start]{$j$};
  \end{tikzpicture}-\begin{tikzpicture}[very
    thick,scale=.65,baseline,xscale=.9] \draw(.2,0) +(0,-1) --
    node[midway,fill=black, inner sep=2pt, circle]{} +(0,1)
    node[below,at
    start]{$i$};  \draw (-.8,0) +(0,-1) --+(0,1) node[below,at
    start]{$j$};
  \end{tikzpicture}
\end{equation}
Similarly in (\ref{w-triple-point},\ref{w-triple-point2}), we just move the
corporeal strand with label $i$ on top of the ghost with label $j$.  Note
that while the relations (\ref{c-first-QH}--\ref{w-triple-point2}) did
not depend at all on the number of edges between vertices because
different edges will have separate ghosts, the relations above need to
be modified if there are multiple edges $i\to j$ (following the
standard recipe of \cite{KLII}).
\end{definition}

Since there are no longitudes and we have merged ghosts into corporeal strands, the top and bottom of an unflavored KLRW diagram will just be an order on the set $\Sc\Sr$, that is, a cyclic (or if one prefers, periodic) word in the union of a black and a red copy of the set $\vertex$.

Up to isomorphism, this algebra only depends on the cyclic ordering of the values $\beta_{i,k}$, or put differently, on the ordered list $\Bj=(j_1,\dots, j_{\ell})$ of labels on the red strands, starting at $x=0$; of course, cyclic permutation of these labels will also give the same algebra, but we have no convenient way of denoting of this cyclic words without breaking it to a linear one.  Thus, we can without ambiguity write $\uf^{\Bj}$ to denote the algebra where these are the labels on red strands (for any number of black strands).

We can relate this algebra to the flavored KLRW algebra via a natural bimodule.  As in Section \ref{sec:wall-cross}, we choose lifts $\hat\beta_e,\hat\beta_{i,k}\in \R$.  

\begin{definition}
	A {\bf half-flavored cylindrical KLRW diagram} is a flavored KLRW diagram where we replace condition (3) by (3'') below, and for each corporeal strand, there is a value $\epsilon \in [0,1)$ such that for $y\leq \epsilon$, the strand and its ghosts coincide as in an unflavored diagram, and for $y>\epsilon$, they satisfy the genericity conditions of a flavored KLRW diagram.  Finally, we constrain the winding of ghosts by the condition:
\begin{itemize}
\item[(3'')] Red strands do not cross.   We count the point where we split apart a ghost and a corporeal that were joined at $y=\epsilon$ as:
\begin{itemize}
\item[$\bullet$]  $1/2$ of an intersection point if the ghost is right of the corporeal for $y\in (\epsilon,\epsilon')$ for some $\epsilon'>\epsilon$, and 
\item[$\bullet$] $-1/2$ of an intersection point if it is to the left.  
\end{itemize}
With this definition, we assume that a ghost for the edge $e$ has $\lceil \hat\beta_e\rceil-1/2$ intersections with its corporeal. 

 As long as $\hat\beta_e\notin \Z$, this will hold if we take $\epsilon=0$, and let the $x$-coordinate of the ghost minus that of the corporeal be $\hat\beta_e$.  We take $\lceil \hat\beta_e\rceil-1/2$ instead of $\lfloor \hat\beta_e\rfloor+1/2$ to account for the ``tiebreak'' that ghosts and corporeals with the same longitude must be ordered with the ghosts to the left.  
	\end{itemize}
	
The span of half-flavored cylindrical KLRW diagrams modulo the relations (\ref{c-first-QH}--\ref{w-black-bigon2}) gives a bimodule $\ubf$ with a left action of the unflavored algebra and a right action of the flavored algebra.  
\end{definition}

Assume for the remainder of the subsection that $\Gamma$ is a tree, that we have fixed an arbitrary order of $\vertex$, and chosen an orientation so that if $i\to j$, then $i>j$.  As noted in \cite[Cor. 2.16]{WebwKLR}, in this case, weighted KLR algebras are Morita equivalent to the original KLR algebra.  More generally, the algebra will not be changed by replacing $\hat{\beta}_*$ with the cohomologous value:
\begin{equation}\label{eq:coboundary}
    \hat\beta'_e=\hat\beta_e-\eta_{h(e)}+\eta_{t(e)}\qquad \hat\beta'_{i,k}=\hat\beta_{i,k}+\eta_{i}
\end{equation}
for a 0-cochain $\eta\colon \vertex\to \R$.  Since $\Gamma$ is a tree, we can choose $\eta$ so that $\beta'_e=0$ for all $e\in \edge$.  A similar Morita equivalence holds here:
\begin{lemma}\label{lem:shift-iso}
If $\Gamma$ is a tree and $\Ab=\R/\Z$, then the ring $\Rring_{\beta}$ is Morita equivalent to the ring $\uf_{\beta'}$ for the parameters $\beta'_{i,k}$.
\end{lemma}
\begin{proof}
{\bf Reduction to the case $\hat\beta_e=0$}:  If $\hat\beta$ and $\hat\beta'$ are related as in \eqref{eq:coboundary}, then the statistic \eqref{eq:circuit-hyperplane} never changes, so the bimodule $B_{\beta,\beta'}$ induces a Morita equivalence by  Proposition \eqref{prop:alcove-equivalence}.  Thus, without loss of generality, we can assume that $\beta_e=\beta_e'=0$ for all $e\in \edge$.  
	  
{\bf Morita equivalence with $\uf$}:
Now assuming that $\beta_e=0$ for all $e\in \edge$, we wish to show that we have a Morita equivalence 
via the bimodules $\fbu$ and $\ubf$. Consider the maps of composition of diagrams  \begin{equation}\label{eq:Morita-FU}
		\fbu\otimes \ubf\to  \Rring_{\beta}\qquad \ubf\otimes \fbu\to \uf_{\beta'}.
	\end{equation} These define a Morita context, so Morita equivalence will follow if we show both maps are surjective (again by \cite[II.3.4]{BassK}).

		  Several times in this proof, we'll discuss the relative positioning of strands in a small interval $I\subset \R/\Z$.  We'll say $y$ is left of $z$ (and $z$ is right of $y$) if $(x,y,z)\in C$  for all $x\notin I$, so these correspond to the usual notions of left and right if we identify $I$ with a subset of the number line, or of the $x$-axis in $\R^2$.

	Choose a real number $0<\epsilon_i\ll 1$ for each $i\in \vertex$ such that if $i<j$ in our fixed order, we have $2\epsilon_i<\epsilon_j$; we may adjust these to be smaller later.
	  Given a flavored sequence $(\Bi, \Ba, <)$ for the flavor $\beta$, we define an idempotent for $\uf$ by placing the strand for $k\in \Sc$ at $\mathfrak{a}'(k)=a_k+\epsilon_{i_k}+\epsilon_{i_k}^2k\in \R/\Z$.    This can be extended to a map $\mathfrak{a}'\colon \Sall\to \R/\Z$ by sending the ghosts $(m,e)$ to $\mathfrak{a}'(m)$, and red strands to the corresponding flavor $\beta_{i,k}$. 
	  
	   Since $\epsilon_{i}$ is very small, this is roughly placing the strands at their longitudes, but using two tiebreaks if strands have the same longitude in this case: 
	   \begin{enumerate}
	   	\item first, we order strands using the order on vertices, 
	   	\item then we add a small separation between the strands with the same labels. 
	   \end{enumerate}
	The latter separation is arbitrary since there is no difference between strands with the same labels; the former, on the other hand, is quite important (and in particular, would give us trouble if we chose an orientation of $\Gamma$ with oriented cycles).  If we choose all $\epsilon_i$'s sufficiently small, the idempotents in $\uf_{\beta'}$ corresponding to different choices will all be equivalent. Note that the upper bound on $\epsilon_i$'s will depend on the flavored sequence, but this is not a problem, since we can show the surjectivity of \eqref{eq:Morita-FU} one idempotent at a time.
	  
	  We can define diagrams $D_1,D_2$ in  $\fbu$ and $\ubf$ that interpolate between these idempotents, joining the terminals corresponding to the same element of $\Sall$ in a way that minimizes crossings; more precisely, we avoid all crossings between pairs of corporeal strands which either have the same label or have different longitudes, and we avoid crossings between corporeal and red strands.   This is possible since if $(x,y,z)\in C \subset \Sall^3$, and $\mathfrak{a}(x),\mathfrak{a}(y)$ and $\mathfrak{a}(z)$ are all distinct, then $(\mathfrak{a}(x),\mathfrak{a}(y),\mathfrak{a}(z))$ is cyclically ordered if and only if $(\mathfrak{a}'(x),\mathfrak{a}'(y),\mathfrak{a}'(z))$ is also cyclically ordered. That is, the diagrams $D_1,D_2$ should only reorder the strands with the same longitude.  Of course, it is always possible to do this without switching strands with the same label.  Finally, we have no corporeal/red crossings, since if $r\in \Sr$ and $m\in \Sc$ have the same longitude $\mathfrak{a}(r)=\mathfrak{a}(m)=\beta_{i,k}$ then $r$ is left of $m$, that is $(r,m,x)\in C$ for all $\mathfrak{a}(x)\neq \beta_{i,k}$, and we have 
	  \[(\mathfrak{a}'(r),\mathfrak{a}'(m),\mathfrak{a}'(x))=(\beta_{i,k}, \beta_{i,k}+\epsilon_{i_m}+\epsilon_{i_m}^2m,\mathfrak{a}'(x))\]
is also cyclically ordered for $\epsilon_{i_m}$ sufficiently small.  
  
	  These diagrams are mirror images under reflection through a vertical line, so when we consider the composition $D_1D_2$, it deviates from being straight vertical by certain bigons, created by the crossings in $D_1$ and their mirror images in $D_2$.  We want to simplify this diagram, and obtain one which is straight vertical with no dots.  This is possible because:
	  \begin{enumerate}
	  	\item corporeal strands with the same labels never cross, by construction, and similarly with red and corporeal strands.
	  	\item there is no crossing 
	  	of a corporeal of label $i$ with a ghost strands corresponding to edges $e$ with $t(e)=i$.  If we consider such a pair of strands, they will not cross unless they have the same longitude $\mathfrak{a}(m)=\mathfrak{a}(k,e)=\mathfrak{a}$.  In this case, the ghost strand must be to the left in the flavored sequence.  On the other hand, in the idempotent in $\uf_{\beta}$, the strand corresponding to $k$ will have $\mathfrak{a}'(k)=\mathfrak{a}+\epsilon_{i_k}+\epsilon_{i_k}^2k $ and similarly for the strand corresponding to $m$.  By assumption, we have an edge $i_m \to i_k$, which implies $2\epsilon_k<\epsilon_m$ in our fixed order.  Thus, for both these parameters sufficiently small, we have 
	  	\[\mathfrak{a}'(k)=\mathfrak{a}+\epsilon_{i_k}(1+\epsilon_{i_k}k)<\mathfrak{a}+2\epsilon_{i_k}<  \mathfrak{a}+\epsilon_{i_m}+\epsilon_{i_m}^2m=\mathfrak{a}'(m) \] since we may assume $\epsilon_{i_k} <1/k$.  This shows that the corporeal strand for $k$, and thus its ghost $(k,e)$ is left of that for $m$, so these are in the same order.
	  \end{enumerate}
	  This shows that all the bigons created involve corporeal strands with different labels, or strands and ghosts where the labels don't match.  That is, they can be resolved by (\ref{c-psi2}) or the first cases of (\ref{w-black-bigon1}--\ref{w-black-bigon2}).  This shows that the composition $D_1D_2$ gives the idempotent for our starting flavored sequence and so the first map of \eqref{eq:Morita-FU} is surjective.  
	  
	  On the other hand, given an idempotent in $\uf$, its different corporeal strands will have different $x$-values which we think of as an injective map $\mathfrak{a}''\colon \Sc\to \R/
	  Z$, and we can consider the flavored sequence such that $\mathfrak{a}(m)= \mathfrak{a}'' (m) -\epsilon_{i_m}-\epsilon_{i_m}^2m $.  For $\epsilon_i$'s sufficiently small, these will all be distinct and give a flavored sequence such that the bottom of the diagram $D_1$ and top of the diagram $D_2$ will be the idempotent in $\uf$ that we started with.  The argument above shows that this idempotent can be written as $D_2D_1$, showing the second map of \eqref{eq:Morita-FU} is surjective. 
	  \end{proof}

\subsubsection{Relation to planar KLRW algebras}
\label{sec:relation-planar-klrw}

As suggested by the name, unflavored cylindrical KLRW algebras are related to the usual
KLRW algebras $\tilde{T}^{\bla}$ as defined in \cite[Def. 4.5]{Webmerged}.
These primarily differ in that the diagrams are drawn in $\R\times [0,1]$, not $S^1\times [0,1]$.
When there is a danger of confusion, we will use the 
adjective ``planar'' to distinguish these from cylindrical KLRW algebras.  Here ``planar'' does not
refer to the structure of the corresponding quiver (which is
arbitrary) but the surface on which the diagrams are drawn.

Another point worth clarifying: there are (planar) flavored KLRW algebras defined in \cite{kamnitzerLieAlgebra2024}, but the connection of these to the cylindrical algebras is more complicated, and for our purposes in this paper, it's enough to consider the unflavored case.  Since we are working in the unflavored case, we'll assume that we have chosen $\hat\beta_*$ with $ {\hat \beta_e}=0$ for all $e\in \edge$.   Up to isomorphism, only the relative order of red strands will matter, so we fix the labels on red strands in order from left to be $\Bj=(j_1,\dots, j_{\ell})$. 

\begin{definition}\label{def:uf-pl}
An (unflavored planar) KLRW diagram (also called Stendhal diagram) is a diagram satisfying the same local rules as the cylindrical unflavored diagrams of Definition \ref{def:unflavored-diagram} but drawn in $\R\times [0,1]$
instead of $S^1\times [0,1]$; the position of the red strands is given by the $x$-values $\hat \beta_{i,k}$.

	The (unflavored planar) KLRW algebra $\tilde{T}^{\Bj}$  is the algebra spanned over $\K$ by these diagrams modulo the local relations
(\ref{c-first-QH}--\ref{w-triple-point2}).  
\end{definition}
\notation{$\tilde{T}^{\Bj}$}{The unflavored planar KLRW algebra for a list $\Bj=(j_1,\dots, j_{\ell})$ of labels on red strands (Definition \ref{def:uf-pl} and \cite[Def. 4.5]{Webmerged}).
}
Unlike the case of the flavored KLRW algebra, we require no additional information about the top and bottom of the diagram other than the order with which red and black strands meet it.  For purposes of notation, it will be useful to think of this as a word $\Bi$ in the union of a black copy and a red copy of $\vertex$, and we will use $e(\Bi)$ to represent the corresponding idempotent where we have vertical red and black strands in the indicated order.

We discuss these algebras
largely with the aim of transferring certain calculations done for
planar KLRW algebras to cylindrical ones.

To describe this connection, 
note that $\tilde{T}^{\emptyset}$ is the usual KLR algebra defined in
\cite{KLII}.  The category of $\tilde{T}^{\emptyset}$-modules has a
monoidal structure induced by the induction functor defined in
\cite[\S 2.6]{KLI}: put simply, horizontal composition induces an
algebra map $\tilde{T}^{\emptyset}\otimes \tilde{T}^{\emptyset}\to
\tilde{T}^{\emptyset}$, and induction is pushforward by this map.

Similarly, horizontal composition also gives a map
$\tilde{T}^{\emptyset}\otimes \tilde{T}^{\Bj}
\otimes\tilde{T}^{\emptyset} \to  \tilde{T}^{\Bj}$, and pushforward
gives a functor sending a triple $(K,M,N)$ consisting of 
$\tilde{T}^{\emptyset}$-modules $K,N$, and a $\tilde{T}^{\Bj}$-module $M$ 
to a $\tilde{T}^{\Bj}$-module $K\circ M\circ N$.  
\begin{lemma}
    The functor $(K,M,N)\mapsto K\circ M\circ N$ is exact in all three inputs.  
\end{lemma}
\begin{proof}
    This requires proving that for each $m,m',m''\in \Z_{\geq}$ $\tilde{T}^{\Bj}_{m+m'+m''}e_{m,m',m''}$ is projective as a right module over $\tilde{T}^{\emptyset}_m\otimes \tilde{T}^{\Bj}_{m'}
\otimes\tilde{T}^{\emptyset}_{m''}$, where the subscript denotes the total number of black strands and $e_{m,m',m''}$ is the image of the identity in $\tilde{T}^{\emptyset}_m\otimes \tilde{T}^{\Bj}_{m'}
\otimes\tilde{T}^{\emptyset}_{m''}$.  This follows from precisely the same argument as \cite[Prop. 2.16]{KLI}: we can factor diagrams into a product where 
\begin{itemize}
    \item below $y=1/2$, we have a diagram that comes from $\tilde{T}^{\emptyset}_m\otimes \tilde{T}^{\Bj}_{m'}
\otimes\tilde{T}^{\emptyset}_{m''}$: it only crosses strands within the leftmost $m$ strands, the rightmost $m''$, and the red strands strands and middle $m'$.
\item above $y=1/2$, we have a diagram that shuffles the strands which are in the three groups at $y=1/2$ without adding any crossings between the strands in the same group or any dots.
\end{itemize}
Fixing the diagram above $y=1/2$ and the sequence of labels at $y=1/2$ gives a projective submodule over $\tilde{T}^{\emptyset}_m\otimes \tilde{T}^{\Bj}_{m'}
\otimes\tilde{T}^{\emptyset}_{m''}$, and the module as a whole is the direct sum of these. 
\end{proof}

Since this functor is exact, we can extend $M\mapsto K\circ M \circ N$ to a $t$-exact functor of bounded (bounded above/below) derived categories for any $K$ and $N$.

We can think of the cylindrical KLRW algebra as an {\it affinization}
of the $\tilde{T}^{\emptyset}\mmod$ bimodule structure on
$\tilde{T}^{\Bj}\mmod$.  We can say this a little more precisely when we
think about the {\bf planar KLRW category} $\tilde{\mathcal{T}}$, where the objects are
words in black and red copies of $\vertex$, and morphisms are unflavored planar KLRW diagrams
joining these words, modulo the relations (\ref{c-first-QH}--\ref{w-triple-point2}); for notational
purposes, let $\tU^-$ be the planar KLRW category with no red lines.
Similarly, there is
a {\bf cylindrical (unflavored) KLRW category} with objects given by cyclic words
(or if you prefer, periodic words) and morphisms by unflavored cylindrical KLRW
diagrams modulo the same local relations.

Consider the category $\mathcal{Q}$ obtained by  adjoining to $\tilde{\mathcal{T}}$ an isomorphism $\xi_{\Bi,\Bi'}\colon \Bi\circ \Bi'\to \Bi'\circ\Bi$ for $\Bi'$ an object in $\tU^-$ (a word only in the black copy of $\vertex$), and $\Bi'$ an object in $\tilde{\mathcal{T}}$ (a  word in the red and black copies of $\vertex$), and impose the additional relations:
  \begin{align}\label{eq:affine-1}
    \xi_{\Bi,\Bi'\circ\Bi''}&=\xi_{\Bi''\circ\Bi,\Bi'}\xi_{\Bi\circ\Bi',\Bi''}\\
    \xi_{\Bm,\Bm'}\circ (f\otimes g) &=   (g\otimes f) \circ  \xi_{\Bi,\Bi'}           \label{eq:affine-2}            
  \end{align}
  for $\Bi,\Bm$ words with red strands labeled by $\bla$, $\Bi',\Bi'',\Bm'$ words only in the black strands and $f\colon \Bi\to \Bm$ and $g\colon \Bi'\to\Bm'$ arbitrary morphisms.  Note the similarity to the work of Mousaaid and Savage \cite{mousaaidAffinizationMonoidal2021} on affinization of monoidal categories; this not quite a special case of their work, since we are using a bimodule category, and they only consider the action of a monoidal category on itself on the left and right.  
\begin{proposition}
  The category $\mathcal{Q}$ defined above is equivalent to the cylindrical KLRW category via the functor sending $\Bi$ to the same word considered cyclically, sending any morphism in the planar KLRW category to the morphism drawn in $S^1\times [0,1]$ by embedding $\R$ as $S^1\setminus\{*\}$, and  $\xi_{\Bi,\Bi'}$ to the diagram with bottom given by the concatenation $\Bi\circ \Bi'$, which moves the strands in $\Bi'$ around the back of the cylinder in the positive direction.  That is: 
  \[\xi_{\Bi,\Bi'}\mapsto 
    \begin{tikzpicture}[anchorbase,very thick, scale=2]
        \draw[\rectcolor] (-0.4,-0.5) -- (-0.4,0.5);
        \draw[\rectcolor] (0.4,-0.5) -- (0.4,0.5);
        \draw[\rectcolor] (0,0.5) ellipse (0.4 and 0.15);
        \draw[\rectcolor] (-0.4,-0.5) arc (180:360:0.4 and 0.15);
        \draw[dashed,\rectcolor] (-0.4,-0.5) arc (180:0:0.4 and 0.15);
        \draw (0,-0.65) node[anchor=north] {$\Bi$} \braidup (0,0.35);
        \draw (0.2,-0.63) node[anchor=north] {$\Bi'$} to[out=up,in=down] (0.4,-0.2);
        \draw[dashed] (0.4,-0.2) to[out=up,in=down] (-0.4,0.2);
        \draw (-0.4,0.2) to[out=up,in=down] (-0.2,0.37);
    \end{tikzpicture}\]
\end{proposition}
\begin{proof}
{\bf The functor is well-defined}: To show this, we need only check that the
relations  (\ref{eq:affine-1}--\ref{eq:affine-2}) hold, which is
an easy geometric verification by the relations shown below:
\begin{equation*} 
    \begin{tikzpicture}[anchorbase,scale=2]
        \identify{-0.7}{-0.5}{0.7}{0.5};
        \draw (-0.3,-0.5) node[anchor=north] {$\Bi$} -- (0.3,0.5) node[anchor=south] {$\Bi$};
        \draw (0.3,-0.5) node[anchor=north] {$\Bi' \circ \Bi''$} to[out=up,in=200] (0.7,0);
        \draw (-0.7,0) to[out=20,in=down] (-0.3,0.5) node[anchor=south] {$\Bi' \circ \Bi''$};
    \end{tikzpicture}
    =
    \begin{tikzpicture}[anchorbase,scale=2]
        \identify{-0.8}{-0.5}{0.8}{0.5};
        \draw (-0.3,-0.5) node[anchor=north] {$\Bi$} -- (0.3,0.5) node[anchor=south] {$\Bi$};
        \draw (0.45,-0.5) node[anchor=north] {$ \Bi''$} to[out=up,in=200] (0.8,-0.15);
        \draw (0.15,-0.5) node[anchor=north] {$ \Bi'$} to[out=up,in=200] (0.8,0.15);
        \draw (-0.8,-0.15) to[out=20,in=down] (-0.15,0.5) node[anchor=south] {$ \Bi''$};
        \draw (-0.8,0.15) to[out=20,in=down] (-0.45,0.5) node[anchor=south] {$ \Bi'$};
    \end{tikzpicture}
    \ ,\qquad
    \begin{tikzpicture}[anchorbase,scale=2]
        \identify{-0.6}{-0.5}{0.8}{0.5};
        \draw (-0.3,-0.5) node[anchor=north] {$\Bi$} -- (0.3,0.5) node[anchor=south] {$\Bm$};
        \draw (0.3,-0.5) node[anchor=north] {$\Bi'$} to (0.3,-0.2) to[out=up,in=200] (0.8,0.2);
        \draw (-0.6,0.2) to[out=20,in=down] (-0.3,0.5) node[anchor=south] {$\Bm'$};
        \coupon{-0.12,-0.2}{g};
        \coupon{0.3,-0.2}{f};
    \end{tikzpicture}
    =
    \begin{tikzpicture}[anchorbase,scale=2]
        \identify{-0.8}{-0.5}{0.6}{0.5};
        \draw (-0.3,-0.5) node[anchor=north] {$\Bi$} -- (0.3,0.5) node[anchor=south] {$\Bm$};
        \draw (0.3,-0.5) node[anchor=north] {$\Bi'$} to[out=up,in=200] (0.6,-0.2);
        \draw (-0.8,-0.2) to[out=20,in=down] (-0.3,0.2) -- (-0.3,0.5) node[anchor=south] {$\Bm'$};
        \filldraw[draw=black,fill=white] (-0.3,0.2) circle (0.15);
        \coupon{-0.3,0.2}{f};
        \coupon{0.12,0.2}{g};
    \end{tikzpicture}
\end{equation*}

{\bf The definition of the inverse}: To show full faithfulness, we define an inverse functor.  Given a cylindrical KLRW diagram that is generic with respect to the seam $x=0$ (it has no triple points or tangencies), we can define its inverse as follows:
\begin{enumerate}
	\item Let $h_1<\dots<h_p$ be the $y$-values where the black strands intersect the seam and let $k_1,\dots, k_p$ be the labels on the strands intersecting the seam at these points.  We further assume that no crossing or dot happens at the height $h_r$ (which will be the case generically).
	\item Let $\mathring{D}_r$ be the portion of the diagram with $y$ values in $(h_r+\epsilon,h_{r+1}-\epsilon)$, and let $D_r$ be the unrolled version of this diagram.  
	\item We use this to factor our diagram into diagrams that never cross the seam (which are the images of planar diagrams) and diagrams which are of the form $\xi_{*,k}^{\pm 1}$ for $*$ an arbitrary word, where the sign depends on whether the strand passes the seam in a positive direction or a negative one.  The image of this diagram under the inverse functor is the composition $D_p\xi_{*,k_p}^{\pm 1}D_{p-1}\cdots D_1\xi_{*,k_1} ^{\pm 1} D_0$.  
\end{enumerate}
{\bf The inverse is well-defined}: First, let us check that the resulting image is invariant under isotopy of diagrams.  Of course, any isotopy that keeps this diagram suitably generic will just isotope the diagrams $D_i$, leaving the structure of the inverse invariant.  The other situations we have to account for are:
\begin{enumerate}
	\item a crossing or dot isotopes from above one of the heights $h_r$ to below it: this follows from \eqref{eq:affine-1}.
	\item a strand isotopes through a tangency with the seam, creating or destroying a bigon: this is that the two legs of the bigon will be $\xi_{*,k}^{\pm 1}$ and thus cancel.
	\item a crossing isotopes through the seam: If the crossing is moving top to bottom, this is \eqref{eq:affine-1} and \eqref{eq:affine-2}.  The sideways isotopy can be shown to be equivalent to the top-to-bottom isotopy by making two bigons with the seam and using $\xi$ as shown below.
\[\tikz[very thick,baseline]{\draw (-1,-1) .. controls (1,0) .. (1,1);\draw (1,-1) .. controls (1,0) .. (-1,1); \draw[dashed](0,-1) --(0,1);}=\tikz[very thick,baseline]{\draw (-1,-1) -- (1,1);\draw (1,-1)  to[in=-90,out=135] (-.2,-.5) to[in=-90,out=90] (.4,0) to [in=-45,out=90] (-1,1); \draw[dashed](0,-1) --(0,1);}=\tikz[very thick,baseline]{\draw (-1,-1) -- (1,1);\draw (1,-1)  to[in=-90,out=135] (-.4,0) to[in=-90,out=90]  (.2,.5) to [in=-45,out=90] (-1,1); \draw[dashed](0,-1) --(0,1);}=\tikz[very thick,baseline]{\draw (-1,-1) .. controls (-1,0) .. (1,1);\draw (1,-1) .. controls (-1,0) .. (-1,1); \draw[dashed](0,-1) --(0,1);}\]
\end{enumerate}
Having shown invariance under isotopies, we now just have to check that relations are sent to relations, which is clear since all definitions are local.  This shows that we have a functor which is inverse up to isomorphism, and so the functor is an equivalence.  
\end{proof}
This equivalence has a manifestation on the level of modules: given a
$\tilde{T}^{\Bj}$-module $M$, we can consider the tensor product
$\Rring^{\Bj}\otimes_{\tilde{T}^{\bla}}-$.  This is the
pushforward by the inclusion of the planar KLRW category into the
cylindrical.
\begin{lemma}\label{lem:R-exact}
  The functor $ {\Rring^{\Bj}}\otimes_{\tilde{T}^{\bla}}-$ is
  exact.  
\end{lemma}
This is proven on page \pageref{proof-lem:R-exact}. This perspective is useful in that it shows that a
$\tilde{T}^{\Bj'}\operatorname{-}\tilde{T}^{\Bj}$-bimodule
$\mathfrak{B}$  compatible with the
bimodule structure on these categories induces a bimodule between
cylindrical KLRW categories.  For our purposes, it will be easier to
say this in terms of functors between derived categories.  We say that
the functor $\mathbb{B}=\mathfrak{B}\Lotimes-\colon D^b(\tilde T^{\Bj}\mmod)\to D^b(\tilde T^{\Bj'}\mmod)$ is {\bf strongly equivariant} if it
commutes with the action of $\tU^-$ on the left and the right, that is, the functors $ {\tU}^-\times D^b(\tilde T^{\Bj})\times \tU^-\to
D^b(\tilde T^{\Bj'})$ defined by
\[K\circ \mathbb{B} M \circ N \mapsfrom (K,M,N) \mapsto
  \mathbb{B}(K\circ M \circ N)\] are isomorphic.
  \notation{$\mathring{\mathbb{B}}$}{The functor between modules over cylindrical KLRW algebra induced by a strongly equivariant functor between planar KLRW algebras.}
\begin{lemma}\label{lem:affinize-commute}
  If $\mathbb{B}$ is strongly equivariant for the left and right actions of $\tU^-$, then there is an induced functor
  \[\mathring{\mathbb{B}}\colon D^b(\Rring^{\Bj}\mmod)\to
    D^b(\Rring^{\Bj'}\mmod)\]
  which is compatible with composition: If $\mathbb{B}=\mathbb{B}_1
  \mathbb{B}_2$, then $\mathring{\mathbb{B}}=\mathring{\mathbb{B}}_1
  \mathring{\mathbb{B}}_2$.
\end{lemma}
Of course, we have
\[\mathring{\mathbb{B}}(M)=\mathring{\mathfrak{B}}\otimes_{\Rring^{\Bj}}M\qquad \text{where}\qquad\mathring{\mathfrak{B}}=\mathring{\mathbb{B}}(\Rring^{\Bj}).\]
\begin{proof}
  It is enough to define the functor $\mathring{\mathbb{B}}$ on modules of the form $\Rring^{\Bj}e(\Bi)$ for $\Bi$ a cyclic word.  This is just
$\Rring^{\Bj'}\otimes_{\tilde{T}^{\Bj'}}\mathfrak{B}e(\tilde{\Bi})$
  for $\tilde{\Bi}$ any lift of the cyclic word to a usual planar word.  This is well-defined and functorial due to the strong equivariance of $  \mathbb{B}$.

  If $M$ is a $\tilde{T}^{\Bj}$-module, then
  \begin{equation}
  \mathring{\mathbb{B}}(\Rring^{\Bj}\otimes_{\tilde{T}^{\Bj}}M)\cong
    \Rring^{\Bj'}\otimes_{\tilde{T}^{\Bj'}}\mathbb{B}(M).\label{eq:affinize-commute}
  \end{equation}
  This is what we need to prove the composition, since we have a functorial isomorphism
  \[\mathring{\mathbb{B}}(\Rring^{\Bj}e(\Bi))\cong \Rring^{\Bj'}\otimes_{\tilde{T}^{\Bj'}}{\mathbb{B}} (\tilde{T}^{\Bj} e(\tilde{\Bi}))= \Rring^{\Bj'}\otimes_{\tilde{T}^{\Bj'}}\mathbb{B}_1
  \mathbb{B}_2 (\tilde{T}^{\Bj} e(\tilde{\Bi}))=\mathring{\mathbb{B}}_1
  \mathring{\mathbb{B}}_2(\Rring^{\Bj}e(\Bi))\] with the last
step applying \eqref{eq:affinize-commute} twice, to $\mathbb{B}_1$ first and then to $\mathbb{B}_2$.
\end{proof}

\section{The tangle action}
\label{sec:tangle}
\subsection{Affine braids}
Throughout this section, we assume that $\Gamma$ is an ADE quiver and that $ {\beta_e}=0$ for all $e\in \edge$; by Lemma \ref{lem:shift-iso}, since $\Gamma$ is a tree, the algebra $\Rring$ for any set of parameters is isomorphic to one of this form.  Recall that the operation of dualizing fundamental representations induces an involution of $\Gamma$, which coincides with the action of $-w_0$ on fundamental weights; this is the unique nontrivial diagram automorphism for $A_n, D_{2n+1}$ and $E_6$, and trivial for $D_{2n}$ and $E_7,E_8$.  We denote this involution by $j\mapsto j^*$.

We fix dimension vectors $\Bv,\Bw$ as usual, and let $\ell=\sum w_i$.  We let $ {\Rring^{\Bj}}$ be the cylindrical KLRW algebra where the labels on the red strands, reading from 0 to 1, are given by the $\ell$-tuple $\Bj=(j_1, \dots, j_\ell)$.

Let $\widehat{B}_{\ell}$ be the extended braid group of affine type A.  This is the group generated by elements $s_0,\dots, s_{\ell-1}, \sigma$ with relations:
\[ s_is_{i+1}s_i=s_{i+1}s_is_{i+1}\qquad s_i s_j=s_js_i \quad  |i-j|>1\]
\[s_{i+1}\sigma =\sigma s_i.\]
Let $\boldsymbol{\Sigma}_{\Bw}$ be the set of sequences in $\vertex$ where $i$
appears in $w_i$ many times.  The group $\widehat{B}_{\ell}$ acts on
$\boldsymbol{\Sigma}_{\Bw}$ with $s_{i}$ acting by transposition of the $i$ and
$i+1$st entries and $\sigma$ by the cyclic permutation
$\sigma\cdot (j_1,\dots, j_{\ell})=(j_{\ell}, j_1,\dots,
j_{\ell-1})$.

\begin{definition}
The {\bf affine braid groupoid} is the action groupoid for the action
of $\widehat{B}_{\ell}$ on $\boldsymbol{\Sigma}_{\Bw}$.
\end{definition}
The notation $ {\Rring^{\Bj}}$ is useful, since the algebra does
not depend up to isomorphism on the position of the red lines, just their cyclic order. It
will be useful for us to fix parameters $\beta_{i,k}$ so they are
evenly spaced around the circle. That is, we define $k_r$ inductively by \[k_r=\max\{ k_{s}\, |\, j_s=j_r \text{ and } s<r\}+1;\] more concretely, this means that
$j_m$ is the $k_m$th appearance of $j_m$ reading from the start of the
word $\Bj$. Let
$\beta_{j_m,k_m}=\frac{m}{\ell}-\frac{1}{2\ell}$.
For each $i$, we define a linear path for $t\in [0,1]$ that swaps the
$i$th and $i$st parameters, that is, 
\[\bbeta_{j_m,k_m}=\frac{m+\delta_{i,m}t-\delta_{i+1,m}t}{\ell}-\frac{1}{2\ell}.\]
Similarly, to $\sigma$ we associate the path rotating one ``click''
around the cylinder:
\[\bbeta_{j_m,k_m}=\frac{m+t}{\ell}-\frac{1}{2\ell}.\]
These are maybe easier to visualize in terms of the path traced by the
red strands:
\begin{equation*}
        \tikz[xscale=.9]{
      \node[label=below:{$ s_i$}] at (-4,0){ 
       \tikz[very thick,xscale=1]{
          \draw[fringe] (-1.7,-.5)-- (-1.7,.5);
          \draw[fringe] (1.7,.5)-- (1.7,-.5);
          \draw[wei] (.3,-.5)-- (-.3,.5);
          \draw[wei] (-1.3 ,-.5)-- (-1.3,.5);
          \draw[wei] (-.3,-.5)-- (.3,.5);
          \draw[wei] (1.3 ,-.5)-- (1.3,.5);
          \node at (.8,0) {$\cdots$};
          \node at (-.8,0) {$\cdots$};
        }
      };
      \node[label=below:{$ \sigma$}] at (4,0){ 
       \tikz[very thick,xscale=1, yscale=-1]{          
       \draw[fringe] (-1.7,.5)-- (-1.7,-.5);
          \draw[fringe] (1.7,-.5)-- (1.7,.5);
          \node at (0,0) {$\cdots$};
          \draw[wei] (1.3,-.5)-- (.5,.5);
          \draw[wei] (-1.3,.5)-- (-.5,-.5);
          \draw[wei] (1.3 ,.5)-- (1.7,0);
        \draw[wei] (-1.3 ,-.5)-- (-1.7,0);
        }
      };
      }
\end{equation*}
Consider the functors of tensor product with the bimodules
$B_{\bbeta}$ corresponding to these paths: \[\mathbb{B}_i\colon D^b(\Rring^{\Bj}\fdmod)\to D^b(\Rring^{s_i\Bj}\fdmod)\qquad \mathbb{B}_{\sigma}\colon D^b(\Rring^{\Bj}\fdmod)\to D^b(\Rring^{\sigma\Bj}\fdmod).\]
The elements of these bimodules are twisted KLRW diagrams where the red strands trace out the paths above, and the black strands behave normally and satisfy the usual relations.

\begin{theorem}\label{thm:braid-action}
  The functors $\mathbb{B}_i$ and $\mathbb{B}_{\sigma}$ define an affine braid groupoid action on the categories $D^b( {\Rring^{\Bj}}\fdmod)$. 
\end{theorem}This is proven on page \pageref{proof-thm:braid-action}.

\subsection{Cups and caps}
\label{sec:cups-caps}
The functor $\mathbb{B}_i$ corresponds to a path through a wall where
the $i$-th and $(i+1)$-st red strands cross.   This is the only kind of wall we can have in the case where $\beta_e=0$ for all $e\in \edge$ (cf. \eqref{eq:circuit-hyperplane01}).    A black strand will be locked at this wall if and only if it is between the $i$-th and $(i+1)$-st red strands.  
 Thus, the statistic $d(\Bi,\Ba)$  is just the number of black strands between the red strands corresponding to $\beta'_{i,k}$ and $\beta'_{j,\ell}$ and
$ {e_m}$ is the sum of all idempotents $e(\Bi,\Ba)$ with $d(\Bi,\Ba)< m$ such strands.  
 
Associated to the
passage to this wall, we have a filtration of the derived category
$D^b(\Rring^{\Bj}\fdmod)$ by subcategories $\mathcal{D}_m$. By Corollary \ref{cor:wall-crush}, it follows immediately that:
\begin{lemma}
The category $\mathcal{C}_m$ is equivalent to the category of modules over $R^{(z)}:=\Rring^{\Bj}/\Rring^{\Bj}e_m\Rring^{\Bj}$.
\end{lemma}
Note, the triangulated category $\mathcal{D}_m$ is not typically the derived category of $\mathcal{C}_m$, since the Ext-algebra of $R^{(z)}$ as an $\Rring^{\Bj}$-module is not concentrated in degree 0.  

We'll be particularly interested in the case where the consecutive strands have labels $j,j^*$.  Without loss of generality, we can assume these are the labels $j_{\ell-1},j_{\ell}$.  Let $\Bj'=(j_1,\dots, j_{\ell-2})$. In this case, the deepest level of the filtration is $\mathcal{C}_z$ where \[z=\rho^{\vee}(\varpi_j+\varpi_{j^*})=\rho^{\vee}(\varpi_j-w_0\varpi_j)=2\rho^{\vee}(\varpi_j).\]
Let $\varpi_j+\varpi_{j^*}=\sum_{i\in \vertex} z_i\alpha_i$.  Note that $z=\sum z_i$.
\begin{lemma}\label{lem:eat-cup}
The ring $R^{(z)}$ in this case is Morita equivalent to the algebra $\Rring^{\Bj'}$ with dimension vector $\Bv'=\Bv-\Bz$.  \end{lemma}This is proven on page \pageref{proof-lem:eat-cup}. Put more informally, this Morita equivalence simultaneously removes from the diagram the two red strands, and then $z_i$ black strands with label $i$.  We can picture the two red strands as colliding (we are approaching the wall where they do so), and pinching off the black strands as they meet to make the top of a red cup.  The elements of the Morita equivalence bimodule look like below:
 
\centerline{\includegraphics[width=.5\textwidth]{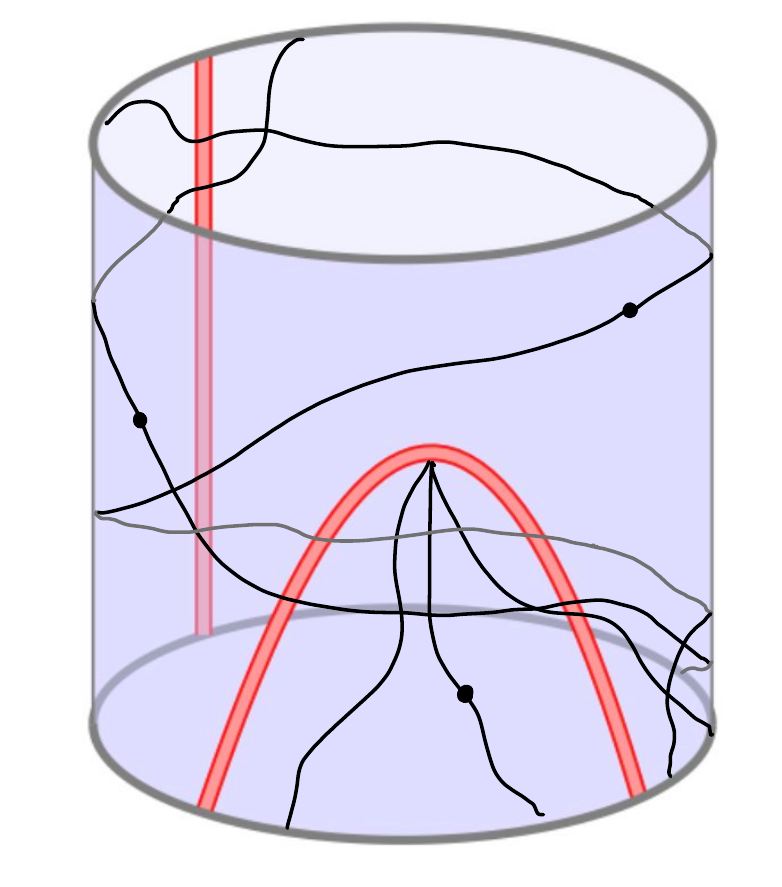}}

\begin{example}
	If $\Gamma$ is an $A_n$ quiver with nodes identified with $\{1,\dots, n\}$, then $j\mapsto j^*=n+1-j$ is the unique automorphism flipping the ends of the quiver.  The coefficients $z_i$ in this case are given by:
	\begin{equation*}
		z_i=\begin{cases}
		i & i\leq j\\
		j & j\leq i\leq j^*\\
		n+1-i & i\geq j^*
	\end{cases}
	\end{equation*}
	So, for $n=1$, which is the case $\fg=\mathfrak{sl}_2$, we have $1^*=1,z_1=1$, so in this case, a pair of colliding red strands will pinch off one black strand, coming from the fact that the unique positive root is twice the unique fundamental weight.  
\end{example}

This Morita equivalence defines a $R^{(z)}\operatorname{-}\Rring^{\Bj'}$-bimodule $B_\cup$ and a  $\Rring^{\Bj'}\operatorname{-}R^{(z)}$-bimodule $B_\cap$; of course, these are projective as left and right modules.  Inflating, we can consider the same abelian group as a $\Rring^{\Bj}\operatorname{-}\Rring^{\Bj'}$-bimodule $B_\cup$ and a  $\Rring^{\Bj'}\operatorname{-}\Rring^{\Bj}$-bimodule $B_\cap$, which are now projective only as right and left modules respectively.  

These are obtained from the quantum trace and coevaluation bimodules defined in \cite[\S 7.3]{Webmerged} by the cylindricalization operation $\mathfrak{K}\mapsto \mathring{\mathfrak{K}}$ defined in Lemma \ref{lem:affinize-commute}.

This defines an exact functor  $\cup_\ell=B_\cup\otimes_{\Rring^{\Bj'}}-\colon \Rring^{\Bj'}\mmod \to \Rring^{\Bj}\mmod.$  We can try to define a left adjoint $\cup_\ell^*\colon D^b(\Rring^{\Bj}\mmod)\to D^b(\Rring^{\Bj}\mmod)$ to this functor by derived tensor product with $B_{\cap}$. We say ``try'' above because there's no guarantee that this functor will preserve being a bounded complex.

Recall that we call a fundamental weight $\varpi_j$ {\bf minuscule} if
all of the non-zero weight spaces of its corresponding representation
are extremal; the number of such representations for a finite dimensional semi-simple Lie algebra is the determinant of the Cartan matrix minus 1 (there is one that transforms under each non-trivial character of the center of the simply-connected form of the Lie group).  This holds for all fundamental weights in type A, for
the vector and both spin representations in type D, two
representations for $E_6$ and one for $E_7$.  For simplicity, we call
$i\in \vertex$ minuscule if the corresponding fundamental
representation is minuscule, and let $\vertex_{\operatorname{min}}$ be
the subset of minuscule elements of $\vertex$.  
\begin{lemma}\label{lemma:cap-finite}
We have a well-defined functor $\cup_\ell^*\colon
D^b(\Rring^{\Bj}\fdmod)\to D^b(\Rring^{\Bj'}\fdmod)$ if
$j\in  \vertex_{\operatorname{min}}$.  
\end{lemma}
\begin{proof}
    This functor is well-defined if and only if
  $\Rring^{\Bj}/\Rring^{\Bj}e_z \Rring^{\Bj}$ has a
  finite resolution as a right module over $\Rring^{\Bj}$.  As noted in the proof of \cite[Prop. 8.3]{Webmerged}, if $\Bj=(j,j^*)$, then the ring $\tilde{T}^{\Bj}$ has finite global dimension, so the simple $L_0$ has a finite length projective resolution.  The module $e(\Bi)(\tilde{T}^{\Bj}/\tilde{T}^{\Bj}e_z\tilde{T}^{\Bj})$ thus has obtained a resolution obtained by just adding the red and black strands from $\Bi$ on the left and right of this resolution, as needed.  Applying cylindricalization sends this to a projective resolution of $e(\Bi)(\Rring^{\Bj}/\Rring^{\Bj}e_z \Rring^{\Bj})$ by the exactness of the cylindricalization functor.  
\end{proof}

\begin{remark}
This result holds in effectively no non-minuscule cases.  
  As in \cite{Webmerged}, we could consider a category of complexes which are bounded above or below with weaker finiteness properties and use non-minuscule weights; for simplicity, we don't work through the details of this.  
\end{remark}
\begin{definition}\label{def:cup-cap}
Let $\cap_\ell=\cup_\ell^*[-m](-m)$; we remind the reader that $[-m](-m)$ means that we shift degree by {\it increasing} the internal and homological degree of any element by $m$.  Note that this is the opposite of the shift in \cite[Def. 7.4]{Webmerged}, because we are defining our functor using a tensor product.  We also have $\cap_\ell=\cup_\ell^![m](m)$ where $\cup_\ell^!$ is the right adjoint to $\cup_\ell$ by \cite[Thm. 8.11]{Webmerged}.  

The cup and cap functors $\cup_k=\mathbb{B}_{\sigma}^{k-\ell} \cup_\ell\mathbb{B}_{\sigma}^{\ell-k}$ and $\cap_k=\mathbb{B}_{\sigma}^{k-\ell}\cap_\ell\mathbb{B}_{\sigma}^{\ell-k}$ are obtained by pre- and post-composing with a 
twist of the cylinder so that there are $k$ red strands left of the
two red strands that meet (i.e., they are the $(k+1)$st and $(k+2)$nd
when reading around the circle from $0$).  
\end{definition}
\notation{$\cup_k,\cap_k$}{The cup and cap functors from Definition \ref{def:cup-cap}.}

Just as the bimodules attached to basic braids can be visualized with KLRW-type diagrams, we can visualize the cup and cap functors in terms of bimodules composed of KLRW diagrams satisfying the usual rules at most points and we insert some special behavior at the top of a cap or a bottom of a cup. \excise{\begin{center}
  \begin{tikzpicture}[MyPersp,font=\large]
\def\h{1.5}\fill[blue,fill opacity=.05]  
		 (1,0,{\h})--(1,0,0)
		\foreach \t in {0,-2,-4,...,-180}
			{--({cos(\t)},{sin(\t)},0)}
-- (-1,0,0)--(-1,0,{\h})
		\foreach \t in {180,178,...,0}
			{--({cos(\t)},{sin(\t)},{\h})}--cycle;
\draw[gray,   very thick] (1,0,0)
		\foreach \t in {0,2,4,...,180}
			{--({cos(\t)},{sin(\t)},0)};
\draw[wei] ({cos(50)},{sin(50)},0) 
		\foreach \t in {0,2,...,150}
		{--({cos((50))} ,{sin((50))},{.01*\t})};
		\fill[blue!20!white,fill opacity=.5]
		 (1,0,{\h})--(1,0,0)
		\foreach \t in {0,-2,-4,...,-180}
			{--({cos(\t)},{sin(\t)},0)}
-- (-1,0,0)--(-1,0,{\h})
		\foreach \t in {-180,-178,...,0}
			{--({cos(\t)},{sin(\t)},{\h})}--cycle;
						\draw[wei] ({cos(220)},{sin(220)},0) 
		\foreach \t in {0,2,...,90}
		{--({cos((\t+220))} ,{sin((\t+220))},{sin(2*\t)})};	
\draw[gray] (1,0,0)--(1,0,{\h});
\draw[gray] (-1,0,0)--(-1,0,{\h});
\draw[gray, very thick] (1,0,0) 
		\foreach \t in {0,-2,-4,...,-180}
			{--({cos(\t)},{sin(\t)},0)};
\draw[gray, very thick] (1,0,\h) 
		\foreach \t in {2,4,...,360} 
			{--({cos(\t)},{sin(\t)},{\h})}--cycle;
                      \end{tikzpicture}
\end{center}}

\subsection{Annular tangles}
Let $\mathbb{A}$ be the annulus $\R^2\setminus\{(0,0)\}$.  Let
$\pi\colon \mathbb{A}\to S^1$ be the obvious projection along rays;
we'll use the same symbol to denote the induced map $\mathbb{A}\times
[0,1]\to S^1\times [0,1]$.
\begin{definition}
An {\bf oriented framed annular tangle} $T$ is a framed tangle in
$\mathbb{A}\times [0,1]$, that is, it is a 1-dimensional oriented submanifold
whose boundary lies in $S^1\times \{0,1\}$, together with a fixed 1-dimensional subbundle of its normal bundle.   We often visualize this by thickening the tangle to a ribbon by extending it a small distance along the subbundle.

As usual, we number the
boundary points of the tangle by their cyclic order around $S^1\cong
\R/\Z$, starting at $0$. We'll consider these up to isotopy in the space
$\mathbb{A}\times [0,1]$ that keep the points on $S^1$.  Note that this means that the cyclic order of points on $S^1$ is preserved.  A {\bf projection} of an annular tangle $T$ is the
image of $T$ under $\pi$ in $S^1\times [0,1]$ when it is isotoped so
that $\pi$ is an immersion on $T$, and any point in
$S^1\times [0,1]$ has at most two pre-images in $T$, whose images cross
transversely (i.e. we avoid triple points and tangencies).  We account
for the framing on tangles by only using projections with the ``blackboard framing'' i.e. where the derivative of the projection always induces an isomorphism between the tangent plane to the ribbon discussed above and the tangent space of $S^1\times [0,1]$.  
\end{definition}

As always when considering tangles in a thickened surface, we have that any two projections for isotopic ribbon tangles are related by a finite chain of isotopies and Reidemeister moves II and III, as well as canceling pairs of Reidemeister I moves which preserve the blackboard framing; a single Reidemeister I move will not preserve the framing.  

We consider labelings of the components of an oriented ribbon tangle with elements $j\in \vertex_{\operatorname{min}}$ (which we think of as the corresponding fundamental representation).  As usual, these induce a labeling of the boundary of $T$, where we use the same element of $\vertex$ if the orientation on $T$ matches the upward orientation of $[0,1]$ under projection, and the ``dual'' $j^*$ if the orientations are opposite.  

\begin{definition}
Let $\mathsf{Tang}$ be the category such that:
\begin{itemize}
\item objects are finite subsets
  of $S^1=\R/\Z$ labeled with elements of $\vertex_{\operatorname{min}}$,
  \item morphisms $S\to S'$ are annular ribbon tangles with boundary
    in $\mathbb{A}\times \{0\}$ given by $S$ and  in $\mathbb{A}\times
    \{1\}$ by $S'$.
    \item composition is just stacking of tangles (followed by appropriate isotopy).   
    \end{itemize}
    We let the list $(j_1,\dots, j_\ell)$ denote any fixed subset of $\R/\Z$ where this is the list of labels in cyclic order. 
\end{definition}

We have an obvious functor from the affine braid groupoid to
$\mathsf{Tang}$ giving the tangles with no minima or maxima.  In order
to generate all tangles, we need only add cup and cap functors joining
two adjacent points, which we also denote $\cup_k, \cap_k$ when the
cup and cap attach to the $(k+1)$st and $(k+2)$nd terminals when reading around the
circle from $0$; of course, this functor depends on the labels $j_{k+1},j_{k+2}\in \vertex_{\operatorname{min}}$, but we leave these implicit.  Consider
two lists $\Bj\in \vertex_{\operatorname{min}}^{s}$ and $\Bj'\in \vertex_{\operatorname{min}}^{s'}$;  these give corresponding dimension vectors $\Bw$
and $\Bw'$ where $w_i$ is the number of $k$ with $j_k=i$, and
similarly with $w_i'$.  Fix a vector $\Bv$, and let
$\Bv'=\Bv+C^{-1}(\Bw'-\Bw)$ where $C$ is the Cartan matrix of $\Gamma$; of course, this is not necessarily
integral, but it will be if there is a morphism $T\colon \Bj \to \Bj'$.
\begin{theorem}\label{thm:tangle-action}
  For each morphism $T\colon \Bj \to \Bj'$ in $\mathsf{Tang}$, there is an associated functor $\Phi(T)\colon D^b(\Rring^{\Bj}_{\Bv}\fdmod)\to D^b(\Rring^{\Bj'}_{\Bv'}\fdmod)$ satisfying $\Phi(T_1\circ T_2)\cong \Phi(T_1)\circ \Phi(T_2)$ such that $\Phi(\tau)$ for an affine braid $\tau$ is the wall-crossing functor $\mathbb{B}_{\tau}$, and for a cup or cap, we have $\Phi(\cup_k)=\cup_k$ and $\Phi(\cap_k)=\cap_k$.
\end{theorem}
\notation{$\Phi$}{The annular tangle invariant defined by Theorem \ref{thm:tangle-action}.}
This is proven on page \pageref{proof-thm:tangle-action}.

You can think of this as defining a functor from $\mathsf{Tang}$ to
the category whose objects are triangulated categories, and whose
morphisms are exact functors up to isomorphism.  

The most important special case for us is a tangle with no boundary points, that is, an annular link. This construction assigns a finite dimensional bigraded vector space to
any annular knot or link with components labeled by minuscule
representations by considering this link as a tangle and applying it
to $\Rring^{\emptyset}_{\mathbf{0}}\cong \C$.  

\begin{theorem}\label{th:knot-invariant}
  If $K\subset B^3\subset \R^3$ is any link, and we embed $K$ as an annular link via any embedding $B^3\hookrightarrow\mathbb{A}\times [0,1]$ (all such embeddings are isotopic), the invariant $\Phi(K)$ is the same as the invariant $\Phi_{\mathsf{L}}(K)$ defined in \cite[\S 8.1]{Webmerged} for the same labeling.  In particular: \begin{enumerate}
      \item If $\Gamma$ is of type $A_1$, then $\Phi(K)$ coincides with Khovanov homology.
      \item If $\Gamma$ is of type $A_n$, then $\Phi(K)$ coincides with Khovanov-Rozansky $\mathfrak{sl}_{n+1}$-homology.  
  \end{enumerate}
\end{theorem}
This is proven on page \pageref{proof-th:knot-invariant}.
A more comprehensive description of all the different manifestations of this invariant in type A is given in \cite[Th. A]{mackaayCategorifiedSkew2018}. 

However, constructing an invariant of knots in $\mathbb{R}^3$ is only using a small portion of the power of this construction: we obtain an invariant for each annular link, which will depend on how the link wraps around the origin. Of course, in type A, invariants of this type are well-known: annular Khovanov-Rozansky homology.  This is defined by Queffelec and Rose in \cite{QRannular}. 

It also seems natural to compare with the following construction: given an annular link $K$, we can cut along the plane over the positive $x$-axis in $\mathbb{A}$ to obtain a usual tangle $K'$ in $\mathbb{R}^2\times [0,1]$.  This has an associated complex of $T^{\bla}\operatorname{-}T^{\bla}$-bimodules $\Phi_{\mathsf{L}}(K')$ over the (planar) KLRW algebras defined by the construction of \cite[\S 8.1]{Webmerged}.  The Hochschild homology $H\!H_{T^{\bla}}(\Phi_{\mathsf{L}}(K'))$ is easily seen to be an annular link invariant: 
\begin{conjecture}\label{conj:annular-KR}
The following invariants of annular knots labeled with minuscule representations coincide:
\begin{enumerate}
    \item the invariant $\Phi(K)$ constructed above;
    \item the Hochschild homology $H\!H_{T^{\bla}}(\Phi_{\mathsf{L}}(K'))$;
    \item if $\Gamma=A_n$, the annular Khovanov-Rozansky homology of \cite{QRannular}. 
\end{enumerate}
\end{conjecture}
We don't expect this conjecture to be exceptionally difficult; the
equivalence of (1-3) in type A should be approachable by rephrasing constructions (1) and (2) using actions of foams as in \cite{mackaayCategorifiedSkew2018}, and showing that (1) and (2) are both defined by annular evaluation of foams.  The main difficulty here is showing that a single essential circle on the annulus evaluates to the corresponding representation over the Lie algebra $\mathfrak{g}_{\Gamma}$.  In case (2), this is effectively just the observation that $T^{\bla}$ has finite global dimension, since this means that higher Hochschild homology of $T^{\bla}$ vanishes, and in degree 0 it matches the Grothendieck group.  In case (1), this is not obvious for $n>2$, and requires a rather complex calculation.  It seems promising to think of the equivalence of (1) and (2) as a generalization of Queffelec and Rose's comparison of the horizontal and vertical traces in \cite[\S 3]{QRannular}, but we have not made much progress on making this precise.

\section{Diagrams for quiver gauge theories}
\label{sec:diagrams}

\excise{
First, we note some basic facts about quiver gauge theories.  In this case, the group $\No=N_{GL(V)}(G)$ is generated by $\gaugeG$, the product
$GL(\C^{ {w_i}})$ acting by precomposition in the obvious way, and by
$GL(\C^{\chi_{i,j}})$ where $\chi_{i,j}$ is the number of edges
$i\to j$, acting by taking linear combinations of the maps along these
edges, that is, via the isomorphism
\[\bigoplus_{i\to j}\Hom(\C^{ {v_i}},\C^{ {v_j}})\cong \bigoplus_{(i,j)\in {\vertex}}
  \Hom(\C^{ {v_i}},\C^{ {v_j}})\otimes \C^{\chi_{i,j}}.\]}

In order to establish the results of Sections
\ref{sec:cylindr-klrw-algebr} and \ref{sec:tangle}, we need to
describe how the results of \cite{WebcohI,websterKoszulDuality2019} can be interpreted in
the quiver case.  In those papers, we explain how the quantum Coulomb
branch $\Asph$ can be written as an endomorphism ring in a larger
category $\scrB^+$, which we call the {\it extended BFN category } (Definition \ref{I-def:extended-BFN}).
This larger category is more easily presented and more amenable to
algebraic methods; this will allow us to make the connection between
Coulomb branches and cylindrical KLR algebras.  
\notation{$\scrB^+$}{The extended BFN category (Definition \ref{I-def:extended-BFN}).}

The definition of $\scrB^+$ depends on a parameter $\delta$.
For simplicity, we'll assume throughout this paper that $\delta=\frac{1}{2}$ (and so in
the ``$p$th root'' conventions of Definition \ref{I-def:pth-root}, we have $\delta=\frac{1}{2p}$).

\subsection{Unrolled diagrams}
\label{sec:diagr-descr}

In the case of a quiver gauge theory, the extended BFN category
$\scrB^+$ has
a more graphical description.

Recall that this is a category whose objects are elements of
$\ft_{1,
  \No,\R}$; this space can be identified with the (real) Lie algebra of the
maximal torus of the normalizer $\No=N^\circ_{GL(V)}(G)$, but the subscript $1$ indicates that we twist the action of these elements on $V((t))$ by the loop action with weight $1$ on $t$.  As
discussed earlier, the Lie algebra $\ft_{\No,\R}$ is generated by the diagonal matrices in $\mathfrak{gl}(\R^{v_i}),\mathfrak{gl}(\R^{w_i})$ and $\mathfrak{gl}(\R^{\epsilon_{i,j}})$;
thus, each object in this category can be represented by choosing
diagonal entries.  We let $\{z_{i,k}\}_{i\in \vertex, k=1,\dots, v_i}$
be the diagonal entries in $\mathfrak{gl}(v_i)$; for slightly
complicated reasons, we let $\{ {\hat\beta_{i,k}}\}_{i\in \vertex, k=1,\dots, w_i}$ and $ \{ {\hat\beta_e}\}_{e\colon i\to j}$ be the appropriate diagonal entries in $\mathfrak{gl}(w_i)$ and $\mathfrak{gl}(\epsilon_{i,j})$ plus $\frac{1}{2}$.  This is to cancel the shift in the definition of $\varphi^{\operatorname{mid}}_i$, given in \eqref{I-eq:varphi}; in particular, the element $\second$ which acts on $vt^a\in V((t))$ with weight $a$ corresponds to $z_{i,k}=0$ and $\hat\beta_{i,k}=\hat\beta_e=\frac{1}{2}$.  This is related to the convention in \cite[Convention, pg. 6]{BFN} that the loop $\C^*$-action on $V((t))$ ($N$ in the notation of \cite{BFN}) has weight $1/2$ on $V$.

In these terms, the unrolled arrangements defined in Section
\ref{I-sec:extended} are given by the unrolled root hyperplanes
$\{\alpha(\acham)=n\mid n\in\Z\}$ of the form:\newseq
\[\subeqn z_{i,k}-z_{i,m}=n\qquad \text{ for all }k\neq m\in [1, {v_i}], n\in \Z,\label{eq:unroll-root}\] and the unrolled
matter hyperplanes $\{\varphi_i^{\operatorname{mid}}(\acham)=n\mid n\in
\Z\}$  of the form for all $n\in \Z$:
\begin{align*}\subeqn\label{eq:unroll-matter1}
z_{i,k}-z_{j,m}+\hat\beta_e&=n\qquad \text{ for all edges } j\to i,
  \text{ for all } k\in [1, {v_i}], m\in [1, {v_j}]\\
\subeqn \hat\beta_{i,k}-z_{i,m}&=n\qquad \text{ for all } i\in {\vertex}, m\in [1, {v_i}], k\in [1, {w_i}] \label{eq:unroll-matter2}
\end{align*}
Note that unlike in Section \ref{sec:cylindr-klrw-algebr}, in this context, the left-hand sides of these equations are real-valued linear functions on $\ft_{{H}}$.  We will sometimes want to fix these and sometimes allow these to vary.  Note that if we change $\hat\beta_*$ by an integer amount, we will not change the set of hyperplanes (\ref{eq:unroll-matter1}--\ref{eq:unroll-matter2}).  

Fix a choice of $ {\hat\beta_*}$; let
$\mathscr{B}_{\hat\beta}$ be the subcategory $\scrB^+$ where we
only consider objects in the coset $\ft_{\hat\beta}$ of this choice under $\ft_{\R}$,
the Cartan of $\mathfrak{g}=\prod_i\mathfrak{gl}(\R^{v_i})$.  For simplicity, we assume that $\hat\beta_*$ are generic in the same sense as discussed in Section \ref{sec:cylindr-klrw-algebr}.  We can describe the objects in this category just using the values $z_{*,*}\in \R$.

The category $\mathscr{B}_{\hat\beta}$ also depends on a flavor
$\phi$, which we can take to be an element of the Lie algebra of
$\Aut_G(V)$, and thus is induced by a regular element of $\prod_i\mathfrak{gl}(\K^{ {w_i}})
\times\prod_{(i,j)\in {\vertex}^2} \mathfrak{gl} (\K^{\chi_{i,j}})$.  We can
assume that this cocharacter lands in the usual torus of diagonal matrices; let $ {b_{i,1}},\dots, b_{i, {w_i}}$ be the diagonal entries of its components into $\mathfrak{gl}  (\K^{ {w_i}})$ and $ {b_e}$ for each edge $e$ the diagonal entries in $\mathfrak{gl}  (\K^{\chi_{i,j}})$.  Let us just emphasize that while the choice of $b_*$ and  $\hat\beta_*$  seem like similar information, the latter parameters are always chosen over $\R$, whereas the former are defined over the base field (which we will often want to think of as $\C$ or $\mathbb{F}_p$) and are chosen independently from $\hat\beta_*$.   \notation{$b_e,b_{i,k}$}{The parameters of the category $\mathscr{B}$ from the flavor deformation, corresponding to the equivariant parameters of $\No/\gaugeG$.}

Consider a path $\pi\colon [0,1]\to \ft_{H}$.  We can understand this path by considering how the coordinates $z_{i,k}, \hat\beta_{i,m},\hat\beta_e$ vary as $t$ varies from $0$ to $1$.  We'll describe this as a
$ {v_i}$-tuple of paths $\pi_{i,k}$ for each $i\in {\vertex}$, a $w_i$-tuple $\hat{\bbeta}_{i,m}$ and a path $\hat{\bbeta}_e$ for each $e\in \edge$.  We can visualize
this by superimposing the graphs of the paths $\pi_{i,k}$ and $\hat{\bbeta}_{i,m}$ as corporeal and red strands on the plane $\R\times [0,1]$, and drawing ghost strands at $\pi_{h(e),k}+\hat{\bbeta}_e$ for all $e\in \edge, k\in [1,v_{h(e)}]$.
We will use
the opposite convention from calculus class, using the $y$-axis for
the independent variable and the $x$-axis for the dependent, so we consider the path $t\mapsto (\pi_{i,k}(t),t)$ landing in
$\R\times [0,1]$.

We cross root hyperplanes when two of these paths for the same element
of ${\vertex}$ are integer distance
from each other, and matter hyperplanes when we solve one of the equations
\begin{equation*}
z_{j,m}=z_{i,k}+\hat\beta_e-n\qquad \text{or}\qquad 
z_{i,m}=n-\hat\beta_{i,k}.
\end{equation*}
Following the convention of \cite[\S 3.1]{Webmerged}, we call a point {\bf generic} if it avoids all these hyperplanes, and {\bf unexceptional} if it avoids the matter hyperplanes (but potentially not the root hyperplanes. 
It is thus convenient to label the points $(\pi_{i,k}(t)+n,t)$ for
$n\in \Z$ with ``partner'' strands and $(\pi_{h(e),k}(t)+\hat{\bbeta}_e(t)+n,t)$ with ``ghost partner'' strands so that we can see when these
crossings occur.  We'll label the strand corresponding to $n$ with $i;n$.  We'll draw partners as solid lines,
and ghosts partners as dashed lines.  We will also draw the path $(\hat{\bbeta}_{i,m}(t)+n,t)$ as a dashed red line, labeled with $i;m$.

\notation{\ensuremath{\What}}{The extended affine Weyl group.  In the case where $W=S_n$, this is the affine permutation group $S_n\ltimes \Z^n$.}
There's an obvious action of the affine Weyl group $\What=\prod_{i\in \vertex}(\Sigma_{v_i}\ltimes \Z^{v_i})$ on the set of such
paths with the finite Weyl group $\prod_{i\in \vertex}\Sigma_{v_i}$ acting on paths by permutation of
the second indices in $\pi_{i,k}$, and the coweight lattice $\Z^{v_i}$ acting by
translations $\pi_{i,k}(t)\mapsto \pi_{i,k}(t)+n_{i,k}$ for integers
$n_{i,k}$.  This leaves the collections of the original corporeal, ghost, and red curves and
their partners unchanged, just changing the indices and which curves
are partners and which are originals. Thus, we can visualize the
action of an affine Weyl group element by changing the labels
accordingly at some fixed value of $y=a$.  The equations
(\ref{I-eq:conjugate2}) and (\ref{I-eq:psiconjugate}) ensure that the
result does not depend on the value of $a$.

\notation{$S_h$}{The polynomial ring $\Sym \ft^*[h]$, considered as endomorphisms of each object in $\mathscr{B}_{\hat\beta}$.}
Another set of morphisms of $\mathscr{B}_{\hat\beta}$ are given by a copy of the polynomial ring $ {S_h}=\Sym \ft^*[h]$; geometrically, this is the equivariant cohomology of a point with respect to $T\times \C^*$.  We can also visualize the action of this ring by identifying the weights
$\ep_{i,k}$ with a dot on the corresponding path, at the top for multiplying on the left and at the bottom for multiplication on the right.  It will be useful to adopt the convention that if we draw dots on the partner strand with label $i;n$, this corresponds to $\ep_{i,k}+nh$.  This ensures that inserting an element of the affine Weyl group above or
below a dot will give the same answer, by equation (\ref{I-eq:dot-commute}).

Thus, the morphisms in $\scrB_{\hat\beta}$ can all be expressed as one of these diagrams.  More precisely:
\begin{definition}
  An {\bf unrolled diagram} is the following data:
 \begin{enumerate}
     \item A finite list $w_1,\dots, w_r\in
  \widehat{W}$ of elements of the extended affine Weyl group, and heights $0<h_1<\dots< h_r<1$ where we will apply these.
  \item Smooth maps \[\varpi_{i,k}\colon [0,1]\to \R\text{ for }k\in [1,v_i]\]
  \[\hat{\bbeta}_{i,m}\colon [0,1]\to \R\text{ for }m\in [1,v_i]\]
  \[\hat{\bbeta}_e\colon [0,1]\to \R\text{ for }e\in \edge\]
  satisfying genericity conditions discussed below.
  \item A placement of any number (including zero) of dots on the strands $\{ (\varpi_{i,k}(t)+n,t) \mid t\in [0,1]\}$ for
  $n\in \Z$, also satisfying genericity conditions given below.  
 \end{enumerate}
 We can define a piecewise smooth path $\pi_{i,k}$ by letting
 \[(\pi_{i,1}(t),\dots \pi_{i,v_i}(t))=w_s\cdots w_1(\varpi_{i,1}(t),\dots \varpi_{i,v_i}(t))\qquad \text{for all }t\in [h_s,h_{s+1}). \]
We'll represent these by drawing diagrams with: 
\begin{enumerate}
    \item a horizontal green squiggly line at each $y=h_s$, labeled with $w_s$.
    \item corporeal strands on $\{ (\pi_{i,k}(t),t) \mid t\in [0,1]\}$ and their partner strands at $\{ (\pi_{i,k}(t)+n,t) \mid t\in [0,1]\}$ for $n\in \Z\setminus \{0\}$, which we label with $i;n$.
    \item 
  ghost  and ghost partner strands 
  at $\{ (\pi_{h(e),k}(t)+\hat{\bbeta}_e(t)+n,t) \mid t\in [0,1]\}$ for each edge
  $e\in \edge$ which we label with $e$ and $e;n$ for $n\neq 0$.
  \item red strands at $\{
  (\hat{\bbeta}_{i,m}(t)+n,t) \mid t\in [0,1]\}$ for all $n\in \Z$, which we draw as dashed and label with $i;n$.
  \item drawing dots at generic points on the corporeal and partner strands, but not the dashed ones.  
\end{enumerate}
For genericity, we require there are no tangencies, triple crossings, or dots on
crossings.  The curves (including ghosts) must
meet the lines at $y=0$ and $y=1$ at distinct points. We consider these
diagrams up to isotopy preserving the conditions above.
\end{definition}
We'll draw an example below; lacking infinitely wide
paper, we can only draw part of the diagram.  For example if $\Gamma=\tikz[very thick,baseline=-5pt]{\node (a) at (0,0){$i$}; \node (b) at (1,0){$j$}; \draw[->] (a) to[out=30,in=150] node[above,midway]{$f$} (b); \draw[->] (b) to[out=-150,in=-30] node[below,midway]{$e$}(a);}$, an unrolled diagram with $w_i=1,w_j=0,v_i=1,v_j=1$ could look like:
\begin{equation} \label{eq:unroll-example}
       \tikz[very thick,xscale=1.7,baseline]{
           \draw (.8 ,-1) to[out=70,in=-110] node[below,at start, scale=.8]{$i$}(-.3,1);
           \draw (-.8 ,-1) to node[above, at end, scale=.8]{$j;-1$} (.3,1);
               \draw (-1.2 ,-1) to[out=70,in=-110]  node[below,at start, scale=.8]{$i;-1$} (-2.3,1);
               \draw (-2.8 ,-1) to node[above, at end, scale=.8]{$j;-2$}  (-1.7,1);
               \draw (1.2 ,-1) to node[above, at end, scale=.8]{$j$}  (2.3,1);
        \draw (2.8 ,-1) to[out=70,in=-110]  node[below,at start, scale=.8]{$i;1$} (1.7,1);
        \draw[dashed] (-.2,-1) -- node[below,at start, scale=.8]{$e;-1$} (-1.3,1);
        \draw[dashed] (.2,-1) to[out=50,in=-110]  node[above, at end, scale=.8]{$f;-1$}  (1.3,1);
        \draw[dashed] (-1.8,-1) to[out=50,in=-110] node[above, at end, scale=.8]{$f;-2$}  (-.7,1);
        \draw[dashed] (.7,1) -- node[below,at end,
        scale=.8]{$e$} (1.8,-1);
        \draw[dashed] (-2.2,-1) -- node[below,at start, scale=.8]{$e;-2$} (-3.3,1);
        \draw[dashed] (2.2,-1) to[out=50,in=-110] node[above, at end, scale=.8]{$f$}  (3.3,1);
        \draw[dashed] (-3.8,-1) to[out=50,in=-110] node[above, at end, scale=.8]{$f;-3$}  (-2.7,1);
        \draw[dashed] (2.7,1) -- node[below,at end,
        scale=.8]{$e;1$} (3.8,-1);
        \node at (4,0){$\cdots$};
         \node at (-4,0){$\cdots$};
         \draw[red,dashed] (2,1) to[out=-80,in=80] node[below,at end,scale=.8]{$i;2$} (2.4,-1);
        \draw[red,dashed] (0,1) to[out=-80,in=80] node[below,at end,scale=.8]{$i;1$}(.4,-1);
        \draw[red,dashed] (-2,1) to[out=-80,in=80] node[below,at end]{$i$}(-1.6,-1);
        }
\end{equation}
Let the {\bf bottom} $D(0)$ of a diagram defined by the path $\pi$ be the
object in $\mathscr{B}_{\hat\beta}$ defined $z_{i,k}=\pi_{i,k}(0)$, and the {\bf
  top} $D(1)$ the object defined by $z_{i,k}=\pi_{i,k}(1)$.  The requirement that strands do not meet at the top and bottom of the diagram implies that these are generic in the sense of avoiding root and matter hyperplanes.

To an unrolled diagram $D$, we can associate a morphism $ {\mathbbm{r}_D}\colon D(0)\to D(1)$ 
from the bottom of $D$ to
its top as follows:
\begin{definition}\hfill\label{def:rD}
  \begin{enumerate}
  \item Given an unrolled diagram $D$ with no dots and no elements of $\widehat{W}$, let $\mathbbm{r}_D$ be the morphism $\mathbbm{r}_\pi$ from \cite[Def. 3.12]{websterKoszulDuality2019} for the corresponding path.
  \item
    Given an unrolled diagram with all $\pi_{i,k}$'s constant and one dot on the strand at $(\pi_{i,k}+n,t)$, we let $\mathbbm{r}_D$ be multiplication by $\ep_{i,k}+nh$.
\item   Given an unrolled diagram with no dots, all strands vertical, and a single relabeling by $w\in \widehat{W}$, we let $\mathbbm{r}_D=y_w$.
  \end{enumerate}
Any other diagram can be written as a composition of these, and we
define $ {\mathbbm{r}_D}$ to be the composition of the corresponding morphisms.
\end{definition}

This associates a morphism to a diagram, and since we hit all the
generators of the category $\mathscr{B}^+$, every morphism is a sum of
$\mathbbm{r}_D$'s; this is clear from \cite[Cor. 3.15]{websterKoszulDuality2019}.  However, this description is redundant, since
there are relations between these generators. 

\newseq
 All the relations of the category $\mathscr{B}^+$ can be described locally in terms of unrolled
 diagrams.      We can visualize Definition  \ref{def:rD}(2) as a convention
 of ``dot migration'' for interpreting dots
 on partner strands:
\begin{equation*}\label{eq:dot-migration}\subeqn
    \begin{tikzpicture}[very thick,baseline,scale=.9]       \draw (1,0)
      +(1,-1) -- +(1,1) node[below,at start,scale=.8]{$i;n$};  \fill (2,0) circle (3.5pt);\node at (.5,0){$\cdots$}; \draw (0,0)
      +(-1,-1) -- +(-1,1) node[below,at start]{$i$};  
    \end{tikzpicture}=\:\begin{tikzpicture}[very thick,baseline,scale=.9]       \draw (1,0)
      +(1,-1) -- +(1,1) node[below,at start,scale=.8]{$i;n$};  \fill (-1,0) circle (3.5pt);\node at (.5,0){$\cdots$}; \draw (0,0)
      +(-1,-1) -- +(-1,1) node[below,at start]{$i$};  
    \end{tikzpicture}+nh\:\begin{tikzpicture}[very thick,baseline,scale=.9]       \draw (1,0)
      +(1,-1) -- +(1,1) node[below,at start,scale=.8]{$i;n$}; \node at (.5,0){$\cdots$}; \draw (0,0)
      +(-1,-1) -- +(-1,1) node[below,at start]{$i$};  
    \end{tikzpicture}
  \end{equation*}

 The fact that isotopy leaves $\mathbbm{r}_D$ invariant is a
combination of relations (\ref{I-eq:dot-commute}) and the ``boring'' cases of (\ref{I-eq:wall-cross1},\ref{I-eq:psi},\ref{I-eq:psipoly},\ref{I-eq:triple}) where the portions of diagrams commuting past each other are distant in $\R$.  As discussed above, (\ref{I-eq:weyl1},\ref{I-eq:conjugate2},\ref{I-eq:weyl2},\ref{I-eq:psiconjugate})
imply that green lines can isotope past all crossings and dots and (\ref{I-eq:coweight2}) that green lines can merge by multiplying their
labels.  
\begin{equation*}\label{eq:antipodal-possible}\subeqn
    \begin{tikzpicture}[baseline,scale=.9]
        \draw[very thick] (-4,0) +(-1,-1) -- +(1,1);
\draw[very thick](-4,0) +(1,-1) -- +(-1,1);
\draw[weyl] (-4,.5) +(1.5,0) -- +(-1.5,0);
  \end{tikzpicture}=   \begin{tikzpicture}[baseline,scale=.9]
        \draw[very thick] (-4,0) +(-1,-1) -- +(1,1);
\draw[very thick](-4,0) +(1,-1) -- +(-1,1);
\draw[weyl] (-4,-.5) +(1.5,0) -- +(-1.5,0);
  \end{tikzpicture}\qquad \qquad 
   \begin{tikzpicture}[baseline,scale=.9]
  \draw[very thick](-1,0) +(0,-1) --  node
  [midway,circle,fill=black,inner sep=2pt]{}
  +(0,1);
  \draw[weyl] (-1,.5) +(0.5,0) -- +(-0.5,0);
\end{tikzpicture}=  \begin{tikzpicture}[baseline,scale=.9]
  \draw[very thick](-1,0) +(0,-1) --  node
  [midway,circle,fill=black,inner sep=2pt]{}
  +(0,1);
  \draw[weyl] (-1,-.5) +(0.5,0) -- +(-0.5,0);
\end{tikzpicture}
\end{equation*}
\begin{equation*}
      \begin{tikzpicture}[baseline,scale=.9]
      
  \draw[very thick] (-12,0) +(-1,-1) -- +(1,1);
 
  \draw[very thick, dashed](-12,0) +(1,-1) -- +(-1,1);
 
  \draw[weyl] (-12,.5) +(1.5,0) -- +(-1.5,0);
\end{tikzpicture}=      \begin{tikzpicture}[baseline,scale=.9]
      
  \draw[very thick] (-12,0) +(-1,-1) -- +(1,1);
 
  \draw[very thick, dashed](-12,0) +(1,-1) -- +(-1,1);
 
  \draw[weyl] (-12,-.5) +(1.5,0) -- +(-1.5,0);
\end{tikzpicture}\qquad 
\begin{tikzpicture}[baseline,scale=.9]

  \draw[very thick, dashed] (-8,0) +(-1,-1) -- +(1,1);
 
  \draw[very thick](-8,0) +(1,-1) -- +(-1,1);
  
  \draw[weyl] (-8,.5) +(1.5,0) -- +(-1.5,0);
\end{tikzpicture}=\begin{tikzpicture}[baseline,scale=.9]

  \draw[very thick, dashed] (-8,0) +(-1,-1) -- +(1,1);
 
  \draw[very thick](-8,0) +(1,-1) -- +(-1,1);
  
  \draw[weyl] (-8,-.5) +(1.5,0) -- +(-1.5,0);
\end{tikzpicture}
\end{equation*}

The relation (\ref{I-eq:dot-commute}) implies that dots can commute past strands with a
different label or their partners, or past all ghosts:
\begin{equation*}\subeqn\label{b-first-QH}
    \begin{tikzpicture}[scale=.9,baseline]
      \draw[very thick](-4,0) +(-1,-1) -- +(1,1) node[below,at start]
      {$i;n$}; \draw[very thick](-4,0) +(1,-1) -- +(-1,1) node[below,at
      start] {$j$}; \fill (-4.5,.5) circle (3pt);
\node at (-2,0){=}; \draw[very thick](0,0) +(-1,-1) -- +(1,1)
      node[below,at start] {$i;n$}; \draw[very thick](0,0) +(1,-1) --
      +(-1,1) node[below,at start] {$j$}; \fill (.5,-.5) circle (3pt);
      \node at (4,0){unless $i=j$};
    \end{tikzpicture}
  \end{equation*}
\begin{equation*}\subeqn\label{b-second-QH}
    \begin{tikzpicture}[scale=.9,baseline]
      \draw[very thick](-4,0) +(-1,-1) -- +(1,1) node[below,at start]
      {$i$}; \draw[very thick](-4,0) +(1,-1) -- +(-1,1) node[below,at
      start] {$j;n$}; \fill (-3.5,.5) circle (3pt);
\node at (-2,0){=}; \draw[very thick](0,0) +(-1,-1) -- +(1,1)
      node[below,at start] {$i$}; \draw[very thick](0,0) +(1,-1) --
      +(-1,1) node[below,at start] {$j;n$}; \fill (-.5,-.5) circle (3pt);
      \node at (4,0){unless $i=j$};
    \end{tikzpicture}
  \end{equation*}
  \begin{equation*}\subeqn\label{b-third-QH}
    \begin{tikzpicture}[scale=.9,baseline]
      \draw[very thick,dashed](-4,0) +(-1,-1) -- +(1,1) node[below,at start]
      {$e;n$}; \draw[very thick](-4,0) +(1,-1) -- +(-1,1) node[below,at
      start] {$i$}; \fill (-4.5,.5) circle (3pt);
\node at (-2,0){=}; \draw[very thick,dashed](0,0) +(-1,-1) -- +(1,1)
      node[below,at start] {$e;n$}; \draw[very thick](0,0) +(1,-1) --
      +(-1,1) node[below,at start] {$i$}; \fill (.5,-.5) circle (3pt);
    \end{tikzpicture}\qquad \qquad
    \begin{tikzpicture}[scale=.9,baseline]
      \draw[very thick](-4,0) +(-1,-1) -- +(1,1) node[below,at start]
      {$i$}; \draw[very thick,dashed](-4,0) +(1,-1) -- +(-1,1) node[below,at
      start] {$e;n$}; \fill (-3.5,.5) circle (3pt);
\node at (-2,0){=}; \draw[very thick](0,0) +(-1,-1) -- +(1,1)
      node[below,at start] {$i$}; \draw[very thick,dashed](0,0) +(1,-1) --
      +(-1,1) node[below,at start] {$e;n$}; \fill (-.5,-.5) circle (3pt);
    \end{tikzpicture}
  \end{equation*}
The previous relations have not depended on the weights
$ {b_{i,1}},\dots, b_{i, {w_i}}$ and $ {b_e}$ defined before.  
We can now use these to write out our version of the relation (\ref{I-eq:wall-cross1}):
      \begin{equation*}\subeqn\label{x-cost-1}
   \begin{tikzpicture}[very thick,baseline,scale=.9]
     \draw (-2.8,0)  +(0,-1) .. controls (-1.2,0) ..  +(0,1) node[below,at start]{$i;m$};
        \draw[dashed,red] (-2,0)  +(0,-1)--node[below,at start ]{$i$}  +(0,1);
   \end{tikzpicture}
 =
  \begin{tikzpicture}[very thick,baseline,scale=.9]
  \draw[dashed,red] (2.3,0)  +(0,-1) -- node[below,at start ]{$i;m$} +(0,1);
        \draw (1.5,0)  +(0,-1) -- +(0,1) node[below,at start]{$i$};
        \fill (1.5,0) circle (3pt);
 \end{tikzpicture} +h(b_{i,k}+m)  \begin{tikzpicture}[very thick,baseline,scale=.9]
  \draw[dashed,red] (2.3,0)  +(0,-1) -- node[below,at start ]{$i;m$} +(0,1);
        \draw (1.5,0)  +(0,-1) -- +(0,1) node[below,at start]{$i$};
 \end{tikzpicture}\qquad  \qquad   \begin{tikzpicture}[very thick,baseline,scale=.9]
     \draw (-2.8,0)  +(0,-1) .. controls (-1.2,0) ..  +(0,1) node[below,at start]{$i$};
        \draw[dashed,red] (-2,0)  +(0,-1)--node[below,at start ]{$j;m$}  +(0,1);
   \end{tikzpicture}=\begin{tikzpicture}[very thick,baseline,scale=.9]
     \draw (-2.8,0)  +(0,-1) --  +(0,1) node[below,at start]{$i$};
        \draw[dashed,red] (-2,0)  +(0,-1)--node[below,at start ]{$j;m$}  +(0,1);
   \end{tikzpicture}
 \end{equation*}
\begin{equation*}
    \subeqn\label{x-cost-2}
  \begin{tikzpicture}[very thick,baseline,scale=.9]
          \draw[dashed,red] (-2,0)  +(0,-1)-- node[below,at start ]{$i;m$} +(0,1);
  \draw (-1.2,0)  +(0,-1) .. controls (-2.8,0) ..  +(0,1) node[below,at start]{$i$};\end{tikzpicture}
           =
  \begin{tikzpicture}[very thick,baseline,scale=.9]
    \draw (2.5,0)  +(0,-1) -- +(0,1) node[below,at start]{$i$};
       \draw[dashed,red] (1.7,0)  +(0,-1) -- node[below,at start ]{$i;m$} +(0,1) ;
       \fill (2.5,0) circle (3pt);\end{tikzpicture} +h(b_{i,k}+m)  \begin{tikzpicture}[very thick,baseline,scale=.9]
    \draw (2.5,0)  +(0,-1) -- +(0,1) node[below,at start]{$i$};
       \draw[dashed,red] (1.7,0)  +(0,-1) -- node[below,at start]{$i;m$} +(0,1) ;
       \end{tikzpicture}\qquad  \qquad   \begin{tikzpicture}[very thick,baseline,scale=.9]
     \draw (-2,0)  +(0,-1) .. controls (-3.6,0) ..  +(0,1) node[below,at start]{$i$};
        \draw[dashed,red] (-2.8,0)  +(0,-1)--node[below,at start ]{$j;m$}  +(0,1);
   \end{tikzpicture}=\begin{tikzpicture}[very thick,baseline,scale=.9]
     \draw (-2,0)  +(0,-1) --  +(0,1) node[below,at start]{$i$};
        \draw[dashed,red] (-2.8,0)  +(0,-1)--node[below,at start ]{$j;m$}  +(0,1);
   \end{tikzpicture}
\end{equation*}
If $e\colon j\to i$, then we have:
   \begin{equation*}\subeqn\label{b-black-bigon1}
      \begin{tikzpicture}[very thick,scale=.65,baseline]
      \draw(-2.8,0) +(0,-1) .. controls (-1.2,0) ..  +(0,1)
      node[below,at start]{$i$}; 
\end{tikzpicture}\quad   \begin{tikzpicture}[very thick,scale=.65,baseline]
      \draw[dashed] (-2.8,0) +(0,-1) .. controls (-1.2,0) ..  +(0,1)
      node[below,at start,scale=.8]{$e;m$}; \draw (-1.2,0) +(0,-1) .. controls
(-2.8,0) ..  +(0,1) node[below,at start]{$j$};
\end{tikzpicture}
=    \begin{tikzpicture}[very thick,scale=.65,baseline]
      \draw(-2.8,0) +(0,-1) -- node[midway,fill=black, inner sep=2pt, circle]{} +(0,1)
      node[below,at start]{$i$};      \draw[dashed] (-.8,0) +(0,-1)-- +(0,1)
      node[below,at start,scale=.75]{$e;m$}; \draw (.2,0) +(0,-1) --+(0,1) node[below,at start]{$j$};
\end{tikzpicture}-\begin{tikzpicture}[very thick,scale=.65,baseline]
      \draw(-2.8,0) +(0,-1) -- +(0,1)
      node[below,at start]{$i$}; 
      \draw[dashed] (-.8,0) +(0,-1)-- +(0,1)
      node[below,at start,scale=.75]{$e;m$}; \draw (.2,0) +(0,-1) --node[midway,fill=black, inner sep=2pt, circle]{}+(0,1) node[below,at start]{$j$};
\end{tikzpicture}+h(b_e-m)\begin{tikzpicture}[very thick,scale=.65,baseline]
      \draw(-2.8,0) +(0,-1) -- +(0,1)
      node[below,at start]{$i$}; 
      \draw[dashed] (-.8,0) +(0,-1)-- +(0,1)
      node[below,at start,scale=.75]{$e;m$}; \draw (.2,0) +(0,-1) --+(0,1) node[below,at start]{$j$};
\end{tikzpicture}
\end{equation*}
   \begin{equation*}\subeqn\label{b-black-bigon2}
      \begin{tikzpicture}[very thick,scale=.65,baseline]
      \draw (-1.2,0) +(0,-1) .. controls
(-2.8,0) ..  +(0,1) node[below,at start]{$i$};
\end{tikzpicture}\quad   \begin{tikzpicture}[very thick,scale=.65,baseline]
      \draw[dashed] (-1.2,0) +(0,-1) .. controls (-2.8,0) ..  +(0,1)    node[below,at start,scale=.75]{$e;m$}; 
      \draw (-2.8,0) +(0,-1) .. controls (-1.2,0) ..  +(0,1) node[below,at start]{$j$};
\end{tikzpicture}
=   \begin{tikzpicture}[very thick,scale=.65,baseline,xscale=.9]
      \draw(-1.8,0) +(0,-1) -- node[midway,fill=black, inner sep=2pt, circle]{} +(0,1)
      node[below,at start]{$i$}; 
\draw[dashed] (.2,0) +(0,-1)-- +(0,1)
      node[below,at start,scale=.75]{$e;m$}; \draw (-.8,0) +(0,-1) --+(0,1) node[below,at start]{$j$};
\end{tikzpicture}-\begin{tikzpicture}[very thick,scale=.65,baseline,xscale=.9]
      \draw(-1.8,0) +(0,-1) -- +(0,1)
      node[below,at start]{$i$}; 
\draw[dashed](.2,0) +(0,-1)-- +(0,1)
      node[below,at start,scale=.75]{$e;m$}; \draw  (-.8,0)+(0,-1) --node[midway,fill=black, inner sep=2pt, circle]{}+(0,1) node[below,at start]{$j$};
\end{tikzpicture}+h(b_e-m)\begin{tikzpicture}[very thick,scale=.65,baseline,xscale=.9]
      \draw(-1.8,0) +(0,-1) -- +(0,1)
      node[below,at start]{$i$}; 
\draw[dashed](.2,0) +(0,-1)-- +(0,1)
      node[below,at start,scale=.75]{$e;m$}; \draw  (-.8,0)+(0,-1) --+(0,1) node[below,at start]{$j$};
\end{tikzpicture}
  \end{equation*}
  Otherwise, we have:
     \begin{equation*}\subeqn\label{b-black-bigon3}
      \begin{tikzpicture}[very thick,scale=.65,baseline]
      \draw(-2.8,0) +(0,-1) .. controls (-1.2,0) ..  +(0,1)
      node[below,at start]{$i$}; 
\end{tikzpicture}\quad   \begin{tikzpicture}[very thick,scale=.65,baseline]
      \draw[dashed] (-2.8,0) +(0,-1) .. controls (-1.2,0) ..  +(0,1)
      node[below,at start,scale=.8]{$e;m$}; \draw (-1.2,0) +(0,-1) .. controls
(-2.8,0) ..  +(0,1) node[below,at start]{$j$};
\end{tikzpicture}
=  \begin{tikzpicture}[very thick,scale=.65,baseline]
      \draw(-2.8,0) +(0,-1) -- +(0,1)
      node[below,at start]{$i$}; 
      \draw[dashed] (-.8,0) +(0,-1)-- +(0,1)
      node[below,at start,scale=.75]{$e;m$}; \draw (.2,0) +(0,-1) --+(0,1) node[below,at start]{$j$};
\end{tikzpicture}\qquad 
      \begin{tikzpicture}[very thick,scale=.65,baseline]
      \draw (-1.2,0) +(0,-1) .. controls
(-2.8,0) ..  +(0,1) node[below,at start]{$i$};
\end{tikzpicture}\quad  \quad  \begin{tikzpicture}[very thick,scale=.65,baseline]
      \draw[dashed] (-1.2,0) +(0,-1) .. controls (-2.8,0) ..  +(0,1)    node[below,at start,scale=.75]{$e;m$}; 
      \draw (-2.8,0) +(0,-1) .. controls (-1.2,0) ..  +(0,1) node[below,at start]{$j$};
\end{tikzpicture}
=  \begin{tikzpicture}[very thick,scale=.65,baseline,xscale=.9]
      \draw(-1.8,0) +(0,-1) -- +(0,1)
      node[below,at start]{$i$}; 
\draw[dashed](.2,0) +(0,-1)-- +(0,1)
      node[below,at start,scale=.75]{$e;m$}; \draw  (-.8,0)+(0,-1) --+(0,1) node[below,at start]{$j$};
\end{tikzpicture}
  \end{equation*}
Note that the equations above are written assuming that $m>0$, but they are equally valid if $m<0$,
with the requisite reordering of strands.
The relation 
(\ref{I-eq:psi2}) implies that:
  \begin{equation*}\subeqn\label{b-psi2}
    \begin{tikzpicture}[very thick,scale=.9,baseline]
      \draw (-2.8,0) +(0,-1) .. controls (-1.2,0) ..  +(0,1)
      node[below,at start]{$i;m$}; \draw (-1.2,0) +(0,-1) .. controls
      (-2.8,0) ..  +(0,1) node[below,at start]{$i$}; \node at (-.5,0)
      {=}; \node at (0.4,0) {$0$};
    \end{tikzpicture}\qquad \qquad    \begin{tikzpicture}[very thick,scale=.9,baseline]
      \draw (-2.8,0) +(0,-1) .. controls (-1.2,0) ..  +(0,1)
      node[below,at start]{$i$}; \draw(-1.2,0) +(0,-1) .. controls
      (-2.8,0) ..  +(0,1) node[below,at start]{$i;m$}; \node at (-.5,0)
      {=}; \node at (0.4,0) {$0$};
    \end{tikzpicture}
      \end{equation*}
The relation 
(\ref{I-eq:psipoly}) is equivalent to isotopy and  \begin{equation*}\subeqn\label{b-nilHecke-1}
    \begin{tikzpicture}[scale=.9,baseline]
      \draw[very thick](-4,0) +(-1,-1) -- +(1,1) node[below,at start,scale=.8]
      {$i;m$}; \draw[very thick](-4,0) +(1,-1) -- +(-1,1) node[below,at
      start] {$i$}; \fill (-4.5,.5) circle (3pt);
\node at (-2,0){$-$}; \draw[very thick](0,0) +(-1,-1) -- +(1,1)
      node[below,at start,scale=.8]
      {$i;m$}; \draw[very thick](0,0) +(1,-1) --
      +(-1,1) node[below,at start] {$i$}; \fill (.5,-.5) circle (3pt);
     \node at (2,0){$=$};  \draw[very thick](4,0) +(-1,-1) -- +(-1,1)
      node[below,at start,scale=.8]
      {$i;m$}; \draw[very thick](4,0) +(0,-1) --
      +(0,1) node[below,at start] {$i$};
      \draw[weyl] (4,0) +(.5,0) -- +(-1.5,0) node[at start, right,green!50!black]{$s$} ;
    \end{tikzpicture}
  \end{equation*}
 \begin{equation*}\subeqn\label{b-nilHecke-2}
    \begin{tikzpicture}[scale=.9,baseline]
      \draw[very thick](-4,0) +(-1,-1) -- +(1,1) node[below,at
      start] {$i$}; \draw[very thick](-4,0) +(1,-1) -- +(-1,1) node[below,at start,scale=.8]
      {$i;m$}; \fill (-4.5,-.5) circle (3pt);
      \node at (-2,0){$-$}; \draw[very thick](0,0) +(-1,-1) -- +(1,1)
      node[below,at
      start] {$i$}; \draw[very thick](0,0) +(1,-1) --
      +(-1,1) node[below,at start,scale=.8]
      {$i;m$}; \fill (.5,.5) circle (3pt);
      \node at (2,0){$=$};  \draw[very thick](4,0) +(-1,-1) -- +(-1,1)
      node[below,at
      start] {$i$}; \draw[very thick](4,0) +(0,-1) --
      +(0,1) node[below,at start,scale=.8]
      {$i;m$}; 
      \draw[weyl] (4,0) +(.5,0) -- +(-1.5,0) node[at start, right,green!50!black]{$s$} ;
    \end{tikzpicture}
  \end{equation*}
  with $s$ denoting the unique reflection in the affine Weyl group
  switching the top and bottom labels of the diagram.

  Finally, the codimension 2 relations show how to relate the two
resolutions of a triple point. The correct relation depends on the
number of strands with the same label going through the triple
point: if all three have the same label, then we use (\ref{I-eq:psi}),
if two have the same label, we use
(\ref{I-eq:triple}) and if there is no such pair, then
(\ref{I-eq:wall-cross1}).  \excise{ Any triple point involving only partners
  \begin{equation*}\subeqn\label{b-triple-coxeter}
      \begin{tikzpicture}[very thick,scale=.8,baseline]  \draw (1,0) +(1,-1) .. controls
      (2,0) .. +(-1,1)
      node[below,at start,scale=.8]{$k;m$}; \draw (1,0) +(-1,-1) .. controls
      (2,0) .. +(1,1)
      node[below,at start]{$i$}; \draw (1,0) +(0,-1) -- node[below, at start,scale=.8]{$e;m$}+(0,1); 
    \end{tikzpicture} =\begin{tikzpicture}[very thick,scale=.8,baseline]
      \draw (-3,0) +(1,-1) .. controls (-4,0) .. +(-1,1) node[below,at start,scale=.8]{$k;m$}; \draw
      (-3,0) +(-1,-1) .. controls (-4,0) .. +(1,1) node[below,at start]{$i$}; \draw[dashed]
      (-3,0) +(0,-1)--  node[below, at start,scale=.8]{$e;m$}+(0,1);
       \end{tikzpicture} 
  \end{equation*}
 
  Finally, the codimension 2 relation
 (\ref{eq:triple}) implies a number of different relations relating the
 two different ways of resolving a triple point
 for all 
 ghosts, we have the equation below and its reflection through a
 vertical line:
  \begin{equation*}\subeqn\label{b-triple-extra-dumb}
    \begin{tikzpicture}[very thick,scale=.9,baseline]
      \draw (-3,0) +(-1,-1) to[out=90,in=-150] node[below,at start]{$e;m$} +(1,1) ; \draw
      (-3,0) +(-2,-1) -- +(2,1) node[below,at start]{$i;m$}; \draw[dashed]
      (-3,0) +(0,-1)--  +(0,1); \node at (-1,0) {=}; \draw (1,0) +(-1,-1) to [out=30,in=-90] node[below,at start]{$e;m$} +(1,1)
      ; \draw
      (1,0) +(-2,-1) -- +(2,1) node[below,at start]{$i;m$}; \draw[dashed] (1,0) +(0,-1) -- +(0,1); 
    \end{tikzpicture}
  \end{equation*}}
These imply we can isotope through any triple point unless it involves
exactly 
\begin{enumerate}
	\item two partners with the label $i\in {\vertex}$ on the outside, and a red strand with label $i$:
 \begin{equation*}\subeqn\label{b-frame-triple-smart}
      \begin{tikzpicture}[very thick,scale=.8,baseline]  \draw (1,0) +(1,-1) .. controls
      (2,0) .. +(-1,1)
      node[below,at start,scale=.8]{$i;n$}; \draw (1,0) +(-1,-1) .. controls
      (2,0) .. +(1,1)
      node[below,at start]{$i$}; \draw[dashed,red] (1,0) +(0,-1) -- node[above, at end,scale=.8]{$ i;m$}+(0,1);  \draw[weyl] (1,.7) +(1.5,0) -- +(-1.5,0) node[at start, right]{$s$} ;
    \end{tikzpicture} - \begin{tikzpicture}[very thick,scale=.8,baseline]
      \draw (-3,0) +(1,-1) .. controls (-4,0) .. +(-1,1) node[below,at start,scale=.8]{$i;n$}; \draw
      (-3,0) +(-1,-1) .. controls (-4,0) .. +(1,1) node[below,at start]{$i$}; \draw[dashed,red]
      (-3,0) +(0,-1)--  node[above, at end,scale=.8]{$ i;m$}+(0,1);  \draw[weyl] (-3,.7) +(1.5,0) -- +(-1.5,0) node[at start, right]{$s$} ;
       \end{tikzpicture} 
  =\begin{tikzpicture}[very thick,scale=.8,baseline]       \draw (0,0)
      +(1,-1) -- +(1,1) node[below,at start,scale=.8]{$i;n$};  \draw (0,0)
      +(-1,-1) -- +(-1,1) node[below,at start]{$i$}; \draw[dashed,red] (0,0)
      +(0,-1) -- node[above, at end,scale=.8]{$ i;m$}+(0,1); 
    \end{tikzpicture}
  \end{equation*}
  \item two partners with the label $i\in {\vertex}$ on the outside, and a ghost for $e$ with $t(e)=i$ in the middle:
      \begin{equation*}
       \subeqn\label{b-triple-smart}
      \begin{tikzpicture}[very thick,scale=.9,baseline,xscale=.9]  \draw (1,0) +(1,-1) .. controls
      (2,0) .. +(-1,1)
      node[below,at start,scale=.8]{$i;n$}; \draw (1,0) +(-1,-1) .. controls
      (2,0) .. +(1,1)
      node[below,at start]{$i$}; \draw[dashed] (1,0) +(0,-1) -- node[above, at end,scale=.8]{$e;m$}+(0,1);    \draw[weyl] (1,.4) +(1.5,0) -- +(-1.5,0) node[at start, right]{$s$} ;
    \end{tikzpicture}   
    - \begin{tikzpicture}[very thick,scale=.9,baseline,xscale=.9]
      \draw (-3,0) +(1,-1) .. controls (-4,0) .. +(-1,1) node[below,at start,scale=.8]{$i;n$}; \draw
      (-3,0) +(-1,-1) .. controls (-4,0) .. +(1,1) node[below,at start]{$i$}; \draw[dashed]
      (-3,0) +(0,-1)--  node[above, at end,scale=.8]{$e;m$}+(0,1);   \draw[weyl] (-3,.4) +(1.5,0) -- +(-1.5,0) node[at start, right]{$s$} ;
       \end{tikzpicture} 
  =\begin{tikzpicture}[very thick,xscale=.9,baseline,xscale=.9]       \draw (0,0)
      +(1,-1) -- +(1,1) node[below,at start,scale=.8]{$i;n$}; \draw (0,0)
      +(-1,-1) -- +(-1,1) node[below,at start]{$i$}; \draw[dashed] (0,0)
      +(0,-1) -- node[above, at end,scale=.8]{$e;m$}+(0,1); 
    \end{tikzpicture} 
  \end{equation*}
\item two ghosts for $e$ with $t(e)=i$ on the outside, and a corporeal strand with label $i$ in the middle:
   \begin{equation*}
       \subeqn\label{b-triple-smart2}
      \begin{tikzpicture}[very thick,scale=.9,baseline,xscale=.9]  \draw[dashed] (1,0) +(1,-1) .. controls
      (2,0) .. +(-1,1)
      node[below,at start,scale=.8]{$e;n$}; \draw[dashed] (1,0) +(-1,-1) .. controls
      (2,0) .. +(1,1)
      node[below,at start,scale=.8]{$e;m$}; \draw (1,0) +(0,-1) -- node[above, at end]{$i$}+(0,1);    \draw[weyl] (1,.4) +(1.5,0) -- +(-1.5,0) node[at start, right]{$s$} ;
    \end{tikzpicture}   
    - \begin{tikzpicture}[very thick,scale=.9,baseline,xscale=.9]
      \draw[dashed] (-3,0) +(1,-1) .. controls (-4,0) .. +(-1,1) node[below,at start,scale=.8]{$e;n$}; \draw[dashed]
      (-3,0) +(-1,-1) .. controls (-4,0) .. +(1,1) node[below,at start,scale=.8]{$e;m$}; \draw
      (-3,0) +(0,-1)--  node[above, at end]{$i$}+(0,1);   \draw[weyl] (-3,.4) +(1.5,0) -- +(-1.5,0) node[at start, right]{$s$} ;
       \end{tikzpicture} 
  =\begin{tikzpicture}[very thick,xscale=.9,baseline,xscale=.9]       \draw[dashed]  (0,0)
      +(1,-1) -- +(1,1) node[below,at start,scale=.8]{$e;n$}; \draw[dashed]  (0,0)
      +(-1,-1) -- +(-1,1) node[below,at start,scale=.8]{$e;m$}; \draw(0,0)
      +(0,-1) -- node[above, at end]{$i$}+(0,1); 
    \end{tikzpicture} 
  \end{equation*}
\end{enumerate} 
  In each of the diagrams above, $s$ denotes the unique reflection in the affine Weyl group
  making the top and bottom match.  

\begin{lemma}\label{lem:unrolled-B}
    Given $\eta,\eta'\in \ft_{\hat\beta}$ generic, the Hom space
    $\Hom_{\scrB^+}(\eta,\eta')$ is spanned by the morphisms
    $ {\mathbbm{r}_D}$ for unrolled diagrams $D$ with top $\eta'$ and
    bottom $\eta$, and all relations between these morphisms are induced by the local relations (\ref{eq:dot-migration}--\ref{b-triple-smart2}).
  \end{lemma}
  \begin{proof}
    We have justified in each individual case why the relations
    (\ref{eq:dot-migration}--\ref{b-triple-smart2}) hold.  Thus, we have a
    map from the formal span of unrolled diagrams modulo these
    relations to $\Hom_{\scrB^+}(\eta,\eta')$.  This is surjective
    because the generating morphisms of the category $\scrB^+$ are
   spanned as a left or right module over the dots by paths of the form $\mathbbm{r}_D$ for different $D$.  On the other hand, the relations
    (\ref{eq:dot-migration}--\ref{b-triple-smart2}) suffice to write any
     $\mathbbm{r}_D$ as a sum of diagrams corresponding to a reduced word in
    $\What$ with all dots and green lines at the bottom, and to
    relate any two reduced words for $w\in \widehat{W}$ modulo the
    diagrams for shorter elements of $\widehat{W}$.  Thus, we find
    that the unrolled diagrams corresponding to the basis of
    \cite[Th. 3.13]{websterKoszulDuality2019} are a spanning set of this quotient.  This
    is only possible if the map is injective as well.
  \end{proof}
By Theorem \ref{I-thm:BFN-pres}, we have a faithful representation of the category $\scrB^+$ on a sum of polynomial rings $\C[z_{i,k}]_{k\in [1,v_i]}\cdot\epsilon_{\eta}$ 
given by the formulas (\ref{I-eq:ract}--\ref{I-eq:muact}).   For the ease of the reader, we translate these into the diagrammatic representation.  The result is exactly the representation of the weighted KLR algebra defined in \cite[Prop. 2.7]{WebwKLR}, extended to the affine case in the only way compatible with \eqref{eq:dot-migration}. It's most convenient to represent the polynomial ring corresponding to $\eta$ as dots on the corporeal strands, positioned according to $\eta$ (and in particular, with no crossings).  A diagram with no crossings acts by the usual multiplication of diagrams, followed by an isotopy to make the strands straight vertical.  For all other diagrams, we only need to specify the action on the dotless generator $\epsilon_{\eta}$ of the polynomial ring.  
\begin{itemize}
\item A crossing of a corporeal strand with a partner with the same label acts by zero on $\epsilon_{\eta}$.  Combining with the equations (\ref{b-nilHecke-1}--\ref{b-nilHecke-2}), we find that it acts by a shifted divided difference operator on a more general polynomial.  For example, if $a>b$, we have:\newseq 
 \begin{equation*}\subeqn\label{b-first-action}
    \begin{tikzpicture}[scale=.9,baseline]
      \draw[very thick](-4,0) +(-1,-1) -- +(1,1) node[below,at
      start] {$i$}; \draw[very thick](-4,0) +(1,-1) -- +(-1,1) node[below,at start,scale=.8]
      {$i;m$}; 
          \end{tikzpicture}\:\:\star\:\:
        \begin{tikzpicture}[scale=.9,baseline]
\draw[very thick](3.3,0) +(0,-1) -- +(0,1)
      node[below,at
      start] {$i$}; 
      \fill (3.3,-.5) circle (3pt);
      \node[xshift=-8pt,yshift=3pt,scale=.85] at (3.3,-.5) {$a$};
      \fill (4,-.5) circle (3pt);
      \node[xshift=8pt,yshift=3pt,scale=.85] at (4,-.5) {$b$};
      \draw[very thick](4,0) +(0,-1) --
      +(0,1) node[below,at start,scale=.8]
      {$i;m$}; 
    \end{tikzpicture}=        \begin{tikzpicture}[scale=.9,baseline]
\draw[very thick](3.3,0) +(0,-1) -- +(0,1)
      node[below,at start,scale=.8]
      {$i;m$}; 
      \fill (3.3,-.5) circle (3pt);
      \node[xshift=-16pt,yshift=3pt,scale=.85] at (3.3,-.5) {$a-1$};
      \fill (4,-.5) circle (3pt);
      \node[xshift=8pt,yshift=3pt,scale=.85] at (4,-.5) {$b$};
      \draw[very thick](4,0) +(0,-1) --
      +(0,1) node[below,at
      start] {$i$}; 
    \end{tikzpicture}+       \begin{tikzpicture}[scale=.9,baseline]
\draw[very thick](3.3,0) +(0,-1) -- +(0,1)
      node[below,at start,scale=.8]
      {$i;m$}; 
      \fill (3.3,-.5) circle (3pt);
      \node[xshift=-16pt,yshift=3pt,scale=.85] at (3.3,-.5) {$a-2$};
      \fill (4,-.5) circle (3pt);
      \node[xshift=16pt,yshift=3pt,scale=.85] at (4,-.5) {$b+1$};
      \draw[very thick](4,0) +(0,-1) --
      +(0,1) node[below,at
      start] {$i$}; 
    \end{tikzpicture}+\cdots +       \begin{tikzpicture}[scale=.9,baseline]
\draw[very thick](3.3,0) +(0,-1) -- +(0,1)
      node[below,at start,scale=.8]
      {$i;m$}; 
      \fill (3.3,-.5) circle (3pt);
      \node[xshift=-8pt,yshift=3pt,scale=.85] at (3.3,-.5) {$b$};
      \fill (4,-.5) circle (3pt);
      \node[xshift=16pt,yshift=3pt,scale=.85] at (4,-.5) {$a-1$};
      \draw[very thick](4,0) +(0,-1) --
      +(0,1) node[below,at
      start] {$i$}; 
    \end{tikzpicture}
  \end{equation*}
  Note that the presence of a partner strand here can also be written as a shift by (\ref{eq:dot-migration}), but not conveniently in a diagram as above.
\item A crossing of a corporeal strand with a partner with a different label sends the $\epsilon_{\eta}$ to $\epsilon_{\eta'}$, so we have:
 \begin{equation*}\subeqn\label{b-nilHecke-action}
    \begin{tikzpicture}[scale=.9,baseline]
      \draw[very thick](-4,0) +(-1,-1) -- +(1,1) node[below,at
      start] {$i$}; \draw[very thick](-4,0) +(1,-1) -- +(-1,1) node[below,at start,scale=.8]
      {$j;m$}; 
          \end{tikzpicture}\:\:\star\:\:
        \begin{tikzpicture}[scale=.9,baseline]
\draw[very thick](3.3,0) +(0,-1) -- +(0,1)
      node[below,at
      start] {$i$}; 
      \fill (3.3,-.5) circle (3pt);
      \node[xshift=-8pt,yshift=3pt,scale=.85] at (3.3,-.5) {$a$};
      \fill (4,-.5) circle (3pt);
      \node[xshift=8pt,yshift=3pt,scale=.85] at (4,-.5) {$b$};
      \draw[very thick](4,0) +(0,-1) --
      +(0,1) node[below,at start,scale=.8]
      {$j;m$}; 
    \end{tikzpicture}=        \begin{tikzpicture}[scale=.9,baseline]
\draw[very thick](3.3,0) +(0,-1) -- +(0,1)
      node[below,at start,scale=.8]
      {$j;m$}; 
      \fill (3.3,-.5) circle (3pt);
      \node[xshift=-8pt,yshift=3pt,scale=.85] at (3.3,-.5) {$b$};
      \fill (4,-.5) circle (3pt);
      \node[xshift=8pt,yshift=3pt,scale=.85] at (4,-.5) {$a$};
      \draw[very thick](4,0) +(0,-1) --
      +(0,1) node[below,at
      start] {$i$}; 
    \end{tikzpicture}
  \end{equation*}
\item A crossing of a corporeal moving NE over a ghost gives multiplication by the right-hand side of (\ref{x-cost-2}) or (\ref{b-black-bigon1}):
\begin{equation*}
    \subeqn
  \begin{tikzpicture}[very thick,baseline,scale=.9]
          \draw[dashed,red] (-2,0)  +(0,-1)-- node[below,at start]{$ i$} +(0,1);
  \draw (-2,0)  +(-1,-1) -- +(1,1) node[below,at start]{$i$};\end{tikzpicture}\:\:\star\:\:    \begin{tikzpicture}[very thick,baseline,scale=.9]
          \draw[dashed,red] (-2,0)  +(0,-1)-- node[below,at start]{$ i$} +(0,1);
  \draw (-2,0)  +(-1,-1) -- +(-1,1) node[below,at start]{$i$};\end{tikzpicture}
           =
  \begin{tikzpicture}[very thick,baseline,scale=.9]
    \draw (2.5,0)  +(0,-1) -- +(0,1) node[below,at start]{$i$};
       \draw[dashed,red] (1.7,0)  +(0,-1) -- node[below,at start]{$ i$} +(0,1) ;
       \fill (2.5,0) circle (3pt);\end{tikzpicture} +h(b_{i,k}+m)  \begin{tikzpicture}[very thick,baseline,scale=.9]
    \draw (2.5,0)  +(0,-1) -- +(0,1) node[below,at start]{$i$};
       \draw[dashed,red] (1.7,0)  +(0,-1) -- node[below,at start]{$ i$} +(0,1) ;
       \end{tikzpicture}
\end{equation*}
\begin{equation*}
    \subeqn
  \begin{tikzpicture}[very thick,baseline,scale=.9]
          \draw[dashed,red] (-2,0)  +(0,-1)-- node[below,at start]{$ j$} +(0,1);
  \draw (-2,0)  +(-1,-1) -- +(1,1) node[below,at start]{$i$};\end{tikzpicture}\:\:\star\:\:    \begin{tikzpicture}[very thick,baseline,scale=.9]
          \draw[dashed,red] (-2,0)  +(0,-1)-- node[below,at start]{$ j$} +(0,1);
  \draw (-2,0)  +(-1,-1) -- +(-1,1) node[below,at start]{$i$};\end{tikzpicture}
           =
  \begin{tikzpicture}[very thick,baseline,scale=.9]
    \draw (2.5,0)  +(0,-1) -- +(0,1) node[below,at start]{$i$};
       \draw[dashed,red] (1.7,0)  +(0,-1) -- node[below,at start]{$ j$} +(0,1) ;
      \end{tikzpicture} \end{equation*}
      If $e\colon j\to i$, then:
  \begin{multline*}\subeqn
      \begin{tikzpicture}[very thick,scale=.65,baseline]
      \draw(-2.8,0) +(1,-1) --  +(-1,1)
      node[below,at start]{$i$}; 
\end{tikzpicture}\quad   \begin{tikzpicture}[very thick,scale=.65,baseline]
      \draw[dashed] (-2.8,0) +(1,-1) --  +(-1,1)
      node[below,at start,scale=.8]{$e;m$}; \draw (-2.8,0) +(-1,-1) -- +(1,1) node[below,at start]{$j$};
\end{tikzpicture}\:\:\star\:\:       \begin{tikzpicture}[very thick,scale=.65,baseline]
      \draw(-2.8,0) +(1,-1) --  +(1,1)
      node[below,at start]{$i$}; 
\end{tikzpicture}\quad   \begin{tikzpicture}[very thick,scale=.65,baseline]
      \draw[dashed] (-2.8,0) +(1,-1) --  +(1,1)
      node[below,at start,scale=.8]{$e;m$}; \draw (-2.8,0) +(-1,-1) -- +(-1,1) node[below,at start]{$j$};
\end{tikzpicture}\\
=    \begin{tikzpicture}[very thick,scale=.65,baseline]
      \draw(-2.8,0) +(0,-1) -- node[midway,fill=black, inner sep=2pt, circle]{} +(0,1)
      node[below,at start]{$i$};      \draw[dashed] (-.8,0) +(0,-1)-- +(0,1)
      node[below,at start,scale=.75]{$e;m$}; \draw (.2,0) +(0,-1) --+(0,1) node[below,at start]{$j$};
\end{tikzpicture}-\begin{tikzpicture}[very thick,scale=.65,baseline]
      \draw(-2.8,0) +(0,-1) -- +(0,1)
      node[below,at start]{$i$}; 
      \draw[dashed] (-.8,0) +(0,-1)-- +(0,1)
      node[below,at start,scale=.75]{$e;m$}; \draw (.2,0) +(0,-1) --node[midway,fill=black, inner sep=2pt, circle]{}+(0,1) node[below,at start]{$j$};
\end{tikzpicture}+h(b_e-m)\begin{tikzpicture}[very thick,scale=.65,baseline]
      \draw(-2.8,0) +(0,-1) -- +(0,1)
      node[below,at start]{$i$}; 
      \draw[dashed] (-.8,0) +(0,-1)-- +(0,1)
      node[below,at start,scale=.75]{$e;m$}; \draw (.2,0) +(0,-1) --+(0,1) node[below,at start]{$j$};
\end{tikzpicture}
\end{multline*}
Otherwise, we have that:
  \begin{equation*}\subeqn
      \begin{tikzpicture}[very thick,scale=.65,baseline]
      \draw(-2.8,0) +(1,-1) --  +(-1,1)
      node[below,at start]{$i$}; 
\end{tikzpicture}\quad   \begin{tikzpicture}[very thick,scale=.65,baseline]
      \draw[dashed] (-2.8,0) +(1,-1) --  +(-1,1)
      node[below,at start,scale=.8]{$e;m$}; \draw (-2.8,0) +(-1,-1) -- +(1,1) node[below,at start]{$j$};
\end{tikzpicture}\:\:\star\:\:       \begin{tikzpicture}[very thick,scale=.65,baseline]
      \draw(-2.8,0) +(1,-1) --  +(1,1)
      node[below,at start]{$i$}; 
\end{tikzpicture}\quad   \begin{tikzpicture}[very thick,scale=.65,baseline]
      \draw[dashed] (-2.8,0) +(1,-1) --  +(1,1)
      node[below,at start,scale=.8]{$e;m$}; \draw (-2.8,0) +(-1,-1) -- +(-1,1) node[below,at start]{$j$};
\end{tikzpicture}
=  \begin{tikzpicture}[very thick,scale=.65,baseline]
      \draw(-2.8,0) +(0,-1) -- +(0,1)
      node[below,at start]{$i$}; 
      \draw[dashed] (-.8,0) +(0,-1)-- +(0,1)
      node[below,at start,scale=.75]{$e;m$}; \draw (.2,0) +(0,-1) --+(0,1) node[below,at start]{$j$};
\end{tikzpicture}
\end{equation*}
\item A crossing of a corporeal moving NW over a ghost sends the identity to the identity: 
\begin{equation*}
    \subeqn
  \begin{tikzpicture}[very thick,baseline,scale=.9]
          \draw[dashed,red] (-2,0)  +(0,-1)-- node[below,at start]{$ i$} +(0,1);
  \draw (-2,0)  +(1,-1) -- +(-1,1) node[below,at start]{$i$};\end{tikzpicture}\:\:\star\:\:    \begin{tikzpicture}[very thick,baseline,scale=.9]
          \draw[dashed,red] (-2,0)  +(-1,-1)-- node[below,at start]{$ i$} +(-1,1);
  \draw (-2,0)  +(0,-1) -- +(0,1) node[below,at start]{$i$};\end{tikzpicture}
           =\begin{tikzpicture}[very thick,baseline,scale=.9]
    \draw (1.7,0)  +(0,-1) -- +(0,1) node[below,at start]{$i$};
       \draw[dashed,red] (2.5,0)  +(0,-1) -- node[below,at start]{$ i$} +(0,1) ;
       \end{tikzpicture}
\end{equation*}
  \begin{equation*}\subeqn\label{b-last-action}
      \begin{tikzpicture}[very thick,scale=.65,baseline]
      \draw(-2.8,0) +(-1,-1) --  +(1,1)
      node[below,at start]{$i$}; 
\end{tikzpicture}\quad   \begin{tikzpicture}[very thick,scale=.65,baseline]
      \draw[dashed] (-2.8,0) +(-1,-1) --  +(1,1)
      node[below,at start,scale=.8]{$e;m$}; \draw (-2.8,0) +(1,-1) -- +(-1,1) node[below,at start]{$j$};
\end{tikzpicture}\:\:\star\:\:     \begin{tikzpicture}[very thick,scale=.65,baseline]
      \draw(-2.8,0) +(0,-1) -- +(0,1)
      node[below,at start]{$i$}; 
      \draw[dashed] (-.8,0) +(0,-1)-- +(0,1)
      node[below,at start,scale=.75]{$e;m$}; \draw (.2,0) +(0,-1) --+(0,1) node[below,at start]{$j$};
\end{tikzpicture}=  \begin{tikzpicture}[very thick,scale=.65,baseline]
      \draw(-2.8,0) +(1,-1) --  +(1,1)
      node[below,at start]{$i$}; 
\end{tikzpicture}\quad   \begin{tikzpicture}[very thick,scale=.65,baseline]
      \draw[dashed] (-2.8,0) +(1,-1) --  +(1,1)
      node[below,at start,scale=.8]{$e;m$}; \draw (-2.8,0) +(-1,-1) -- +(-1,1) node[below,at start]{$j$};
\end{tikzpicture}
\end{equation*}
\end{itemize}
As stated above, Theorem \ref{I-thm:BFN-pres} implies that:
\begin{corollary}\label{cor:faithful-action}
	The rules above define a faithful representation of the category $\scrB_{\hat\beta}$.
\end{corollary}

\subsection{Quantum cylindrical KLRW algebras}
\label{sec:rolled-diagrams}

  The reader may have noticed that these diagrams are quite
  difficult to draw and interpret, but there is a symmetry that we
  have not exploited, the action of the extended affine Weyl group.
  The quotient of $\prod_i\R^{ {v_i}}$ by the extended Weyl
  group $\What$ is
  given by the space $\prod_i (\R/\Z)^{v_i}/\Sigma_{v_i}$, which we can
  interpret as the moduli space of multisubsets of the circle $\R/\Z$
  labeled with elements of ${\vertex}$, such that $v_i$ elements have label $i\in {\vertex}$.
  Thus the path $\pi_{*,*}\colon [0,1]\to \prod_i\R^{v_i}$ composed
  with the projection $\prod_i\R^{v_i}\to \prod_i
  (\R/\Z)^{v_i}/\Sigma_{v_i}$ can be thought of as a path in 
  this moduli space.

In terms of diagrams, this is the result of 
  considering the quotient of the plane $\R\times [0,1]$ by $\Z$ acting by addition to
  the $x$-coordinate.  Note that this sends all the partners to
  a single curve in  $\R/\Z\times [0,1]$, and all ghosts that differ by translation by $\Z$
  to a single curve.  The result is a cylindrical
  KLRW diagram with flavors in $\Ab=\R/\Z$ as defined in Definition \ref{def:cylindrical-diagram}, where we take the images of the paths $\pi_{i,k}$ to be corporeals $\Sc$, the images of $\hat{\bbeta}_{i,m}$ to be red strands $\Sr$ and the images of $\pi_{t(e),k}+\hat{\bbeta}_e$ to be ghosts $\Sg$ and the flavors to coincide with $x$-values at the top and bottom of the diagram, i.e. the values of the paths mod $\Z$ at $y=0$ and $y=1$.

Every cylindrical
    diagram has a unique lift to a path $[0,1]\to \ft_{\hat\beta}$ which starts in the fundamental region of $\widehat{W}$ where the coordinates $z_{i,k}$ satisfy
\begin{equation}
 -\frac{1}{2}< z_{i,1}<z_{i,2}<\cdots < z_{i, {v_i}}<\frac{1}{2}.  
\end{equation}
That is, by the path lifting property of the universal cover, each of the curves $\bar{\pi}\colon [0,1]\to \R/\Z$ has a unique lift $\pi$ with $-\frac{1}{2}<\pi(0)<\frac{1}{2}$, and we can number these so that \begin{equation}
    -\frac{1}{2}<\pi_{i,1}(0)<\cdots <\pi_{i,v_i}(0)<\frac{1}{2}.
\end{equation} Let $\tilde{D}$ be the unrolled diagram defined by the
paths $\pi_{i,k}$, followed by the unique element of $\What$
sending the top of this diagram back to the fundamental region.  
\begin{definition}\label{def:cylindrical
  KLRW-r}
 Given a cylindrical
   diagram $D$ with no dots, let $\mathbbm{r}_D$ denote the
  morphism $ {\mathbbm{r}_{\tilde{D}}}$ associated to the lifted unrolled diagram $\tilde{D}$. 
 
 If the diagram contains dots, then place these in the lifted diagram on the unique partner preimage which has $x$-value in $(-\frac{1}{2}, \frac{1}{2} )$.
\end{definition}
We have to be careful about lifting
   diagrams with dots,
because if we do so in the most naive way, the result will not be
compatible with composition, which the definition above is.  We could accomplish the same effect if instead of 
applying a Weyl group element at the end, we  applied one
immediately whenever we left the fundamental region to move back into
it.  

We can also interpret the relations
(\ref{eq:dot-migration}--\ref{b-triple-smart2}) as relations on
cylindrical KLRW
   diagrams, following the rule that $\sum a_iD_i=0$
if we have that $\sum a_i\mathbbm{r}_{D_i}=0$.  Some of these can
interpreted locally exactly as they appear above:
(\ref{b-first-QH}--\ref{b-second-QH}) and
(\ref{b-psi2}--\ref{b-nilHecke-2}) are of this type. On the other, if
a dot on a cylindrical diagram is slid over the half-integer ghost with
label $\infty$, then it goes between lifting to the ghost just right
of $x=-\frac{1}{2}$ to that just left of $x=\frac{1}{2}$.  Thus, if we
draw $x=\frac{1}{2}$ as a fringed grey line, the effect
of (\ref{eq:dot-migration})  is thus the following relation on cylindrical
   diagrams: 
    \begin{equation}\label{a-dot-slide}
    \begin{tikzpicture}[very thick,baseline,scale=.7]
  \draw(-3,0) +(-1,-1) -- +(1,1);
  \draw[fringe](-3,0) +(0,-1) --  +(0,1);
\fill (-3.5,-.5) circle (3pt); \end{tikzpicture}
=
 \begin{tikzpicture}[very thick,baseline,scale=.7] \draw(1,0) +(-1,-1) -- +(1,1);
  \draw[fringe](1,0) +(0,-1) --  +(0,1);
\fill (1.5,.5) circle (3pt);
    \end{tikzpicture} +h  \begin{tikzpicture}[very thick,baseline,scale=.7] \draw(1,0) +(-1,-1) -- +(1,1);
  \draw[fringe](1,0) +(0,-1) --  +(0,1);
    \end{tikzpicture}
\qquad \qquad     \begin{tikzpicture}[very thick,baseline,scale=.7]
  \draw(-3,0) +(1,-1) -- +(-1,1);
  \draw[fringe](-3,0) +(0,-1) --  +(0,1);
\fill (-2.5,-.5) circle (3pt); \end{tikzpicture}
=
 \begin{tikzpicture}[very thick,baseline,scale=.7] \draw(1,0) +(1,-1) -- +(-1,1);
  \draw[fringe](1,0) +(0,-1) --  +(0,1);
\fill (.5,.5) circle (3pt);
    \end{tikzpicture} -h  \begin{tikzpicture}[very thick,baseline,scale=.7] \draw(1,0) +(1,-1) -- +(-1,1);
  \draw[fringe](1,0) +(0,-1) --  +(0,1);
    \end{tikzpicture}
  \end{equation}
The other relations need to be interpreted carefully to be compatible
with lifting. For example, the relations
(\ref{x-cost-1}--\ref{x-cost-2}) and  the relations
(\ref{b-black-bigon1}--\ref{b-black-bigon1}) need to be applied 
in the version where all original strands have $x$-values in $(-\frac
12,\frac 12)$ (we can isotope to avoid any values in the coset
$\Z+\frac 12$).  For (\ref{x-cost-1}--\ref{x-cost-2}), this means that
$m\in (-\frac
12,\frac 12)$, and for 
(\ref{b-black-bigon1}--\ref{b-black-bigon1}) that $m\in (-1,1)$, with
the sign determined by how the strands are cyclically ordered compared
with $x=1/2$.

\begin{definition}\label{def:qcKLRW}  The {\bf quantum cylindrical KLRW (qcKLRW)
    algebra} $\hRring$ for the dimension vectors $\Bv,\Bw$ and parameters
  $ {\hat\beta_*},  {b_*}$ is the quotient of the formal span of cylindrical
  KLRW diagrams  with 
  \[\hat{\bbeta}_e(0)=\hat{\bbeta}_e(1)=\hat\beta_e\qquad\qquad  \hat{\bbeta}_{i,m}(0)=\hat{\bbeta}_{i,m}(1)=\hat\beta_{i,m} \]
  over $\K[b_*,h]$  modulo the relations induced by
  (\ref{eq:dot-migration}--\ref{b-triple-smart2}), in particular  by
  (\ref{a-dot-slide}), with the usual rule of
  multiplication by stacking.
  
  The {\bf quantum cylindrical KLRW
category} is the category whose objects are cylindrical flavored sequences, and 
morphisms are diagrams with fixed flavored sequences at top and bottom.
\end{definition}
\notation{$\hRring$}{The quantum cylindrical KLRW algebra (Definition \ref{def:qcKLRW}).}
Immediately from Lemma \ref{lem:unrolled-B} we have:
\begin{proposition}\label{prop:KLR-B}
The qcKLRW category is equivalent to the category $\scrB_{\hat\beta}$ via the functor sending a diagram $D$ to the morphism $\mathbbm{r}_D$.
\end{proposition}
Since the relations of the nilHecke algebra are not deformed in the
qcKLRW algebra, we can define idempotents $e(a)$ exactly as in Definition \ref{def:ea}.  These correspond to
elements of $\ft_{\hat\beta}$ which are scalar matrices.  Thus the corresponding object $\eta_a$ lies on the root
hyperplanes is unexceptional, but not generic. 
\begin{lemma}\label{lem:quantum-coulomb} The algebra $e(a)\hRring e(a)$ is isomorphic to the 
  quantum Coulomb branch $\Asph$ for a choice of flavor depending
  on $a$.
\end{lemma}
\begin{proof}
 By \cite[Rmk. 3.8]{websterKoszulDuality2019}, the endomorphisms of   $\eta_a$ are gotten by taking endomorphisms of a nearby generic point and cutting with an idempotent, which exactly corresponds to the nilHecke idempotent in $e(a)$.  Thus we have that $\End_{\scrB^+}(\eta_a)=e(a)\hRring e(a).$

Now, let us calculate these endomorphisms. The corresponding parahoric $\Iwahori_{\eta_a}$ is exactly
$G[[t]]$.
  Consider the cocharacter $\mu\colon \C^*\to\prod_{i\in
    \vertex}GL(\C^{w_i})$ which acts on the $k$th basis vector in
  $\C^{w_i}$ with weight $\lfloor a+ \hat\beta_{i,k}\rfloor$; note that
  this is locally constant in terms of $a$, only changing when $a$
  passes one of the red lines.
The subspace $U_{\eta_a}$ is given by $t^{-\mu} V[[t]]$, and so
multiplication by $t^\mu$ induces a $G[[t]]$-equivariant  isomorphism
between the spaces ${}_{\tau}\EuScript{X}_{\tau}\cong
{}_{\eta_a}\EuScript{X}_{\eta_a}$. However, this action does not
commute with the action of the cocharacter $\varphi$ (since it doesn't
commute with the loop action); it intertwines the action of $\varphi$
with its product with $\mu$, so this gives the necessary shift of
flavor.  
We have an isomorphism $\End_{\scrB^+}(\tau,\tau)\cong \Asph$ by \cite[(3.2)]{websterKoszulDuality2019}.
\end{proof}
This construction is closely related to the flag Yangian introduced in \cite[Def. 4.12]{KTWWYO}.  In that paper, we assumed that $\quiver $ was bipartite with the sets of nodes called {\bf even} and {\bf odd} such
that for each edge $i\to j$, the vertex $i$ is even and $j$ is odd.  Furthermore, the
definition depended on a polynomial $p_i$.  If we have that
$h=2$,   $b_e=0, \hat\beta_e=1/2$ for all edges in $\Gamma$,  $\hat\beta_{i,k}=0$ for all $i,k$, and the scalars $b_{i,k}$ are the roots (with
multiplicity) of $p_i(2u-1)$, then the cylindrical  KLR category is
closely related to the flag Yangian category,
via the transformation of diagrams sending all odd strands to their ghosts.  Since there are some other minor differences of convention between
these categories, we will not make a precise statement about the
relationship between them.

Let us give a simple example. Consider  $\gaugeG=GL(2)$ and $\matterV\cong \C^2\oplus \mathfrak{gl}_2$.  We have a natural isomorphism $\ft_{\R}\cong \R^2$ with the coordinates given by $z_{1},z_{2}$.  The unrolled matter hyperplanes are $z_1,z_2,z_1-z_2\in \Z-\frac{1}{2}$ and the unrolled root hyperplanes are $\alpha= z_1-z_2\in \Z$.  

With these conventions, we match morphisms of the extended category with cylindrical
  KLRW diagrams.  We'll draw these on a cylinder sliced open at $x=\frac{1}{2}$.  
\begin{equation*} 
       \tikz[very thick,scale=.8,baseline]{
\draw (1.2,2.5)-- (1.2,-2.5) node[at start,above,scale=.8]{$z_1=\frac{1}{2}$}; \draw (-1.2,2.5)--
(-1.2,-2.5) node[at start,above,scale=.8]{$z_1=-\frac{1}{2}$};
\draw (2.5,1.2)-- (-2.5,1.2) node[at start,right,scale=.8]{$z_2=\frac{1}{2}$}; \draw (2.5,-1.2)--
(-2.5,-1.2) node[at start,right,scale=.8]{$z_2=-\frac{1}{2}$}; 
 \draw[dotted] (-2.5,-2.5) -- node[above right,at
        end,scale=.8]{$\alpha=0$}(2.5,2.5); 
        \draw(-2.5,-1.3) -- node[left,at
        start,scale=.8]{$\alpha=\frac{1}{2}$}(1.3,2.5); 
         \draw (-1.3,-2.5) -- node[left,at
        start,scale=.8]{$\alpha=-\frac{1}{2}$}(2.5,1.3); 
 \draw[dotted] (-2.5,-.1) -- node[left ,at
        start,scale=.8]{$\alpha=-1$}(.1,2.5); 
 \draw[dotted] (-.1,-2.5) -- node[right,at
        end,scale=.8]{$\alpha=1$}(2.5,.1); 
\draw[->,dashed] (-.8,.8) to (.3,-.3);
}\qquad \leftrightarrow \qquad
       \tikz[very thick,xscale=1.5,baseline]{
          \draw[fringe] (-1,-1)-- (-1,1);
          \draw[fringe] (1,1)-- (1,-1);
           \draw (.8 ,-1) to (-.3,1);
        \draw (-.8 ,-1) to (.3,1);
        \draw[dashed] (-.2,-1) -- (-1,0.45454545454);
        \draw[dashed] (.2,-1) -- (1,0.45454545454);
        \draw[dashed] (-.7,1) -- (-1,0.45454545454);
        \draw[dashed] (.7,1) -- (1,0.45454545454);
        }
\end{equation*}
\begin{equation*} 
       \tikz[very thick,scale=.8,baseline]{
\draw (1.2,2.5)-- (1.2,-2.5) node[at start,above,scale=.8]{$z_1=\frac{1}{2}$}; \draw (-1.2,2.5)--
(-1.2,-2.5) node[at start,above,scale=.8]{$z_1=-\frac{1}{2}$};
\draw (2.5,1.2)-- (-2.5,1.2) node[at start,right,scale=.8]{$z_2=\frac{1}{2}$}; \draw (2.5,-1.2)--
(-2.5,-1.2) node[at start,right,scale=.8]{$z_2=-\frac{1}{2}$}; 
 \draw[dotted] (-2.5,-2.5) -- node[above right,at
        end,scale=.8]{$\alpha=0$}(2.5,2.5); 
                \draw (-2.5,-1.3) -- node[left,at
        start,scale=.8]{$\alpha=\frac{1}{2}$}(1.3,2.5); 
         \draw (-1.3,-2.5) -- node[left,at
        start,scale=.8]{$\alpha=-\frac{1}{2}$}(2.5,1.3); 
 \draw[dotted] (-2.5,-.1) -- node[left ,at
        start,scale=.8]{$\alpha=-1$}(.1,2.5); 
 \draw[dotted] (-.1,-2.5) -- node[right,at
        end,scale=.8]{$\alpha=1$}(2.5,.1); 
\draw[->,dashed] (-.7,0) to (-.1,1.8);
}\qquad \leftrightarrow \qquad
       \tikz[very thick,xscale=1.5,baseline]{
          \draw[fringe] (-1,-1)-- (-1,1);
          \draw[fringe] (1,1)-- (1,-1);
           \draw[dashed] (-1 ,-1) to(.8,1);
        \draw (-.7 ,-1) to (-.1,1);
        \draw[dashed] (.3 ,-1) to (.9,1);     
        \draw (0 ,-1) to (1,0.1111111111);
            \draw (-1 ,0.1111111111) to (-.2,1);
        \draw[dashed] (.3 ,-1) to (.9,1); }
    \end{equation*}
 
\subsection{The classical limit}
\label{sec:classical-limit}

Now, let us consider the classical limit where we set $h=0$.  In this
case, the relations (\ref{eq:dot-migration}--\ref{b-triple-smart2})
become exactly the relations
(\ref{c-first-QH}--\ref{w-triple-point2}).  As the name suggests, we
thus have:
\begin{lemma}\label{lem:h=0}
  The cylindrical KLRW algebra $\Rring$ of Definition \ref{def:cKLRW} is the
  specialization of $\hRring$ at $h=0$.  
\end{lemma}
This observation allows us to finally begin proving results from Section \ref{sec:cylindr-klrw-algebr}:
\begin{proof}[Proof of Thm. \ref{coulomb-idempotent}]
  \label{proof-coulomb-idempotent}
  Specialize Lemma \ref{lem:quantum-coulomb} at $h=0$.  
\end{proof}
Note that after this specialization, the shift of flavor that was
needed in the quantum version this theorem disappears; since the classical
Coulomb branch can be written as $G[[t]]$-equivariant homology of
${}_{\tau}\EuScript{X}_{\tau}$, the map in the proof of Lemma
\ref{lem:quantum-coulomb} induces an algebra isomorphism on the nose.

Note that this implies a basis theorem for the cylindrical KLRW
algebras.   Given a cylindrical loading with
$ {v_i}$ elements mapping to $i\in {\vertex}$, we have a unique way of lifting to
real numbers $z_{i,1}, \dots, z_{i,v_i}$ in the fundamental region
(that is, satisfying $  -\frac{1}{2}< z_{i,1}<\cdots < z_{i,v_i}<\frac{1}{2}
$), and the extended affine Weyl group $\widehat{W}$ acts freely
transitively on the set of possible lifts.  Having fixed two
cylindrical loadings $S$ and $T$, there is an unrolled diagram with a
minimal number of crossings with 
the bottom given by this lift of $S$ and the image of this  lift of
$T$ under $w\in \widehat{W}$; this diagram is not unique, but as usual, any
two choices differ by the diagram for a shorter permutation by the
relations (\ref{red-triple-correction}--\ref{w-triple-point2}).  The
image of this diagram $D_w$ on the cylinder $\R/\Z\times [0,1]$ gives a cylindrical KLRW diagram.  

From \cite[Cor. 3.15]{websterKoszulDuality2019}, we find that: 
\begin{lemma}\label{lem:cyl-basis}
The Hom space between two objects in the cylindrical KLRW category is a
free module for the left action of polynomials in the dots, with 
basis $D_w$ for $w\in \widehat{W}$.
\end{lemma}

There is a second sense in which the cylindrical flavored KLRW algebra describes the structure of the Coulomb branch: it also encodes the representation theory of the quantum Coulomb branches $A_1$.  This algebra contains a commutative subalgebra $\Cft^W$.  We call a point in the spectrum of this ring {\bf integral} if it corresponds to a cocharacter of groups; most importantly for us, if $G=GL_n$, we can think of a point in $\Cft^W$ as fixing the characteristic polynomial of an $n\times n$-matrix, and integrality means that the roots of this polynomial are integral (which is stronger than requiring its coefficients to be).  We call a finite-dimensional representation {\bf integral} if the spectrum of the action of $\Cft^W$ on lies in the integral points of the spectrum.  
We can restate  Theorem \ref{I-thm:pStein-equiv} as follows:  
\begin{theorem}\label{thm:pStein-2}
The category of integral finite-dimensional representations of the quantum Coulomb branch
over $\Fp$  is equivalent to the category of representations of  $\Rring_{\beta}$ with $\Ab=\Fp$ for the parameters \begin{equation}
 {\beta_e}=\frac{ {b_e}}p\qquad
   {\beta_{i,k}}=\frac{ {b_{i,k}}}p.\label{eq:pthroot}
\end{equation}
\end{theorem} 
\begin{proof}
	As originally stated, Theorem \ref{I-thm:pStein-equiv} gives an equivalence of the integral finite-dimensional representations of the quantum Coulomb branch 
to the representations of a
subcategory $\widehat{\mathsf{A}}_p(\mathbb{F}_p)$ in the completed extended BFN
category (Definition \ref{I-def:sfA}) with the ``$p$th
root'' parameters (Definition \ref{I-def:pth-root}), in particular, $h=0$. The objects in this category exactly correspond to the elements of $\ft_{1,\Fp}$; the coordinates of these elements in turn give the different possible choices of flavored sequences.   The $h=0$ limit of Proposition \ref{prop:KLR-B} shows that this latter category is exactly the flavored cylindrical KLRW category.  
\end{proof}
Under this equivalence:
\begin{enumerate}
    \item The image of the idempotent $e'(\boldsymbol{\alpha})$ corresponds to the weight space for the corresponding maximal ideal of $S_h$.
    \item The dots correspond to the action of the nilpotent part of elements of $\Cft^W$. 
    \item The action of a diagram $D$ corresponds roughly with the action of the path $pD$, where we map $\R/\Z\to \R/\Z$ by multiplication by $p$.  The appearance of $p$th root parameters is exactly that taking image under this map will scale the distance between strands and ghosts by $p$.
\end{enumerate}

\subsection{Change of flavor}
\label{sec:change-flavor}

 In this
context,  we can interpret twisted diagrams as objects in twisting bimodules,  defined in \cite[Def. 3.20]{websterKoszulDuality2019}.  These bimodules are defined by considering the BFN category $\mathscr{B}^{\To}$ attached to the group
$\To$ acting on $V$; as defined in Section \ref{I-sec:background}, $\To$ is the preimage of a maximal torus in $\No/\gaugeG$, or equivalently, the subgroup of $\No$ generated by $\gaugeG$ and a maximal torus of $\No$.  \notation{$\To$}{The subgroup of $\No$ generated by $\gaugeG$ and a maximal torus of $\No$.}

For two different choices of parameters $ {b_e}, {b_{i,k}}\in \Fp$ for $\phi$, and  $b_e',b_{i,k}'\in
\Fp$ for $\phi'$, we have twisting bimodules corresponding to each choice of integers
\[\nu_e\equiv b_e'-b_e\pmod p\qquad \nu_{i,k}\equiv
  b_{i,k}'-b_{i,k}\pmod p.\]
We can interpret $\nu$ here as a real cocharacter of the flavor torus
$T_H/T_G$ (with the usual caveats about redundancy).  To match the notation of \cite{WebcohI}, we choose corresponding flavors $\phi$ and $\phi'=\phi+\nu$, in which case, we let ${}_{\phi+\nu}\scrT{}_{\phi}$ denote the twisting bimodule     over the categories
$ {\mathscr{B}_{\phi+\nu}}$ and $ {\mathscr{B}_{\phi}}$, defined in (\ref{I-eq:aXnua}). This bimodule arises from the grading on $\Hom_{\mathscr{B}^{\To}}(\second,\second)$ by the cocharacter lattice of $F$ (or equivalently the action of $T_F^{\vee}$).  We consider the elements of grading $\nu$, and define ${}_{\phi+\nu}\scrT{}_{\phi}$ to be the quotient by the left action of the maximal ideal in $\Sym(\mathfrak{t}_{F}^*)$ vanishing at $h(\phi+\nu)$, or equivalently, the right action of those vanishing at $h\phi$.  

Thus, applying  Proposition \ref{I-prop:B-equiv} (or equivalently, Theorem \ref{thm:pStein-2}) to 
$\mathscr{B}_{\phi+\nu}, \mathscr{B}_{\phi}$ and $\mathscr{B}^{\To}$,
we obtain that the twisting bimodule
${}_{\phi+\nu}\scrT{}_{\phi}$  is intertwined with the
corresponding a similar bimodule  ${}_{\phi'_{1/p}}\mathsf{T}_{\phi_{1/p}}$
with $p$th root conventions (Definition \ref{I-def:pth-root}).  Since $\frac{1}{p}\nu$ might not be
integral, we cannot apply the definition of
${}_{\phi+\nu}\scrT{}_{\phi}$ directly, and we take the
description above to be the definition, but let us say a few words
about why the fact that $h=0$
allows us to extend this definition to arbitrary cocharacters of $\ft_{F}$.

We will spare the reader the blizzard of notation required to say this
carefully, but in brief the categories with $p$th root conventions $\mathsf{B}_{\phi_{1/p}}$ and
$\mathsf{B}_{\phi'_{1/p}}$ can be realized as subcategories of $ {\mathsf{B}
^{\To}}$ modulo the action of polynomial morphisms $\ft_F^*$.
This quotient is only well-defined because $h=0$ as the definition above illustrates: in this case, the unique graded maximal ideal in $\Sym \ft_F^*$ is preserved by all morphisms at $h=0$, whereas maximal ideals are shifted if $h\neq 0$.  We identify the
object sets of $\mathsf{B}_{\phi_{1/p}}$ and
$\mathsf{B}_{\phi'_{1/p}}$ with orbits of $\ft$ in $\ft_{1;\To}$ which differ
by $\frac{1}{p}\nu$, and note that morphisms in these subcategories are precisely those in the larger category
generated by paths, polynomials, and $u_{\al}$ and the extended affine
Weyl group of $\gaugeG$ (as opposed to the affine Weyl group of $\To$, which
has more translations).  We can define
${}_{\phi'_{1/p}}\mathsf{T}_{\phi_{1/p}}$ as the space of morphisms in $\Hom_{\mathsf{B}
^{\To}}(\second,\second)/(\ft_F^*)$ generated by paths, polynomials, and $u_{\al}$ and elements of the extended affine
Weyl group of $G$ whose image in the cocharacter lattice of $F$ is $\nu$.  

In the quiver case, we can also recover the twisting bimodules
${}_{\phi+\nu}\scrT{}_{\phi}$  and
${}_{\phi'_{1/p}}\mathsf{T}_{\phi_{1/p}}$  using the appropriate
modification of Proposition \ref{prop:KLR-B} and Lemma \ref{lem:h=0}.
Let $\hat{b}_e\in\Z$ and $\hat{b}_{i,k}\in\Z$ be preimages of $b_e,b_{i,k}$ and let 
\begin{equation}
\hat{\bbeta}_e(t)=\frac{ {\hat{b}_e}+\nu_et}p\qquad
  \hat{\bbeta}_{i,k}(t)=\frac{ {\hat{b}_{i,k}}+\nu_{i,k}t}p.\label{eq:pthroot-functions}
\end{equation} and let ${\inter}_e^{(p)}$ be the vector defined by the number of intersections of a strand of label $h(e)$ with the ghost for $e$ in a diagram that uses these functions as the distance between corporeal and ghost strands and the position of red strands, respectively.  Note that if $\frac{1}{p}\nu_e\in \Z$, then ${\inter}_e^{(p)}=\frac{1}{p}\nu_e$.  

We can define a bimodule over qcKLRW algebras given by the span of twisted
  cylindrical KLRW diagrams for the functions 
  $\mathbf{\inter}_e=\nu_e$ modulo the local relations
  (\ref{eq:dot-migration}--\ref{b-triple-smart2}).
\begin{proposition}\label{prop:KLRW-bimodule}
  The isomorphism of Proposition \ref{prop:KLR-B} extends to a 
  bimodule isomorphism of ${}_{\phi+\nu}\scrT{}_{\phi}$ with the span of twisted
  cylindrical KLRW diagrams for the functions $\mathbf{\inter}_e=\nu_e$ modulo the local relations
  (\ref{eq:dot-migration}--\ref{b-triple-smart2}); 

 With $p$th root conventions, this means that 
  ${}_{\phi'_{1/p}}\mathsf{T}_{\phi_{1/p}}$ is isomorphic to the span
  of twisted cylindrical KLRW diagrams with $h=0$ for ${\inter}_e^{(p)}$.
\end{proposition}
This now establishes Theorem \ref{thm:D-equivalence}(3) by Theorem \ref{I-ithm:Schobers}. 

\begin{proof}[Proof of Thm. \ref{thm:partial-resolution}]
  \refstepcounter{dummy} \label{proof-thm:partial-resolution}
By the $h=0$
special case of Proposition \ref{prop:KLRW-bimodule} above, we can
rewrite $e(a)  {\rif^{\mathbf{d}}}e(a)$ as the bimodule
${}_{\phi+n\nu}\mathscr{T}_{\phi}(\eta_a,\eta_a)$. By the definition
(\ref{I-eq:aXnua}), we thus have
\[e(a) B_{\bbeta^{(k)}}e(a) \cong
  H^{BM}_*({}_{\eta_a}\EuScript{X}^{(k\nu)}_{\eta_a}).\] Since the
space ${}_{\eta_a}\EuScript{X}^{(k\nu)}_{\eta_a}$ is precisely the
same as the quotient by $G[[t]]$ of that denoted
$\mathcal{\tilde{R}}^{(k\nu)}$ in \cite{BFNline}, this shows that our
definition matches exactly the projective coordinate ring of the
partial resolution attached to the cocharacter $\nu$ in
\cite{BFNline}.
  \end{proof}

\subsection{Connection to tilting bundles}

Studying the $h=0$ case is particularly important because of its
connection to coherent sheaves, as shown in \cite{WebcohI}.    
Thinking about Theorem \ref{thm:pStein-2} in terms of 
quantizations in characteristic $p$, we use this isomorphism to give a derived equivalence between the category of representations of  $\Rring_{\Fp}$ and the coherent
sheaves $\Coh(\tilde{\fM}_{\Fp})$ as long as $p$ is sufficiently large, and the parameters are generic (Theorem \ref{I-th:Q-equiv}). Note that $\Fp$ plays two roles here. It is both:
\begin{enumerate}
	\item the coefficients $\K$ of the algebra $\Rring$ and the base ring of the variety $\tilde{\fM}_{\Fp}$, and
	\item the group $\Ab$ from which we choose longitudes.  This reflects the fact that we are diagonalizing the action of $\Cft^W$ in a representation of $A_1$ over $\Fp$.
\end{enumerate} 
If our primary interest is in coherent sheaves, we can vary the coefficient ring $\K$ and abelian group $\Ab$ separately, by giving a definition of the functor from $\Rring_{\K}\mmod$ to  $\Coh(\tilde{\fM}_{\K})$ which is independent of the choice of $\K$.  

Given a module $M$, we can consider the $\Rring_{\K}$-modules 
$$M^{(k)}={}_{\phi_{1/p}+k\nu }\mathsf{T}_{\phi_{1/p}}\otimes_{\Rring}M.$$ For a fixed $a\in \frac{1}{p}\Z/\Z$, we can construct a module $M^{\geq 0}=\bigoplus_{k\geq 0} e(a)M^{(k)}$ over the projective coordinate ring $\mathbf{A}^{\mathbf{\inter}}$.

\begin{definition}
	Let $\cQ_{M}$ be the localization of $M^{\geq 0}$ to a coherent sheaf on the variety $\tilde{\fM}_{\K}$.  
\end{definition}   
An important thing to note about this sheaf: for each $\boldsymbol{\alpha}\in \mathbb{L}$, we have an idempotent $e(\boldsymbol{\alpha})$ defined by Definition \ref{def:prefered}, and thus a projective module which defines a sheaf $\cQ_{\boldsymbol{\alpha}}=\cQ_{\Rring e(\boldsymbol{\alpha})}$.

If the set $\Ab$ is finite, then the most natural $M$ to consider is the algebra $\Rring$ itself.   If the set $\Ab$ is infinite, then we choose the finitely generated left projective $\Rring$-module generated by one representative of each equivalence class of idempotents; this module is finitely generated and every graded indecomposable projective module is a summand of it.  Let $\cQ_{\mathbf{b}}=\cQ_{M}$ for this  choice.

  We say that the choice of $\mathbf{b}\in \Ab$ is {\bf generic} if the number of different isomorphism classes of flavored sequences is equal to the maximal number possible amongst all choices of $\mathbf{b}'\in \R/\Z$ with the same quiver $\Gamma$ and dimension vectors $\Bw,\Bv$.  
Theorem
\ref{I-th:Q-equiv} shows that: 
\begin{theorem}\label{th:Q-equiv-2}\hfill
\begin{enumerate}
	\item The sheaf $\cQ_{\mathbf{b}}$ is always tilting. 
	\item 
 If a BFN resolution $\tM$ exists and $\K$ is characteristic 0,  then for generic $\mathbf{b}$, the vector bundle 
	 $\cQ_{\mathbf{b}}$ is a tilting generator.
	 \item  We have an isomorphism $\Hom(\cQ_{\boldsymbol{\alpha}},\cQ_{\boldsymbol{\alpha}'})\cong e(\boldsymbol{\alpha}')\Rring e(\boldsymbol{\alpha})$ compatible with multiplication.  If $\Ab$ is finite, then $\End(\cQ_{\mathbf{b}})$ is the corresponding flavored KLRW algebra $\Rring$ with parameters $\mathbf{b}$. 
\end{enumerate}
\end{theorem}
In the case where $\Ab$ is infinite, the endomorphism algebra $\End(\cQ_{\mathbf{b}})$ is not precisely $\Rring$, since we have to put a restriction on the flavored sequences we'll allow in order to get a finite rank vector bundle.  However, since we have considered one representative of each equivalence class, adding in more flavored sequences will just add more copies of the same summands to $\cQ_{*}$. In particular, we can consider the  infinite rank vector bundle $\cQ_{\Rring}$, which  is still equiconstituted with $\cQ_{\mathbf{b}}$.  Thus $\End(\cQ_{\mathbf{b}})$ is Morita equivalent to $\Rring$.  

This gives us a very concrete combinatorial understanding how small values of $p$ can give us trouble:  since there is a relatively small number of longitudes possible in $\Ab=\Fp$, there may be no generic choice of $\mathbf{b}$.  Similarly, it explains a phenomenon already noted in the tilting bundles constructed using the approach of \cite{KalDEQ}: as $p$ grows, the rank of the vector bundle grows, but only by adding more and more copies of the same summands.

\begin{remark}
	Theorem \ref{th:Q-equiv-2}(2) also holds for $\K$ having sufficiently large positive characteristic by semi-continuity; it seems likely that it holds for arbitrary $\K$ (for a generic choice of $\mathbf{b}$ in $\Ab=\R/\Z$), but at the moment, we have no strategy for proving this.
\end{remark}

This establishes Theorem \ref{thm:NCSR} and 
allows to us to prove two of the claimed results in the introduction:
Theorem \ref{thm:D-equivalence}(2) \label{proof-E2} now follows from Theorem
\ref{I-th:Q-equiv} and Theorem \ref{ith:NCCR} from Corollary
\ref{I-cor:A-nccr}. \label{proof-NCCR}

\begin{example}\label{example:NZ}
Consider the case where $\gaugeG=\C^*$ acting on $\C^2$ by scalars.  In this case, the Coulomb branch is $T^*\mathbb{P}^1$.  The corresponding cylindrical KLRW algebra has two red strands, and one black strand, all with the same label. 
There are two idempotents in this algebra, corresponding to the two
cyclic orders of the 3 strands.  Since the corresponding quiver has no
edges, the black strand has no ghosts.
\begin{equation*}
        \tikz[xscale=.9]{
      \node[label=below:{$ x$}] at (-4.5,0){ 
       \tikz[very thick,xscale=1]{
          \draw[fringe] (-.7,-.5)-- (-.7,.5);
          \draw[fringe] (1.7,.5)-- (1.7,-.5);
          \draw[wei] (.3,-.5)-- (.3,.5);
          \draw[wei] (1.5 ,-.5)-- (1.5,.5);
\draw (-.7,0) to[out=0,in=-90] (-.2,.5);
           \draw (.9 ,-.5) to[out=90,in=180] (1.7,0);
        }
      };
      \node[label=below:{$ x^*$}] at (0,0){ 
       \tikz[very thick,xscale=1, yscale=-1]{          
       \draw[fringe] (-.7,.5)-- (-.7,-.5);
          \draw[fringe] (1.7,-.5)-- (1.7,.5);
          \draw[wei] (.3,-.5)-- (.3,.5);
          \draw[wei] (1.5 ,-.5)-- (1.5,.5);
\draw (-.7,0) to[out=0,in=-90] (-.2,.5);
           \draw (.9 ,-.5) to[out=90,in=180] (1.7,0);
        }
      };
       \node[label=below:{$ y $}] at (4.5,0){ 
        \tikz[very thick,xscale=1]{
          \draw[fringe] (-.7,-.5)-- (-.7,.5);
          \draw[fringe] (1.7,.5)-- (1.7,-.5);
          \draw[wei] (.3,-.5)-- (.3,.5);
          \draw[wei] (1.5 ,-.5)-- (1.5,.5);
\draw (.9 ,-.5)  to[out=90,in=-90] (-.2,.5);
       }
      };
      \node[label=below:{$ y^* $}] at (9,0){ 
        \tikz[very thick,xscale=1, yscale=-1]{
           \draw[fringe] (-.7,.5)-- (-.7,-.5);
          \draw[fringe] (1.7,-.5)-- (1.7,.5);
          \draw[wei] (.3,-.5)-- (.3,.5);
          \draw[wei] (1.5 ,-.5)-- (1.5,.5);
\draw (.9 ,-.5)  to[out=90,in=-90] (-.2,.5);
       }
      };
      }
\end{equation*}

These satisfy the quadratic relations 
\begin{equation}
    xx^*=yy^*\qquad x^*x=y^*y,
\end{equation}
and it's easy to check that these are a complete set of relations.  
This corresponds to the endomorphisms of the tilting generator $\mathcal{O}\oplus \mathcal{O}(-1)$ on $T^*\mathbb{P}^1$, via the isomorphism sending $x,y$ to the corresponding multiplication maps $\mathcal{O}(-1)\to \mathcal{O}$, and $x^*,y^*$ to the maps $\frac{\partial}{\partial x}, -\frac{\partial}{\partial y}\colon \mathcal{O}\to \mathcal{O}(-1)$ (thinking of vector fields as functions on $T^*\mathbb{P}^1$).  
This algebra is Koszul and its Koszul/quadratic dual is easily seen to be defined by
\begin{equation}
    xx^*=-yy^*\qquad x^*x=-y^*y\qquad y^*x=x^*y=yx^*=xy^*=0.
\end{equation}
This latter set of relations defines an 8-dimensional algebra studied by Nandakumar and Zhao in \cite{nandakumarCategorificationBlocks2021}, which appears as the endomorphisms of a projective generator for exotic sheaves on $T^*\mathbb{P}^1$.  
\end{example}

\begin{example}
The most important example of a case where weighting is useful is the Jordan quiver where we only have a single node in our
quiver, which carries a loop, equipped with the weight
$ {\hat\beta_e}=\vartheta$, and a single red strand which we can put at $x=0$ (of course, labeled with
this node). The corresponding Coulomb branch is $\mathfrak{M}=\Sym^n(\C^2)$, and the resolution of $\tilde{\fM}^{d}$ for any $d\neq 0$ is the Hilbert scheme of $n$ points in $\mathbb{C}^2$ by \cite[Prop. 3.2]{BFNline}.  

A non-commutative resolution of singularities on this Hilbert scheme was constructed by Bezrukavnikov, Finkelberg and Ginzburg in \cite[Th. 1.4.1]{bezrukavnikovCherednikAlgebras2006}.  We expect that this will be Morita equivalent to the resolution $A$ we have constructed; this should follow from the appearance of the rational Cherednik algebra of $S_n$ as a morphism space in $\mathscr{B}^+$, shown independently in \cite[Th. 4.2]{BEF} and \cite[Lem. 4.2]{Webalt}.  Under the equivalence of Proposition \ref{prop:KLR-B}, this induces an isomorphism \[\mathsf{H}_{h,b}\cong e(\boldsymbol{\gamma})\Rring^h e(\boldsymbol{\gamma})\]
where $\boldsymbol{\gamma}$ is the point $(\frac{1}{2n},\frac{3}{2n},\dots, \frac{-1}{2n})$ where the points are evenly spaced around the circle, and $0<\vartheta<1/n$.

The objects in the cylindrical KLRW category are thus $n$-tuples of distinct points in $S^1$, where each point has a ghost $\vartheta$ units to its right, which the other points avoid.  This information can be recorded by listing the order in which one encounters dots and ghosts; the set of possible configurations for a given $\vartheta$ corresponds to the set $\bar \Lambda$ discussed earlier.  

Note that the set of possible configurations is locally constant, and
will only change at values of $\theta$ where one has a non-simple
hyperplane arrangement. This can only be the case if there is a loop of equations 
\begin{align*}
    z_1-z_2&\equiv\vartheta\pmod \Z\\
    z_2-z_3&\equiv\vartheta\pmod \Z\\
    \vdots&\\
    z_k-z_1&\equiv\vartheta\pmod \Z
\end{align*}
for $k\leq n$.  This implies that $k\vartheta\in \Z$, i.e. that $\vartheta$ is rational with denominator $\leq n$.  Of course, this same set of values has shown up in the structure of Hilbert schemes and Cherednik algebras in other contexts. 
\end{example}

\begin{proof}[Proof of Prop. \ref{prop:DE}]
  \refstepcounter{dummy}\label{proof-prop:DE}
  By Theorem \ref{I-thm:asymptotic-derived}, we can approximate $ {\beta}$ and $\beta'$ by choices of $ {\mathbf{b}},\mathbf{b}'$ in $\Ab=\Fp$ as in \eqref{eq:pthroot-functions} such that derived localization holds at these parameters for sufficiently large $p$.  Furthermore, the choice of lifts $\hat\beta$ and $\hat\beta'$ induce integral lifts of $\mathbf{b},\mathbf{b}'$, and thus a corresponding wall-crossing functor, induced by $\nu={\widehat{\mathbf b}}'-{\widehat{\mathbf b}}$.  
  Thus, the corresponding wall-crossing functor is a derived
  equivalence by Lemma \ref{I-lem:localize-twist}, and so Proposition
  \ref{prop:KLRW-bimodule} shows the same is true for tensor product
  with the change-of-charge bimodule $ {\rif_{\hat\beta',\hat\beta}} $ with coefficients in $\mathbb{F}_p$.

  For any
  base ring $\K$, 
  derived tensor product with $\rif_{\hat\beta',\hat\beta}$ is an equivalence if and
  only if the natural map $\Rring_{\beta}\to
  \RHom_{\Rring_{\beta'} }(\rif_{\hat\beta',\hat\beta}, \rif_{\hat\beta',\hat\beta})$ is an
  isomorphism, and similarly with $\beta, \beta'$ reversed.
  First consider the cone $K^\bullet$ of this map with $\K=\Z$.  The result is a
  complex whose cohomology is a finitely generated graded bimodule
  over $\Rring_{\beta}$ in each degree.  Since each graded degree of
  $\Rring_{\beta}$ is a finitely generated abelian group, the
  same is true of $H^k(K^\bullet)$ for each homological degree $k$.
  In particular, we can use the universal coefficient theorem to
  compute the cohomology of $K^\bullet\otimes_{\Z}\K$ for any ring
  $\K$. We have shown above that $K^\bullet$ is exact
  after base change to $\mathbb{F}_p$ for some $p$; this shows that
  the cohomology of $H^k(K^\bullet)$ is torsion of order coprime to $p$, and so becomes
  trivial for $\K=\Q$.  
\end{proof}

\section{Proofs from Sections \ref{sec:cylindr-klrw-algebr} and \ref{sec:tangle}}

We deferred the proofs of a number of results in Sections
\ref{sec:cylindr-klrw-algebr} and \ref{sec:tangle} which required the
results of Section \ref{sec:diagrams}; in this section we will cover
these and any preliminary lemmata needed for them.

\begin{proof}[Proof of Lem. \ref{lem:R-exact}]
  \refstepcounter{dummy}\label{proof-lem:R-exact}
By Lemma~\ref{lem:cyl-basis}, the module $e(\Bi)  {\Rring^{\Bj}}$ has a basis as a free right module over the action of the dots indexed by affine permutations, with labels on strands fixed by the labeling on the top given by $\Bi$.  This basis is given by diagrams that trace out this affine permutation on the cylinder with a minimal number of crossings and no dots.  We can easily check that the set of these diagrams in a fixed left coset\footnote{As always confuses the author, this means an orbit for right multiplication of a subgroup.} of the finite permutation group spans a projective right $ \tilde{T}^{\bla}$ module, freely generated by the unique shortest element of this coset.  This is isomorphic to $e(\Bi') \tilde{T}^{\bla}$, where $\Bi'$ is determined by the bottom of this shortest coset diagram (the one with a minimal number of crossings).  
\end{proof}

In order to prove Lemma~\ref{lem:Z-match}, we need to give a lemma comparing the central charge $Z$ with a representation-theoretic central charge similar to the ones considered in \cite{annoStabilityConditions2015}.

   Fix a choice of parameters $ {\hat{b}_e}, {\hat{b}_{i,k}}\in \Z$ so that the
   associated parameters of the form \eqref{eq:pthroot-functions} lie
   in our preferred alcove $C_0$.  Consider the associated sheaf of
   algebras $ {\hat{\psalg}_\phi}$ for the
   corresponding flavor defined by Definition \ref{I-def:psalg}.  For any other $\hat{b}_e',\hat{b}_{i,k}'\in \Z$, we can
   consider the quantized line bimodule
   ${}_{\phi+\nu}\mathscr{T}_{\phi}$ with
   $\nu_e=\hat{b}_e'-\hat{b}_e,\nu_{i,k}=\hat{b}_{i,k}'-\hat{b}_{i,k}$.  Tensor product $
   {}_{\phi+\nu}\mathscr{T}_{\phi}\otimes -$ gives an equivalence of
   abelian categories $\hat{\psalg}_\phi\mmod
   \cong \hat{\psalg}_{\phi+\nu}\mmod$, so we can implicitly identify
   all these categories.  Let ${\psalg}_\phi\mmod_0$ be the
   subcategory of sheaves of modules that are set-theoretically
   supported on the fiber over the cone point in $\Coulomb$; note
   that we have left out the completion here, since the action on any such module
   automatically extends to the completion.
   \begin{definition}
     The representation theoretic central charge on $K^0(\hat{\psalg}_\phi\mmod_0)$ is defined by 
     \[\mathcal{Z}_{\mathbf{b}'}([\mathcal{M}])=\frac{1}{p^{V}}\chi(\R\Gamma(
       \mathcal{M}))\]
   \end{definition}
    Since the parameters \eqref{eq:pthroot-functions} lie in the alcove $C_0$,  we
    have an equivalence 
    \[ \mathcal{G}\colon D^b({\psalg}_\phi\mmod_0)\cong
    D^b( {\Rring_{C_0}}\fdmod).\] This is effectively Lemma
    \ref{I-lem:A-H} but also follows from applying  
   Theorem \ref{th:Q-equiv-2}, together with equivalence of
   $\hat{\psalg}_\phi\mmod_0$ with coherent sheaves with the same
   support; the characteristic assumption of Lemma
   \ref{th:Q-equiv-2} is unnecessary by Theorem
   \ref{I-thm:asymptotic-derived}.  Using the
   equivalence, we can interpret
   $\mathcal{Z}_{\mathbf{b}'}$ as a function on $K^0(\Rring_{C_0}\fdmod)$.
   
\begin{lemma}\label{lem:RT-charge}
     The RT central charge $\mathcal{Z}_{\mathbf{b}'}$ is:
     \begin{enumerate}
     \item polynomial in the variables $\frac{\hat{b}_e'}{p}$ and
       $\frac{\hat{b}_{i,k}'}{p}$.
     \item for $(\hat\beta_*)\in C_0$, this function limits
       to $ {Z_{\hat\beta }}$ in the sense that if $p^{(m)}$ is a
       series of primes with $\lim_{m\to \infty} p^{(m)}=\infty$, and $\hat{b}_e^{(m)},c^{(m)}_{i,k}$ a series of
       parameters such that $\frac{\hat{b}_e^{(m)}}{p^{(m)}}$ and
       $\frac{\hat{b}_{i,k}^{(m)}}{p^{(m)}}$ converge to $\hat\beta_e$ and
       $\hat\beta_{i,k}$, then
       $\mathcal{Z}_{\mathbf{b}^{(m)}}$ converges to
       $Z_{\hat\beta}$ as $m\to \infty$.
     \end{enumerate}
   \end{lemma}
 Note any point in $\R^\ell$ is a limit of the desired form, so $
 Z_{\hat\beta}$ is determined by this property.  
 \begin{proof}
  Given an $\hat{\psalg}_\phi$-module $\mathcal{M}$, we can
  consider its Rees module as a module over the usual microlocalization $W$ of $\Asph$ to $\tilde{\fM}$, as discussed in point (1) below Corollary \ref{I-cor:BFN-split}.  We can then reduce
     this sheaf modulo $h$, and obtain a coherent sheaf $\mathcal{N}$.  The important property of this sheaf is that  $\R\Gamma(\mathcal{N})$ has the same Euler
     characteristic of $\R\Gamma(\mathcal{M})$, since the former is
     the limit of a spectral sequence with $E_1$ page given by the
     latter.  
This means that we can also compute the RT central charge using $\R\Gamma(\mathcal{N})$.

     Of course, if we consider the tensor product
     $\mathcal{M}'={}_{\phi+\nu}\mathscr{T}_{\phi}\otimes
     \mathcal{M}$, then the resulting sheaf $\mathcal{N}'$ is simply
     $\mathcal{N}'\cong \mathcal{N}\otimes \mathcal{O}(\nu)$.  This
     shows that the polynomiality in $\hat{b}_e'/p$ and $\hat{b}_{i,k}'/p$, since
     changing these parameters is just tensoring $\mathcal{N}$ with a
     line bundle; the Euler characteristic of  $\R\Gamma(\mathcal{N}\otimes \mathcal{O}(\nu))$ depends polynomially on $\nu$ by Grothendieck-Hirzebruch-Riemann-Roch.

     By Theorem \ref{I-thm:which-polytope},
     the central charge $p^{V}\prod v_i! \cdot \mathcal{Z}_{\mathbf{b}'}$
     is equal to a sum over the set $\Lambda$ of the dimension of the
     image of the corresponding idempotent $eM$ times the number of
     $p$-torsion points in the corresponding polytope.

     On the other hand, the $ {Z}_{\hat\beta}$ is the sum over the
     same set, weighted by volume, and divided by $\prod v_i! $ to
     account for the fact that in our set of longitudes $a_1<
     \cdots<a_n$ is increasing.

     Thus, the agreement of these in the limit is just the fact that
     as $p\to \infty$, the number of $p$-torsion points in a
     polytope inside of
     $(\R/\Z)^n$ times $\frac{1}{p^n}$ limits to the volume of this
     set, since this is the leading coefficient of the Erhart polynomial.    
   \end{proof}

   \begin{remark}\label{rem:Z-match}
     Applying Grothendieck-Hirzebruch-Riemann-Roch, we can also compute this representation-theoretic central charge as an integral over the zero fiber of the resolution $\tilde{\fM}\to
     \fM$.  Up to normalization, this should match \cite[(5.18)]{aganagicKnotCategorification2020}.
   \end{remark}
\begin{proof}[Proof of Lem. \ref{lem:Z-match}]
  \refstepcounter{dummy}\label{proof-lem:Z-match}
  Consider
  two alcoves $C^+$ and $C^-$ with $C^+$ above $C^-$ across a single
  hyperplane, and the algebras $ {\Rring_{C^\pm}}$ for choices of parameters in these alcoves.

  This implies that for any sufficiently large prime $p$ there are parameters $\hat{b}_e,\hat{b}_{i,k}$ and a integer
  $m>0$ such that taking $\nu=m\chi$, 
the path  \eqref{eq:pthroot-functions} begins in $C^-$, ends in $C^+$
and passes through no other alcoves.  By Lemma
\ref{I-lem:localize-twist}, we have a commutative diagram:
\begin{equation}
 \tikz[->,thick,baseline]{
\matrix[row sep=12mm,column sep=35mm,ampersand replacement=\&]{
\node (d) {$D^b(\psalg_{\phi}\mmod_0)$}; \& \node (e)
{$D^b(\psalg_{\phi+\nu}\mmod_0)$}; \\
\node (a) {$D^b(\Rring_{C^-}\fdmod)$}; \& \node (b)
{$D^b(\Rring_{C^+}\fdmod)$}; \\
};
\draw (a) -- (b) node[below,midway]{$ {\rif_{\hat\beta',\hat\beta}}\otimes-$}; 
\draw (d) -- (a) node[left,midway]{$\mathcal{G}$} ; 
\draw (e) -- (b) node[right,midway]{$\mathcal{G}$}; 
\draw (d) -- (e) node[above,midway]{${}_{\phi+\nu}\mathscr{T}_{\phi}\otimes-$}; 
}\label{eq:wall-cross}
\end{equation}
Thus, we have that the RT central charge defined with respect to $C^-$ for an
$\Rring_{C^-}$-module $M$ is the same as that defined with
respect to $C^+$ for $\rif_{\hat\beta',\hat\beta}\Lotimes M$. 
Since the equivalence $\mathbb{B}_{C,C_0}$ is uniquely defined by the
relation that
\[\mathbb{B}_{C^-,C_0}(M)=\mathbb{B}_{C^+,C_0}(\rif_{\hat\beta',\hat\beta}\Lotimes M),\]
for any such pair, the RT central charge with respect to
$C_0$ for $\mathbb{B}_{C,C_0}(M)$ is the same as this charge for $M$
with respect to $C$.  This establishes the analog of Lemma
\ref{lem:Z-match} for the RT central charge, and thus taking limit as
in Lemma \ref{lem:RT-charge} gives the desired result.
\end{proof}
Before proving Theorem \ref{thm:real-variation}, we need to establish
some facts about the relationship of Gelfand-Kirillov dimension and Gelfand-Tsetlin modules, that is, modules locally finite under the ring $S_1^W$.  

Let us just remind the reader of a few facts: the ring ${S}_1$ is a polynomial ring generated by
variables $z_{i,k}$ for $k=1,\dots, v_k$.  Thus, a maximal ideal of
this ring is described by choosing scalars $a_{i,k}\in \C$, and a maximal ideal in $S_1^W$ corresponds to an orbit under $W$.  There is
a natural notion of equivalence of maximal ideals: consider the left-hand sides of (\ref{eq:unroll-matter1}--\ref{eq:unroll-matter2}) as linear functions; we will call these the {\bf matter functions}.
\begin{definition}
    We say that $\Ba,\Ba'$ are equivalent if $\Ba-\Ba'\in \prod_{i\in \quiver} \Z^{v_i}$ and there is {\em  no} matter function $f$ such that $f(\Ba)\in \Z_{>0} $ and $f(\Ba')\in \Z_{\leq 0}$.   
\end{definition}
Note, in particular,  (Weil) generically, none of these functions will have integer values, and if this is the case, $\Ba$ will be equivalent to any ideal that has integral difference from it.  There is another equivalence relation on maximal ideals which relates $\Ba,\Ba'$ if the weight spaces $\Wei_{\Ba}(M)$ and $\Wei_{\Ba'}(M)$ are naturally isomorphic as functors.  The equivalence classes of this relation are often called {\bf clans}.
By \cite[Lem. 4.15]{WebGT}, we have that:
\begin{lemma} If $\Ba,\Ba'$ define equivalent maximal ideals in the sense above, then they lie in the same clan.  
  \end{lemma}
This follows by the
same logic as \cite[Prop. 5.4]{WebGT}, though that result is only stated in a special case.  

Recall that the support $\supp(M)$ of a Gelfand-Tsetlin module $M$ is defined by 
 \[\supp(M)=\{\Ba\in \MaxSpec(S_1^W) \mid \Wei_{\Ba}(M)\neq 0\}\subset \MaxSpec(S_1^W)\}.\]
Let $M$ be a simple GT module over $\EuScript{A}_{\phi}^{\operatorname{sph}, \K}.  $
 By \cite[Lem. 4.11]{kamnitzerLieAlgebra2024}, the dimension of the Zariski closure $\dim \overline{\supp(M)}$ is the Gelfand-Kirillov dimension of $M$ itself.  

 Since $M$ is simple, if $\Ba$ is chosen so that $\Wei_{\Ba}(M)\neq 0$, all other ideals where $\Wei_{\Ba'}(M)\neq 0$ have integral difference with $\Ba'$.  Thus, since there are only finitely many linear functions we consider the sign of, this shows that there are finitely many clans in $\supp(M)$.  This shows that $\overline{\supp(M)}$ is the union of the Zariski closures of the clans in the support.  These are always formed by the integral points of a (possibly unbounded) polytope, and thus a finite union of affine subspaces by \cite[Prop. 7.1.2]{MVdB}.  This subspace is defined by the zero-set of the matter functions which take a finite number of values on the clan.  Just as with the polytope $P_{\hat\beta}$ in Section \ref{sec:real-variation}, you can think of this as sets of strands becoming ``locked'' to each other by the vanishing of matter functions.
 
There is another appearance of this statistic which is relevant for us.  Consider the image of $ \EuScript{A}_{\phi}^{\operatorname{sph}, \K}$ in $\End(M)$: 
\begin{lemma}
	Let $M$ be a Gelfand-Tsetlin module for $\EuScript{A}_{\phi}^{\operatorname{sph}, \K}$.  Then 
	\[\operatorname{GK-dim}(\EuScript{A}_{\phi}^{\operatorname{sph}, \K}/\operatorname{ann}(M))=2\operatorname{GK-dim}(M)=2\dim \overline{\supp(M)}.\]
\end{lemma}
\begin{proof}
	Let $d=\dim \overline{\supp(M)}$.  
	
	We can consider the  module $M'$ corresponding to $M$ over the Morita equivalent
	algebra $\EuScript{A}_{\phi}^{\K}$.  By \cite[Prop. 2.5]{kamnitzerLieAlgebra2024}, the algebra $\EuScript{A}_{\phi}^{\K}$  contains as a subalgebra the Coulomb branch algebra $\EuScript{A}_{\phi}^{\operatorname{ab}, \K}$ associated to the maximal torus $T\subset G$ with the same matter representation.  Let $A=\EuScript{A}_{\phi}^{\K}/\operatorname{ann}(M')$ and $A_{\operatorname{ab}}=\EuScript{A}_{\phi}^{\operatorname{ab},\K}/\operatorname{ann}_{\EuScript{A}_{\phi}^{\operatorname{ab},\K}}(M')$. By \cite[Cor. 8.2.5]{MVdB}, we have that  \[\operatorname{GK-dim}(A_{\operatorname{ab}})=2\operatorname{GK-dim}(M')=2d.\] Note that the proof of \cite[Lem. 4.11]{kamnitzerLieAlgebra2024} shows that $d$ is the GK dimension of $M'$ as an $A_{\operatorname{ab}}$-module, in addition to being its GK dimension as an $A$-module.
	
	Furthermore we have an injective map $A_{\operatorname{ab}}\subset A$, so we have an inequality
	\[\operatorname{GK-dim}(A)\geq \operatorname{GK-dim}(A_{\operatorname{ab}}).\]
	
Thus, we need only prove the opposite inequality.  For $w\in \widehat{W}$, let $\EuScript{A}_{\phi}^{\K}(\leq w)$ be the span as a left $S_{h}$-module of the elements $\mathbbm{r}_{w'}$ for $w'\leq w$; this is the same as the span of these elements as a right module, and thus an  $S_{h}\operatorname{-}S_{h}$-subbimodule.  
Using the geometric model for this algebra (following the notation of \cite[Def. 2.2]{kamnitzerLieAlgebra2024}), this is the homology of  $R_{G,N}^B(\leq w)$, the preimage of the Schubert variety $\overline{IwI}/I$ in $R_{G,N}^B$. 

 Consider $\EuScript{A}_{\phi}^{\K}(\leq w)/\EuScript{A}_{\phi}^{\K}(< w)$.  This is spanned by the image of $\mathbbm{r}_w$, as a free module of rank 1 over $S_{h}$ as a left module or as a right module, with the two actions differing by the action of $w$ by \cite[(3.6c) \& (3.9d)]{websterKoszulDuality2019}.  For $n\geq 0$, let  $\EuScript{A}_{\phi}^{\K}(\leq n)$ be the span of $\EuScript{A}_{\phi}^{\K}(\leq w)$ for all $w$ of length $\ell(w)\leq n$ 
 
 Taking the corresponding quotient $A(\leq w)/A(<w)$, we thus obtain a $S_h\operatorname{-}S_h$-bimodule whose support as a left and a right module must be in $\overline{\supp(M')}$.  Since these actions differ by $w$, the support as a left $S_h$-module must lie in $\overline{\supp(M')}\cap w\cdot \overline{\supp(M')}$.  The affine Weyl group elements where this intersection is $\geq k$ dimensional have translation parts that lie in a $2d-k$ dimensional variety, since all the components of $\overline{\supp(M')}$ are affine subspaces which as $\leq d$ dimensional.  Now, consider the span  $\EuScript{A}_0$  of
\begin{enumerate}
	\item the degree 1 elements $\ft^*\subset S_h $ and 
	\item generators of $ \EuScript{A}_{\phi}^{\K}(\leq n)$ as a left $S_h$-module for a fixed $n$.
\end{enumerate} 
We will choose $n$ sufficiently large that this will be a set of generators of $\EuScript{A}_{\phi}^{\K}$ as an algebra.  
In $\EuScript{A}_0^k$ we will only obtain elements of $ \EuScript{A}_{\phi}^{\K}(\leq nk)$.  Furthermore, if we let $d(w)$ be the degree of the unique generator of $\EuScript{A}_{\phi}^{\K}(\leq w)/\EuScript{A}_{\phi}^{\K}(< w)$, then this depends at worst linearly on $\ell(w)$: $|d(w)|\leq C\ell(w)$ for some constant $C$.  Similarly, the degree of any length $q$ monomial is $\leq  C'q$ for some $C'$.  Thus, the length $m$ monomials can only hit the span of the polynomials of degree $\leq C'm + C\ell(w)$ in $\EuScript{A}_{\phi}^{\K}(\leq w)/\EuScript{A}_{\phi}^{\K}(< w)$.  

Let $A_0$ be the image of $\EuScript{A}_0$ in $A$.  
If $\dim \overline{\supp(M')}\cap w\cdot \overline{\supp(M')} \leq k$, then we must have that the dimension of the span of the elements of degree $\leq p$ in $S_h$ times the cyclic generator in $A(\leq w)/A(<w)$ must be bounded by $D'' p^k$ for some constant $D''$; since this intersection is a union of affine spaces, whose number of components is bounded by the number of pairs of components, we can choose one $C''$ which works for all $w$.  To simplify our estimate below, we assume choose $D$ is chosen so if $\ell(w)\leq nq$, the span of the elements of degree $\leq C'm + C\ell(w)$ in $S_h$ times the cyclic generator in $A(\leq w)/A(<w)$ has dimension $\leq Dq^k$; taking $D=\max (D''(C'+Cn)^{2m}, D'')$ will suffice.  Combining our estimates, we find:
\begin{enumerate}
	\item The number of $w\in \widehat{W}$ of length $\leq nq$ such that $\dim \overline{\supp(M')}\cap w\cdot \overline{\supp(M')}=k$ is bounded above by $D'q^{2d-k}$.
	\item The dimension of $(A(\leq w)\cap  A_0^q)/(A(< w) \cap A_0^q)$ is bounded above by $Dq^k$ if $\ell(w)\leq nq$ .  
\end{enumerate}
Thus, summing over $k=1,\dots, 2d$,  we have  $\dim A_0^q\leq 2dDD'q^{2d}$.  Thus, we have
\[\log_m(\dim A_0^m)\leq \log_m(2dDD'm^{2d})=2d +\frac{\log(2DD')}{\log m} \]
so taking the limit, we have $\operatorname{GK-dim}(A)\leq 2d$, 
completing the proof.   
\end{proof}

\begin{lemma}\label{lem:GT-annihilator}
  There is a faithful GT module over any  quotient $\EuScript{A}_{\phi}^{\operatorname{sph}, \C}/I$.
  \end{lemma}
\begin{proof}
  Let $S'$ denote the image of $S$ in $A=\EuScript{A}_{\phi}^{\operatorname{sph}, \C}/I$.  Let $\mathfrak{p}\subset S'$ be the ideal of a component of $\Spec S'$ of maximal dimension
  $d=\dim V(\mathfrak{p})=\dim\Spec S'$, and let $\mathfrak{m}$
  be a generic maximal ideal containing $\mathfrak{p}$.  This ideal is
  defined by scalars $a_{i,k}\in \C$ as above, and in particular,
  since it is generic, we minimize the number of pairs of  $(a_{i,k},a_{j,m})$ or
  $(\hat{b}_{i,j},a_{j,m})$ with integral difference.  We can thus have at
  most $V-d$ independent equations of the form $a_{i,k}=a_{j,m}+p$ or
  $\hat{b}_{i,k}=a_{j,m}+p$ which are satisfied.   That is, we can find $d$ disjoint subsets
  $\Omega_1,\dots, \Omega_{d}$
  of  $\Omega$ such that if $(i,k)$ in $\Omega_q$, then $a_{i,k}$ does not
  have integral difference with any $\hat{b}_{j,m}$ or with $a_{j,m}$
  outside of $\Omega_q$.  Thus, for any $(x_1,\dots, x_{d})\in
  \Z^{d}$, we can add $x_q$ to all elements of $  \Omega_q$ and
  obtain an equivalent maximal ideal, since we do not flip the sign of any of the matter functions that take on an integral value.  Let $X$ be the Zariski closure
  of these points, which is a $d$-dimensional affine subspace.  

Now consider the module $A/A
  \mathfrak{m}^N$, which is a GT module.  Note that this has non-zero
  multiplicity at all points in $\Spec S$ that lie in the equivalence
  class of $\mathfrak{m}$; in particular, at the points obtained by
  translation as above.   This shows that $X\subset \Spec S'$, and for
  dimension reasons, it must be a component of this variety. Since $X$
  contains $V(\mathfrak{m})$, this is only possible if $X=V(\mathfrak{p})$.

Thus, the annihilator of $A/A
  \mathfrak{m}^N$ is a 2-sided ideal $I'$  of $A/I$; for $N$
  sufficiently large, the ideal $I'\cap
  S'$  has trivial $\mathfrak{p}$-primary component.  Thus, $I'\cap
  S'$ has strictly fewer associated primes of dimension $d$ than $S'$.

  Now, apply the same logic to $I'$ as a left $A$-module: let $S''$ be
  the quotient of $S_1$ by the elements annihilating $A/A
  \mathfrak{m}^N$.
  Note that $S''$ is a quotient of $S'$, and as noted above, it either
  has dimension $<d$ or fewer components of dimension $d$. Choose
  $\mathfrak{m}_2$ generic in a component of $\Spec S''$, etc.

  We can inductively define $I_k$ to be the annihilator of
  $\bigoplus_{q=1}^{k-1}A/A
  \mathfrak{m}_q^{N_q}$, the ring $S_{k}$ be $S_1$ modulo the
  annihilator of this module, and $\mathfrak{m}_k$ as a generic
  maximal ideal in a component of $\Spec S_k$.  Since at each step,
  the number of components or dimension of $\Spec S_k$ drops,
  eventually, this process will terminate at a module $M$ which is
  faithful, since its annihilator is killed by all elements of $S'$.  
 \end{proof}

We can now apply this result to understand
the perverse structure on wall-crossing
functors (also called translation functors) given by Losev \cite[Prop. 7.3]{losevModularCategories2021};  we are interested in
this filtration in the characteristic $p$ case, but since it is
defined by starting with ideals in characteristic $0$, we will need to
consider that case as well.

Every wall is associated to a function \begin{equation}
\label{eq:circuit-hyperplane2}
     f({\widehat{\mathbf b}})=\hat{b}_{j,\ell}-\hat{b}_{i,k}+\sum_{p=1}^n \varepsilon_i \hat{b}_{e_i}
  \end{equation} 
for an unoriented path \[i=i_0\overset{e_1}\longrightarrow i_1\overset{e_2}\longrightarrow i_2\overset{e_3}\longrightarrow\cdots \overset{e_n}\longrightarrow i_n=j,\] exactly as in \eqref{eq:circuit-hyperplane}.     
  The perverse
structure for the functor of crossing this wall is based on a chain of ideals $\mathcal{I}^{0}\subset \cdots
\subset \mathcal{I}^{V}$, defined as follows for parameters such that 
$f({\widehat{\mathbf b}})=r$ for $r$ a generic integer: consider 
$ {\EuScript{A}_{\phi}^{\operatorname{sph}, \C}}$, and
let $\mathcal{I}^k$ be the minimal ideal such
that $\EuScript{A}_{\phi}^{\operatorname{sph}, \C}/\mathcal{I}^k$ has Gelfand-Kirillov
dimension $2(V-k)$.  Lemma \ref{lem:GT-annihilator} shows that this is the same as the common annihilator of all Gelfand-Tsetlin modules with dimension of support $\leq V-k$.  

Having constructed the ideal $\mathcal{I}^k$, we then extend it to arbitrary integral $ {\widehat{\mathbf b}}$
by continuity.  Note this means that at these less generic values, there may be Gelfand-Tsetlin modules whose GK dimension is $\leq V-k$ that are not annihilated by $\mathcal{I}^k$; however, these don't deform to our more generic parameters.  

We've now defined $\mathcal{I}^k$ for arbitrary parameters over $\C$.  This defines an ideal in $\EuScript{A}_{\phi}^\Z$
by intersection, and then over any other base field by base change.

\begin{proof}[Proof of Thm. \ref{thm:real-variation}]
    \refstepcounter{dummy}\label{proof-thm:real-variation}
  By Lemma \ref{lem:Z-match}, it's enough to check these properties for the chamber $C_0$.  In this case $ {Z_{\hat\beta}}(M)$ is the integral of a positive function, and thus is positive; this shows \hyperref[RV1]{(1)}.

Now we turn to \hyperref[RV2a]{(2a)}.  The standard $t$-structure on $D^b( {\Rring_{C_0}}\fdmod)$  is compatible with  the
filtration by $\mathcal{D}_n$ by definition, so this is clear.

Finally, consider \hyperref[RV2b]{(2b)}; this requires us to show the equivalence
$\mathbb{B}_{C,C_0}$ for $C$ below $C_0$ is perverse with respect to
this filtration.  By Proposition \ref{prop:KLRW-bimodule}, it is
equivalent to the same question as whether the functor
${}_{\phi+\nu}\mathscr{T}_{\phi}\Lotimes-$ is perverse.  This is
proven in \cite[Prop. 7.3]{losevModularCategories2021}, and so we need only
prove that our filtration of the category agrees with Losev's.

We have already described Losev's filtration above.  It is based on
the ideals $\mathcal{I}^q$, which by Lemma \ref{lem:GT-annihilator}
is the annihilator of all GT modules with GK dimension $\leq V-q$ when we have chosen $\hat{b}$ so that $f({\widehat{\mathbf b}})=r$ for $r$ a generic integer, and that $\hat{b}$ is chosen to be generic with respect to this constraint.  Given a weight $\Ba$ in this support, let 
\[ F_{\Ba}=\{-f | f\text{ a matter function, } f(\Ba)\in \Z_{\geq 0}\}\cup \{f | f\text{ a matter function, } f(\Ba)\in \Z_{< 0}\}.\]
By \cite[Prop. 7.1.2]{MVdB}, the support of $M$ will have Zariski closure of dimension $\leq V-q$ if only if for every $\Ba$ in this support, the positive span of $F_{\Ba}$ contains a $q$-dimensional subspace $H$.  

This also means that the elements of $H$ vanish on the 
polytope corresponding to this clan at a generic point where $f=0$.  This in turn implies that this polytope lies in the $V-q$ dimensional subspace $H^{\perp}$ It follows that a GT module $M$ at this generic point is killed by $\mathcal{I}^q$ if and only if its support entirely lies in polytopes that become $\leq V-q$ dimensional at a generic point where $f=0$, i.e. that satisfy $d(\Bi, \Ba)\geq q$.  By continuity, this is true not just at ${\widehat{\mathbf b}}$ where $f({\widehat{\mathbf b}})=r\in \Z$, but at all points, since this is a Zariski dense locus.   By Lemma \ref{lem:vanishing-order}, this is the same as saying that its central charge vanishes to order $q$ as we pass to a generic point of the wall defined by $f$.

Thus, we have that the module $M$ is in $\mathcal{C}_q$ with respect to a given wall if and only $Z_{\hat \beta}(M)$ vanishes to order $q$ as $\hat \beta$ approaches a generic point of the wall.
\end{proof}

\begin{proof}[Proof of Thm. \ref{thm:braid-action}]
   \refstepcounter{dummy} \label{proof-thm:braid-action}
  This follows immediately from Proposition \ref{I-prop:pi-action}: the
  space $\mathring{T}_{1,F}$ is the torus minus the toric braid
  arrangement, so we can write the affine braid groupoid inside the
  fundamental groupoid of this space as usual, and the action of the
  functors $\Phi^{\phi',\phi''}_w$ match the bimodules
  $\mathbb{B}_{\tau}$ by Proposition \ref{prop:KLRW-bimodule}.
\end{proof}
\begin{proof}[Proof of Lem. \ref{lem:eat-cup}]
   \refstepcounter{dummy} \label{proof-lem:eat-cup}
Recall that in \cite[\S
  7.3]{Webmerged}, we defined a $\tilde{T}^{\Bj'}\operatorname{-}\tilde{T}^{\Bj}$
  bimodule $\mathfrak{k}^{\Bj'}_{\Bj}$ where the red strands trace out
  a cap with an element of a particular simple right module $L_0$ over $T^{\varpi_j,\varpi_{j^*}}$.
  \begin{equation*}
  \begin{tikzpicture}[very thick,xscale=1.4,yscale=-1.4]

\node (v) at (0,-1) [fill=white!80!gray,draw=white!80!gray, thick,rectangle,inner xsep=10pt,inner ysep=6pt, outer sep=-2pt] {$v$};

\begin{pgfonlayer}{background} \begin{scope}[very thick]
\draw[wei] (-4.5,-1) -- +(0,2) node[at start,above]{$j_1$} node[at end,below]{$j_1$};
    \draw (-3.75,-1) -- +(0,2) node[at start,above]{$i$} node[at end,below]{$i$};
    \node at (-3,0){$\cdots$};
    \draw[wei] (v.170) to[in=280,out=170] node[at end, below]{$j$} (-2.5,1);
\draw[wei] (v.10) to[out=10,in=260] node[at end, below]{$j^*$} (2.5,1) ;
\draw (v.55) to[in=270,out=55] (2.1,1);
    \draw (v.65) to[in=270,out=65] (1.7,1);
    \draw (v.75) to[in=270,out=75] (1.3,1) ;
    \draw (v.125) to[in=270,out=125](-2.1,1) ;
    \draw (v.115) to[in=270,out=115](-1.7,1);
    \draw (v.105) to[in=270,out=105] (-1.3,1);
    \draw[ultra thick,loosely dotted,-] (-.35,.5) -- (.35,.5);
  \end{scope}
\end{pgfonlayer}
  \end{tikzpicture}
\end{equation*}
 This bimodule is strongly equivariant, and we let
 $\mathring{\mathfrak{k}}$ be the corresponding cylindrical bimodule
 as in Lemma \ref{lem:affinize-commute}.  We let
 $\mathring{\mathfrak{k}}'$ be the reflection of this bimodule through
 a horizontal line.  

These bimodules will induce the Morita equivalence we desire. By the basis theorem \cite[Lemma 7.17]{Webmerged}, we see that this module is killed by $ {e_z}\in \Rring^{\Bj}$ for $z$ the number of strands inside the cup above. Thus, the
  left $\Rring^{\Bj}$-module structure on the cup bimodule
  factors through the quotient $R^{(z)}=\Rring^{\Bj}/\Rring^{\Bj}e_z\Rring^{\Bj}$. We now show that this bimodule induces a Morita equivalence between $\Rring^{\Bj}$ and $R^{(z)}$. 

  {\bf Construction of a Morita context}: The
  difficult step in this is to show that it forms a Morita
  context with the bimodule $\mathfrak{\mathring{k}}'$ given by reflecting
  these diagrams in the horizontal axis (and keeping all relations the same).  Since we've reflected, at the center of the cup, we must have elements of the left module $\dot{L}_0$, using the notation of \cite[\S 5.2]{Webmerged}.

  That is, we need to define maps
  \begin{equation}\label{eq:morita-context}
  	\alpha\colon \mathfrak{\mathring{k}}'\otimes_{\Rring^{\Bj} }\mathfrak{\mathring{k}}\to
    \Rring^{\Bj'}\qquad \omega\colon \mathfrak{\mathring{k}}\otimes_{\Rring^{\Bj'} }\mathfrak{\mathring{k}}'\to
    \Rring^{\Bj}
  \end{equation} that make the matrix space $
  \Big[\begin{smallmatrix}
    \Rring^{\Bj}&\mathfrak{\mathring{k}}\\ \mathfrak{\mathring{k}}'& \Rring^{\Bj'}
  \end{smallmatrix}\!\Big]
$
into an associative algebra.

By \cite[Prop. 5.11]{Webmerged}, we can choose an isomorphism of $\dot{L}_0$ to the dual of $L_0$, that is, a non-degenerate pairing $\alpha
\colon L_0\otimes
L_0\to \K$ satisfying
$\alpha(va,v')=\alpha(v,v'\dot{a})$.  We can extend this to define  $\alpha(d,d')$ as in \eqref{eq:morita-context} by 
stacking the diagrams $d\in \mathfrak{\mathring{k}}$ on top of $d'\in \mathfrak{\mathring{k}}'$, which creates a closed red circle, with strands springing from the top and bottom.  We can
simplify so that none of these strands cross the red circle using the relations \cite[(7.24)]{Webmerged}.  Thus, the strands from the top of the circle all attach to the bottom of the diagram, tracing out a KLR diagram in between.  By acting at the top or bottom of the circle, we can reduce to the case where
a circle has an element $v$ at the top, an element $v'$ at the bottom and only
straight strands between; in this case, we define $\alpha(d,d')$ to be the diagram with the circle deleted multiplied by the scalar $\alpha(v,v')$.

Now we turn to defining $\omega$.  First, consider the
planar KLRW algebra $T$ with $j,j*$ as labels, $z$ black strands of any
label, and consider the quotient $T^{(z)}=T/T {e_z}T$ by every idempotent with
black strands outside the two red.  The module $L_0$ is the unique right simple
module which is highest weight for the categorical
action by \cite[Prop. 7.5]{Webmerged}. This also shows that it is the unique simple that factors through the quotient $T^{(z)}$.   Furthermore,  \cite[Lemma 7.3]{Webmerged} shows that $L_0$ is
projective over this quotient:  The standard module $S$
appearing in that result is the quotient by all idempotents where a black
strand is left of all reds, and the kernel of
the map $S\to L_0$ is generated by the image of $e_z$.  Thus, $T^{(z)}$
is a matrix algebra equipped with an isomorphism
\[T/Te_zT\cong L_0\otimes_{\K}\dot{L}_0\cong \End_{\K}(L_0).  \] Thus, we
can define the map $\omega$ by using the relations \cite[\S
7.5-6]{Webmerged} to remove any strands from between the cup and cap,
and then replacing the pair of elements $v$ at the bottom of the cup
and $v'$ at the top of the cap with the linear map
$w\mapsto \alpha(v',w)v$, thought of as an element $T/Te_zT$.
\excise{

  By \cite[Lemma 7.16]{Webmerged}, we have that for some scalar
  $\epsilon$, we have an equality modulo  \begin{equation}
   \begin{tikzpicture}[very thick,yscale=-1,baseline]
\node at (-3.4,0) {$\epsilon$};
\begin{pgfonlayer}{background} \begin{scope}[very thick]
    \draw[wei] (-2.5,-1.5) to node[at end, below]{$\mu$} (-2.5,1.5);
\draw[wei] (2.5,-1.5) to node[at end, below]{$\mu^*$}(2.5,1.5) ;
    \draw (3,-1.5) to node[at start,above]{$j$} node[at end,below]{$j$} (3,1.5)  ;
    \draw (2.1,-1.5) to node[below,at end]{}(2.1,1.5);
    \draw (1.7,-1.5) to node[below,at end]{} (1.7,1.5);
    \draw (1.3,-1.5) to node[below,at end]{} (1.3,1.5);
    \draw (-2.1,-1.5) to node[below,at end]{} (-2.1,1.5);
    \draw (-1.7,-1.5) to node[below,at end]{} (-1.7,1.5);
    \draw (-1.3,-1.5) to node[below,at end]{}  (-1.3,1.5);
    \draw[ultra thick,loosely dotted,-] (-.35,0) -- (.35,0);
\end{scope}
\end{pgfonlayer}
\node at (3.5,0){=};
  \end{tikzpicture}
    \begin{tikzpicture}[very thick,yscale=-1,baseline]
\node (v) at (0,-1.3) [fill=white!80!gray,draw=white!80!gray, thick,rectangle,inner xsep=10pt,inner ysep=6pt, outer sep=-2pt] {$v$};
\begin{pgfonlayer}{background} \begin{scope}[very thick]
    \draw (3,-1.5) to[in=-20,out=150]   node[at start,above]{$j$}
    (-2.8,-.2) to[out=160,in=-90] (-3,0) to[out=90,in=-160] (-2.8,.2)  to [in=-150,out=20] node[at end,below]{$j$} (3,1.5)  ;
    \draw[wei] (-2.5,-1.5) to node[at end, below]{$\mu$} (-2.5,1.5);
\draw[wei] (2.5,-1.5) to node[at end, below]{$\mu^*$}(2.5,1.5) ;
    \draw (3,-1.5) to node[at start,above]{$j$} node[at end,below]{$j$} (3,1.5)  ;
    \draw (2.1,-1.5) to node[below,at end]{}(2.1,1.5);
    \draw (1.7,-1.5) to node[below,at end]{} (1.7,1.5);
    \draw (1.3,-1.5) to node[below,at end]{} (1.3,1.5);
    \draw (-2.1,-1.5) to node[below,at end]{} (-2.1,1.5);
    \draw (-1.7,-1.5) to node[below,at end]{} (-1.7,1.5);
    \draw (-1.3,-1.5) to node[below,at end]{}  (-1.3,1.5);
    \draw[ultra thick,loosely dotted,-] (-.35,0) -- (.35,0);
\end{scope}
\end{pgfonlayer}
\end{tikzpicture}
\end{equation}}
This shows that we have a Morita context. 

{\bf The context is an equivalence}:  A Morita context is a Morita equivalence if and only if both multiplication maps are surjective by \cite[II.3.4]{BassK}.  Since the image is a 2-sided ideal, it's enough to show that 1 is in the image in both cases. This is easy to see from the non-degeneracy of
the pairing $\alpha$.  
\end{proof}
\begin{proof}[Proof of Thm. \ref{thm:tangle-action}]
   \refstepcounter{dummy} \label{proof-thm:tangle-action}
   To prove that we have an action of affine ribbon tangles, we use the
  annular version of the formalism of \cite[Ch. 3]{ohtsukiQuantumInvariants2002}.
  
 {\bf Functors for sliced affine tangles}:  We define a
  sliced affine tangle diagram to be such a diagram where we have cut the tangle at a finite list of heights $h_1,\dots, h_k$, such that between $h_{q-1}$ and $h_q$ there is one crossing, or one minimum or one maximum of the tangle.  The assignment in the theorem gives a
  well-defined functor to the tangle between heights $h_{q-1}$ and $h_q$.  Thus, we can define a functor for each sliced affine tangle by composing these.
  
 {\bf Independence of slicing}: We need
  only show that any two ways of doing the slicing will result in the
  same functor.  By \cite[Th. 3.1]{ohtsukiQuantumInvariants2002}, this requires showing the ribbon
  Turaev moves: the ribbon Reidemeister moves, the $S$-move, the
  pitchfork move, and the commutation of distant tangles.  Each one
  of these is proven in the planar case in \cite[Th. 8.6]{Webmerged},
  and thus follows in the annular case by Lemma
  \ref{lem:affinize-commute}.  
\end{proof}
\begin{proof}[Proof of Thm. \ref{th:knot-invariant}]
   \refstepcounter{dummy} \label{proof-th:knot-invariant}
  Note that when we have no red or black strands, the algebra
  $\Rring^{\emptyset}_{\mathbf{0}}$ is just $\C$; in this case, the planar and
  cylindrical KLRW algebras coincide.    
  If a link arises from an inclusion of the 3-ball, then it can be
  presented as a sliced tangle in $\R^2\times [0,1]$, and then hit
  with the map compactifying one of the $\R$-directions to $S^1$.
Its value $\Phi(K)$ on a link is the cylindricalization of the functor
$\Phi_{\mathsf{L}}(K)$, since this is true for each individual slice,
and   Lemma \ref{lem:affinize-commute} implies the compatibility of
affinization with composition.  Since the planar and cylindrical KLRW
algebras are the same in the source and target categories, the functor
is tensor product with the same vector spaces in either case, and the
theories coincide.  The correspondence with other knot homologies
follows from \cite[Th. A]{mackaayCategorifiedSkew2018}.
\end{proof}

 \appendix

\section{Slodowy slices in type A}
\label{sec:slodowy-slices-type}

One particularly interesting special case of the constructions we have discussed is the {\bf S3 varieties for $\mathfrak{sl}_n$}.  These are resolutions of the intersections of Slodowy slices and nilpotent orbits in $\mathfrak{sl}_n$.  Each of these varieties can be written as a Nakajima quiver variety and as an affine Grassmannian slice (both in type $A$).  That is, they have a realization both as Higgs branches and as Coulomb branches of quiver gauge theories.  

Let us remind the reader of the combinatorics underlying this realization.  Given a partition $\lambda=(\la_1\geq \la_2\geq \cdots)$ of $N$ with $n$ parts, we can consider $\lambda$ as a (co)weight of $\mathfrak{sl}_n$, in the usual way.  Given $\mu$, another partition of $N$, 
we let \[ {w_i}=\lambda_i-\lambda_{i+1}\qquad  {v_i}=\sum_{k=1}^i \la_k-\mu_k.\]
The significance of these is more easily seen from the familiar formulae \[\la=\sum_{i=1}^n {w_i}\omega_i\qquad \mu=\la-\sum_{i=1}^n {v_i}\al_i.\]
This can be visualized as follows: if we start with a grid with $n$ columns and $e\geq \lambda_1$ rows where the number of dots in the $i$th column is given by $\la_i$.  The statistic $w_i$ is the number of rows with $i$ dots.  If $\mu\leq\la$ in the dominance order, then we can move the dots right (keeping them in the same row) so that there are $\mu_i$ in the $i$th column; in this case, we obtain a diagram of dots where we obtain the diagram for $\la$ (transposed from French notation) by pushing left, and that for $\mu$ by pushing down.   For example, for $\la=(3,3,1)$ and $\mu=(2,2,1,1,1)$, one example of such a diagram is:
\[   \begin{tikzpicture}[very thick,scale=.5]
      \draw[thick] (-2.5,1.5) -- (-2.5,-1.5);
      \draw[thick] (2.5,1.5) -- (2.5,-1.5);
      \draw[thick] (-2.5,1.5) -- (2.5,1.5);
      \draw[thick] (-2.5,-1.5) -- (2.5,-1.5);
      \draw[thick] (-1.5,1.5) -- (-1.5,-1.5);
      \draw[thick] (1.5,1.5) -- (1.5,-1.5);
      \draw[thick] (-0.5,1.5) -- (-0.5,-1.5);
      \draw[thick] (0.5,1.5) -- (0.5,-1.5);
      \draw[thick] (-2.5,0.5) -- (2.5,0.5);
      \draw[thick] (-2.5,-0.5) -- (2.5,-0.5);
      \node[circle, draw, inner sep=3pt] at (0,0){}; 
      \node[circle, draw, inner sep=3pt] at (0,1) {}; 
      \node[circle, draw, inner sep=3pt] at (1,0) {}; 
      \node[circle, draw, inner sep=3pt] at (1,1) {}; 
      \node[circle, draw, inner sep=3pt,fill=black] at (-1,-1) {}; 
      \node[circle, draw, inner sep=3pt,fill=black] at (-1,0) {}; 
      \node[circle, draw, inner sep=3pt,fill=black] at (-1,1) {};
      \node[circle, draw, inner sep=3pt] at (2,0) {}; 
      \node[circle, draw, inner sep=3pt] at (2,1) {}; 
      \node[circle, draw, inner sep=3pt,fill=black] at (-2,0) {}; 
      \node[circle, draw, inner sep=3pt,fill=black] at (-2,1) {}; 
      \node[circle, draw, inner sep=3pt] at (2,-1) {}; 
      \node[circle, draw, inner sep=3pt,fill=black] at (-2,-1) {}; 
      \node[circle, draw, inner sep=3pt,fill=black] at (0,-1) {};
      \node[circle, draw, inner sep=3pt] at (1,-1) {}; 
    \end{tikzpicture}
    \qquad 
    \begin{tikzpicture}[very thick,scale=.5]
      \draw[thick] (-2.5,1.5) -- (-2.5,-1.5);
      \draw[thick] (2.5,1.5) -- (2.5,-1.5);
      \draw[thick] (-2.5,1.5) -- (2.5,1.5);
      \draw[thick] (-2.5,-1.5) -- (2.5,-1.5);
      \draw[thick] (-1.5,1.5) -- (-1.5,-1.5);
      \draw[thick] (1.5,1.5) -- (1.5,-1.5);
      \draw[thick] (-0.5,1.5) -- (-0.5,-1.5);
      \draw[thick] (0.5,1.5) -- (0.5,-1.5);
      \draw[thick] (-2.5,0.5) -- (2.5,0.5);
      \draw[thick] (-2.5,-0.5) -- (2.5,-0.5);
      \node[circle, draw, inner sep=3pt,fill=black] at (0,0){}; 
      \node[circle, draw, inner sep=3pt] at (0,1) {}; 
      \node[circle, draw, inner sep=3pt] at (1,0) {}; 
      \node[circle, draw, inner sep=3pt] at (1,1) {}; 
      \node[circle, draw, inner sep=3pt] at (-1,-1) {}; 
      \node[circle, draw, inner sep=3pt,fill=black] at (-1,0) {}; 
      \node[circle, draw, inner sep=3pt,fill=black] at (-1,1) {};
      \node[circle, draw, inner sep=3pt] at (2,0) {}; 
      \node[circle, draw, inner sep=3pt] at (2,1) {}; 
      \node[circle, draw, inner sep=3pt] at (-2,0) {}; 
      \node[circle, draw, inner sep=3pt,fill=black] at (-2,1) {}; 
      \node[circle, draw, inner sep=3pt,fill=black] at (2,-1) {}; 
      \node[circle, draw, inner sep=3pt,fill=black] at (-2,-1) {}; 
      \node[circle, draw, inner sep=3pt] at (0,-1) {};
      \node[circle, draw, inner sep=3pt,fill=black] at (1,-1) {}; 
    \end{tikzpicture}\qquad\begin{tikzpicture}[very thick,scale=.5]
      \draw[thick] (-2.5,1.5) -- (-2.5,-1.5);
      \draw[thick] (2.5,1.5) -- (2.5,-1.5);
      \draw[thick] (-2.5,1.5) -- (2.5,1.5);
      \draw[thick] (-2.5,-1.5) -- (2.5,-1.5);
      \draw[thick] (-1.5,1.5) -- (-1.5,-1.5);
      \draw[thick] (1.5,1.5) -- (1.5,-1.5);
      \draw[thick] (-0.5,1.5) -- (-0.5,-1.5);
      \draw[thick] (0.5,1.5) -- (0.5,-1.5);
      \draw[thick] (-2.5,0.5) -- (2.5,0.5);
      \draw[thick] (-2.5,-0.5) -- (2.5,-0.5);
      \node[circle, draw, inner sep=3pt] at (0,0){}; 
      \node[circle, draw, inner sep=3pt] at (0,1) {}; 
      \node[circle, draw, inner sep=3pt] at (1,0) {}; 
      \node[circle, draw, inner sep=3pt] at (1,1) {}; 
      \node[circle, draw, inner sep=3pt,fill=black] at (-1,-1) {}; 
      \node[circle, draw, inner sep=3pt,fill=black] at (-1,0) {}; 
      \node[circle, draw, inner sep=3pt] at (-1,1) {};
      \node[circle, draw, inner sep=3pt] at (2,0) {}; 
      \node[circle, draw, inner sep=3pt] at (2,1) {}; 
      \node[circle, draw, inner sep=3pt,fill=black] at (-2,0) {}; 
      \node[circle, draw, inner sep=3pt] at (-2,1) {}; 
      \node[circle, draw, inner sep=3pt,fill=black] at (2,-1) {}; 
      \node[circle, draw, inner sep=3pt,fill=black] at (-2,-1) {}; 
      \node[circle, draw, inner sep=3pt,fill=black] at (0,-1) {};
      \node[circle, draw, inner sep=3pt,fill=black] at (1,-1) {}; 
    \end{tikzpicture}\]
The statistic $v_i$ is thus the number of dots pushed from the $i$th column to the $i+1$st for each $i$.     

Consider the S3 variety $\mathfrak{X}^\la_\mu$ given by the slice to nilpotent matrices of Jordan type $\mu$ in the closure of those of Jordan type $\la$.   Each such diagram of dots for $\lambda,\mu$ as above corresponds to a torus-fixed point in the resolved S3 variety. Note that all symplectic resolutions of the closure of the nilpotent orbit with generic Jordan type $\la$ are of the form $T^*G/P\to \bar{\mathscr{O}}$ for $P$ the parabolic of block upper-triangular matrices with block sizes given by the transpose partition to $\la$.  Different orders of block sizes can potentially give non-isomorphic resolutions; one can check $w_i$ as defined above is the number of blocks of size $i$ along the diagonal.
All resolutions of S3 varieties are given by the fiber product of $T^*G/P$ with a slice to the orbit with Jordan type $\mu$.
\begin{theorem}[\mbox{\cite[Th. 1.1.1]{MV22}, \cite[Th. 5.6]{BFNline}}]
  The S3 variety $\mathfrak{X}^\la_\mu$ is isomorphic to the affine Grassmannian slice to $\operatorname{Gr}^{\mu}$ inside $\operatorname{Gr}^{\bar \la}$, that is, to the Coulomb branch $
  \Coulomb$ of the quiver gauge theory with dimension vectors $\Bw$ and $\Bv$.   The resolution $\tM$ attached to a cocharacter $\xi \colon \C^*\to G_{\Bw}$ is isomorphic to the convolution resolution of $\operatorname{Gr}^{\bar \la}_\mu$ with order on fundamental coweights induced by the cocharacter $\xi$, and to the resolution $  T^*G/P\times_{\mathfrak{sl}_n^*}\mathfrak{X}^\la_\mu$ where $\xi$ determines the order on blocks in $P$.
\end{theorem}

Thus, Theorems \ref{ith:NCCR} and \ref{thm:D-equivalence} give us a noncommutative resolution of the S3 variety which is D-equivalent to any symplectic resolution of this variety. This is given by a cylindrical KLRW algebra with $w_i$ red strands and $v_i$ black strands of label $i$.  Below, we'll discuss in a bit more detail what we see in different cases of interest.  

\subsection{Kleinian singularities}

The simplest special case is the Kleinian singularity $\C^2/(\Z/\ell\Z)$.  This is isomorphic to the slice to the subregular orbit of $\mathfrak{sl}_\ell$ in the full nilcone, that is, Jordan types $\la=(\ell,0)$ and $\mu=(\ell-1,1)$.  Thus, this corresponds to the case where $w_1=\ell$ and $v_1=1$.  That is, we have $\ell$ red strands with the same label and a single black strand.

We have $\ell$ different idempotents depending on the position of the black strand, which we think of as positioned in a cycle.  The algebra of endomorphisms is generated by these idempotents, and by the degree 1 maps joining adjacent chambers by crossing the red strand:
\begin{equation*}
    \tikz{      \node at (2.5,0){ 
        \tikz[very thick,xscale=1]{
          \draw[fringe] (-1.7,-.5)-- (-1.7,.5);
          \draw[fringe] (1.7,.5)-- (1.7,-.5);
           \draw[wei] (-1,-.5)-- (-1,.5);
          \draw[wei] (.3,-.5)-- (.3,.5);
          \draw[wei] (1.5 ,-.5)-- (1.5,.5);
\draw (.9 ,-.5)  to[out=90,in=-90] (-.2,.5);
       }
      };
      \node at (9,0){ 
        \tikz[very thick,xscale=1, yscale=-1]{
           \draw[fringe] (-1.7,.5)-- (-1.7,-.5);
          \draw[fringe] (1.7,-.5)-- (1.7,.5);
         \draw[wei] (-1,-.5)-- (-1,.5);
          \draw[wei] (.3,-.5)-- (.3,.5);
          \draw[wei] (1.5 ,-.5)-- (1.5,.5);
\draw (.9 ,-.5)  to[out=90,in=-90] (-.2,.5);
       }
      };
      }
\end{equation*}
Thus, this algebra can be written as a quotient of the path algebra of the quiver with $\ell$ cyclically ordered nodes and edges joining adjacent pairs of edges in both directions.  The only relations needed are that the two paths of length two starting and ending at a given node are equal: they are both equal to multiplication by a single dot on the single black strand by (\ref{w-cost-1}).
Example \ref{example:NZ} covers the $\ell=2$ case.  In the $\ell=3$ case, we have the quiver shown below, with the diagrams above corresponding to a single pair of edges (with the others coming from rotations of these diagrams).
\[\tikz[very thick,scale=1.8]{
\node[circle,fill=black, inner sep=3pt,outer sep=2pt] (a) at (0,0){};
\node[circle,fill=black, inner sep=3pt,outer sep=2pt] (b) at (-.7,1){};
\node[circle,fill=black, inner sep=3pt,outer sep=2pt] (c) at (.7,1){};
\draw[->] (a) to[out=105, in=-45] (b);
\draw[<-] (a) to[out=135, in=-75] (b);
\draw[->] (a) to[out=45, in=-105] (c);
\draw[<-] (a) to[out=75, in=-135] (c);
\draw[->] (c) to[out=165, in=15] (b);
\draw[<-] (c) to[out=-165, in=-15] (b);
}\]

\subsection{2-row Slodowy slices}

Another case which has attracted considerable attention is that of 2-row Slodowy slices.  That is, for $k\leq \ell/2$, we consider the case $\la=(\ell,0)$ and $\mu=(\ell-k,k)$.  Thus, we have $w_1=\ell,v_1=k$.  The result is that all red and black strands are labeled by the same simple root.  This is thus a cylindrical version of the algebras $\tilde{T}^\ell_k$ defined in \cite[Def. 2.3]{WebTGK}.

Anno and Nandakumar \cite{ANexotic} show that there is an action of the category of affine tangles on the category of coherent sheaves in this case, which we expect to match ours.  In fact, it is virtually certain that this is the case. The same affine braid group action appears (up to reflection of braids), so we need only check that the cup and cap functors match.  From the fact that the composition of a cup and then a crossing is a cup shifted by $\pm 1$ (depending on the sign of the crossing) in both our tangle action by Theorem \ref{thm:real-variation} and in Anno-Nandakumar's action, as shown in \cite[Prop. 4.7]{ANexotic}, we see that the cup functors must have the same subcategory as image (the category of the eigenobjects for the crossing functor). However, this does not show that the functors are isomorphic.  We leave carefully matching these functors to another time or perhaps an industrious reader.

\subsection{Cotangent bundles to projective spaces and Grassmannians}

Dual to the examples of Kleinian singularities and 2-row Slodowy slices respectively are the cotangent bundles to projective spaces and Grassmannians.

The example of $T^*\mathbb{P}^n$  corresponds to thinking of this as the S3 variety for the minimal orbit in type A, that is, for the Jordan types $\la=(2,1,\dots, 1,0)$ and $\mu=(1,\dots, 1)$.  This corresponds to the quiver gauge theory attached to a linear quiver with $n-1$ nodes and vectors $\Bw=(1,0,\cdots,0,1)$ and $\Bv=(1,\dots, 1)$. 
One can easily check that the associated representation is that of 
$G=D\cap SL_n$, the diagonal matrices of determinant $1$ acting on $V=\C^n$, so indeed the associated Higgs branch is the Kleinian singularity $\C^2/(\Z/n\Z)$ by \cite[\S 4(vii)]{BFN}

To obtain the cotangent bundle to the Grassmannian $\operatorname{Gr}(n,k)$ of $k$ planes in $\C^n$ with $n\geq 2k$, we consider $\la=(2^{k},1^{n-2k},0^{k})$, and $\mu=(1,\dots, 1)$.  This gives $w_p=\delta_{k,p}+\delta_{n-k,p}$ and $\Bv=(1,2,3,\dots, k-1,k,\dots, k,k-1,\dots,2,1)$.  For example, $\operatorname{Gr}(4,2)$ corresponds to $\Bw=(0,2,0),\Bv=(1,2,1)$ and $\operatorname{Gr}(5,2)$ to $\Bw=(0,1,1,0),\Bv=(1,2,2,1)$.   This case is examined in more detail by the author and Suter in \cite{suterTiltingGenerator2024}.

\subsection{The noncommutative Springer resolution}
\label{sec:nonc-spring-resol}

One final variety of considerable interest that appears here is the cotangent bundle to the type A flag variety $T^*GL_n/B$.  This arises from the dimension vectors $\Bv=(1,2,3,\dots, n-1)$ and $\Bw=(0,\cdots, 0, n)$.  In this case, $\mu=\la^t=(1,\dots, 1)$ and $\la=\mu^t=(n)$.   It is a well-known theorem of Nakajima \cite[Th. 7.3]{nakajimaInstantonsALE1994} that the Higgs branch of this theory is $T^*G/B$.  
The equality $\mu=\la^t$ shows that this example is self-dual, and the Coulomb branch arises in the same way.

It's also well-known that the quantum Coulomb branch that arises this way is essentially the universal enveloping algebra of $\mathfrak{gl}_n$; if we fix the flavors to numerical values, then this is the quotient of this ring by a maximal ideal of its center, but keeping the flavors as variables, it is easy to construct $U(\mathfrak{gl}_n)$ on the nose (see, for example, \cite[Cor. 3.16]{weekesGeneratorsCoulomb2019}).

The construction of a tilting generator in Section \ref{I-sec:geometry} is thus just a rephrasing of the noncommutative Springer resolution as constructed by Bezrukavnikov, \Mirkovic, Rumynin, and Riche \cite{BMRR, BezNon}.
Recall that this construction operates by turning differential operators on the flag variety $X$ in characteristic $p$ into an Azumaya algebra $\mathcal{D}$ on $T^*X$, completing in a formal neighborhood of the zero section and finding a splitting of the resulting Azumaya algebra $\hat{\mathcal{D}}$ on this formal neighborhood. 

Thus, the constructions of Section \ref{I-sec:geometry} can be recast in this case in purely Lie theoretic terms.  The ring homomorphism $\sigma$ of Theorem \ref{I-thm:lonergan} is just the map $\sigma(X)= X^p-X^{(p)}$ for $X\in \mathfrak{gl}_n$.  The integrable system given by the equivariant parameters on the Coulomb branch is precisely the Gelfand-Tsetlin system as discussed in \cite[Th. 4.3]{WWY}. 
Identifying the Harish-Chandra center of $U(\mathfrak{gl}_n)$ with $\mathbb{F}_p[\mathfrak{t}]^W$, the homomorphism $\sigma$ is pullback by the Artin-Schreier map $\operatorname{AS}\colon \mathfrak{t}^*\to \mathfrak{t}^*$ (which is $W$-equivariant).  This describes the induced map on the full Gelfand-Tsetlin subalgebra $\Gamma$, which is the tensor product of the Harish-Chandra centers of $\mathfrak{gl}_k$ for $k=1,\dots, n$.

That is, we can think of a point in $\MaxSpec(\Gamma\otimes \overline{\mathbb{F}}_p)$ as a choice $\mathbf{a}_{1},\dots, \mathbf{a}_{n}$ with $\mathbf{a}_k$ an unordered $k$-tuple in $\overline{\mathbb{F}}_p$, and $\sigma$ the map that sends $a_{i,j}\mapsto a_{i,j}^p-a_{i,j}$.  In particular, when we complete in a formal neighborhood of the zero section, the ideal in $\sigma(\Gamma)$ generated by $a_{i,j}^p-a_{i,j}$ acts topologically nilpotently.  

The sections of $\mathcal{D}$ are identified with a completion of the cylindrical KLRW algebra by Lemma \ref{I-lem:A-H} and Proposition \ref{prop:KLR-B}.  By Sun-Tzu's Remainder Theorem applied to the image of $\Gamma$, this completion also contains idempotents attached to each maximal ideal lying over that in $\sigma(\Gamma)$, that is, those with $a_{i,j}\in \mathbb{F}_p$.  
These correspond to the idempotents in the KLR algebra, given by $e'(\Ba)$ where the labels $i$ have longitude $a_{i,j}/p\in \R/\Z$ (following Definition \ref{def:i-a}); as before we let $e(a)$ denote when we have fixed $a_{i,j}=a$ for all $i,j$, and multiplied by a primitive idempotent in the nilHecke algebra.  The image $\mathcal{D}e(a)$ is a splitting bundle for this Azumaya algebra by Lemma \ref{I-lem:0-split}.

Thus, while we obtain a familiar object, we obtain a new perspective on it, since this KLR presentation is not at all obvious from the Lie theoretic perspective.  Developing its consequences will have to wait for future work.

\bigskip
\IndexOfNotation

{\renewcommand{\markboth}[2]{}\printbibliography}
\end{document}